\documentclass[twoside,11pt]{article}

\usepackage[preprint]{jmlr2e}

\usepackage{blindtext}
\usepackage{mathtools}
\usepackage{bbm}
\usepackage{enumitem}
\usepackage{bm}
\usepackage{lastpage}
\usepackage{booktabs}

\usepackage{dashrule}
\usepackage{array}
\usepackage{arydshln}

\newcommand{\R}{\mathbb R}

\newcommand{\E}{\mathbb E}
\newcommand{\Pp}{\mathbb P}

\newcommand{\Law}{\mathcal L}

\newcommand{\KL}{D_{\mathrm{KL}}}

\newcommand{\Unif}{\operatorname{Unif}}
\newcommand{\Exp}{\operatorname{Exp}}

\newcommand{\argmax}{\operatorname*{arg\,max}}

\newcommand{\Err}{\operatorname{err}}

\newcommand{\rsc}[1]{\breve{#1}}
\newcommand{\rvec}[1]{\breve{\mathbf{#1}}}
\newcommand{\rmat}[1]{\breve{\mathbf{#1}}}

\begin{document}

\title{Near-Optimal Lower Bounds for Randomized Algorithms\\ in
Exact Value Zeroth-Order Convex Optimization }

\author{\name Haihan Zhang \email zhanghaihan@stu.pku.edu.cn \\
       \addr School of Intelligence Science and Technology, Peking University
       \AND
       \name Chenheng Zhang \email chenhengz@stu.pku.edu.cn \\
       \addr School of Intelligence Science and Technology, Peking University
       \AND
       \name Zhiquan Qi \email zqqi@stu.pku.edu.cn \\
       \addr School of Electronics Engineering and Computer Science, Peking University
       \AND
       \name Zhouchen Lin \email ZLIN@pku.edu.cn \\
       \addr School of Intelligence Science and Technology, Peking University}

\maketitle

\begin{abstract}
Whether exact scalar feedback intrinsically incurs the additional dimension
\(d\) paid by known zeroth-order methods remains open even for Lipschitz
convex optimization.  For a universal Lipschitz scale, the value only bound
\(O(d^2\log(d+1)\log(1/\epsilon))\)
of \citet{protasov1996algorithms} and the two-point bound
\(O(d\epsilon^{-2})\)
of \citet{JMLR:v18:16-632} yield the upper bound 
$\widetilde O\left(d\min\{d,\epsilon^{-2}\}\right)$.
By contrast, prior lower bounds for arbitrary \textit{randomized} algorithms
under exact value access yield only the first-order scale
\(\Omega(\min\{d,\epsilon^{-2}\})\)~\citep{7919238}, leaving an
unexplained factor \(d\).
We close this gap, up to logarithmic factors, for arbitrary adaptive
\textit{randomized} algorithms minimizing a  convex 
objective with a universal Lipschitz scale over the \(d\)-dimensional Euclidean unit ball, where each
query returns only the exact scalar value. 
Let \(T_\epsilon\) denote the minimum
number of queries required to return an \(\epsilon\)-suboptimal point
with probability at least \(1/2\), uniformly over the function class.
We prove that
\[
T_\epsilon
\ge
c\,
\frac{
d\min\{d,\epsilon^{-2}\}
}{
\log\!\bigl(\min\{d,\epsilon^{-2}\}\bigr)
},
\]
for \(d\ge d_0\) and \(0<\epsilon\le\epsilon_0\), where
\(c,\epsilon_0>0\) and \(d_0\in\mathbb N\) are universal constants.  
This gives $\Omega\left(
\frac{d}{
\epsilon^2\log(1/\epsilon)
} 
\right)$ in the low-accuracy regime
\(\epsilon\ge d^{-1/2}\) 
and $\Omega\left(
\frac{d^2}{\log d}
\right)$ in the high-accuracy regime
\(\epsilon\le d^{-1/2}\)  
with the latter independent of \(\epsilon\).
These bounds match the corresponding upper bound  up
to logarithmic factors.  To our knowledge, this is the first
near-optimal lower bound for arbitrary adaptive \textit{randomized} algorithms
throughout both accuracy regimes of exact value Lipschitz convex
optimization.  
The proof uses a random
support function hard family and develops a posterior mean energy method
for adaptive exact max observations, in place of first-order zero chain
constructions and noise based transcript inequalities.
\end{abstract}


\section{Introduction}

Zeroth-order methods replace derivative access by scalar function values,
from which they construct gradient estimates or descent directions through 
finite differences and randomized smoothing~\citep{matyas1965random,conn2009introduction,Nesterov2015RandomGM,JMLR:v18:16-632}. 
This paradigm is fundamental to derivative-free and black-box optimization~\citep{NIPS2015_ab817c93,chen2019zo} 
and has become increasingly
relevant in large scale learning, where backpropagation may be unavailable
or memory intensive~\citep{malladi2023fine,zhang2024revisiting,wang2025scalable}.

Across many problem classes, 
known zeroth-order algorithms incur an additional dimension dependent cost relative to comparable first-order
methods~\citep{ghadimi2013stochastic,duchi2015optimal,Nesterov2015RandomGM, JMLR:v18:16-632,kornowski2024algorithm}. At a heuristic level, this is natural: a first-order oracle returns a vector in
\(\mathbb R^d\), whereas a zeroth-order oracle returns a single real
number.  Yet lower bounds explaining whether this cost is 
intrinsic under exact, noiseless scalar feedback remain limited, and even
for canonical convex problems the optimal joint dependence on dimension
and target accuracy is unresolved. An
exact scalar response may mix information about many coordinates,
and adaptive algorithms may combine several values into derivative
estimates. A basic question is therefore: 
\begin{center}
\textbf{\textit{What is the intrinsic price of replacing gradient vectors
by function values?}}
\end{center}
We study this question in the exact scalar value model.
The objective \(f:B_2^d\to\mathbb R\) is convex and Lipschitz on the
Euclidean unit ball $B_2^d
:=
\{\mathbf x\in\mathbb R^d:\|\mathbf x\|_2\le1\}$,  
and a query \(\mathbf x\in B_2^d\) returns only the exact real number
\(f(\mathbf x)\). The algorithm may be \textit{randomized} and adaptive, but it
receives no gradient, subgradient, or observation noise. Throughout the
paper, oracle complexity counts individual scalar evaluations.

For the canonical Lipschitz convex optimization over a Euclidean ball, first-order
subgradient methods have the familiar
\(O(\epsilon^{-2})\)
complexity~\citep{nesterov2018lectures}.
Under exact scalar value access, the known upper bound landscape is
governed by the lower envelope of two complementary guarantees.
After translating its relative objective error guarantee to fixed
Lipschitz and radius scales, the value only method of
\citet{protasov1996algorithms} gives $O\left(
d^2\log(d+1)\log\frac1\epsilon
\right)$ 
exact evaluations.  The later two-point method of
\citet{JMLR:v18:16-632}, specialized to a fixed objective by setting
\(f_t=f\) in every round, gives $O(d\epsilon^{-2})$ 
exact evaluations.  Consequently, for constant Lipschitz and radius scales, the two results
together imply that \(\epsilon\)-suboptimality can be attained using $O\left(
\min\left\{
d\epsilon^{-2},
\,
d^2\log(d+1)\log\frac1\epsilon
\right\}
\right)$
exact function value evaluations.  Up to logarithmic factors, this
upper bound landscape has polynomial scale $d\min\{d,\epsilon^{-2}\}$.

By contrast, a full factor-\(d\) gap remains between this upper bound
landscape and the best lower bounds applicable to arbitrary \textit{randomized}
exact value algorithms.
\citet{7919238,10411928} establish distributional lower bounds for arbitrary
local oracles, and hence for \textit{randomized} algorithms.
Their previous lower bounds recover only
the first-order polynomial scale $\Omega\left(
\min\{d,\epsilon^{-2}\}
\right)$,
rather than the previous exact value upper bound scale $d\min\{d,\epsilon^{-2}\}$. 
The entire additional factor \(d\), representing the 
information cost of scalar feedback, was therefore unaccounted.
In particular, at the transition
\(\epsilon\asymp d^{-1/2}\), the \textit{randomized} lower bound scale
is only
\(\Omega(d)\),
whereas the upper bound scale is
\(O(d^2)\).
Independent concurrent work by
\citet{kerger2026closing}
establishes a \textit{deterministic} high-accuracy endpoint: for a sufficiently
small universal constant \(\beta>0\), every \textit{deterministic} exact value
algorithm requires $\Omega\left(
\frac{d^2}{\log(d+1)}
\right)$ 
queries at accuracy
\(\beta d^{-1/2}\).
However, its resisting oracle is tied to a \textit{deterministic} transcript and,
as noted in that work, does not yield a common hard distribution for
\textit{randomized} algorithms.  It also does not recover the low-accuracy
regime
\(d\epsilon^{-2}\).
Thus, before the present work, no lower bound for arbitrary adaptive
\textit{randomized} exact value algorithms captured the scalar feedback factor
\(d\) throughout both accuracy regimes.

A complementary theory gives sharp minimax lower bounds for stochastic
zeroth-order optimization, where a query returns random values such as
\(F(\mathbf x;\xi)\) with
\(f(\mathbf x)=\mathbb E_\xi F(\mathbf x;\xi)\)
~\citep{NIPS2012_e555ebe0,duchi2015optimal}. These lower bounds reduce optimization to
statistical testing between noisy observation laws and recover the
dimension dependence of two-point methods. They do not directly settle
the exact value model: conditioned on the objective and the algorithmic seed, the transcript is deterministic, so the
noise induced divergence controls available in the stochastic model are
absent. Thus prior work provides both an exact local oracle baseline and
a sharp stochastic theory, but leaves open whether the factor \(d\) in
exact value upper bounds is an intrinsic cost of scalar feedback.

Our main theorem closes this exact feedback gap up to logarithmic
factors. Let
\(\mathcal F_{L_0}\) be the class of convex \(L_0\)-Lipschitz
functions on \(B_2^d\), where \(L_0>0\) is a universal
constant. Let $T_\epsilon
:=
T_\epsilon(\mathcal A_{\rm rand},\mathcal F_{L_0})$
denote the minimum number of exact scalar
evaluations required by an adaptive \textit{randomized} algorithm to return an
\(\epsilon\)-suboptimal point with probability at least \(1/2\), uniformly
over \(\mathcal F_{L_0}\).  We prove
that there exist universal constants \(c,\epsilon_0>0\) and
\(d_0\in\mathbb N\) such that
\[
        T_\epsilon
        \ge
        c\,
        \frac{
        d\min\{d,\epsilon^{-2}\}
        }{
        \log\!\bigl(\min\{d,\epsilon^{-2}\}\bigr)
        },
\]
for every \(d\ge d_0\) and \(0<\epsilon\le\epsilon_0\).  
In the low-accuracy regime \(\epsilon\ge d^{-1/2}\), the result gives \(T_\epsilon
        =\Omega(d/(\epsilon^2\log(1/\epsilon)))\),
matching the \(O(d\epsilon^{-2})\) two-point exact value upper bound up to a logarithmic factor~\citep{JMLR:v18:16-632}. In the high accuracy regime
\(\epsilon\le d^{-1/2}\), the lower bound
saturates at \(T_\epsilon
        = \Omega(d^2/\log d)\), with no remaining dependence on \(\epsilon\), and 
matches the  \(\widetilde O(d^2)\) evaluation oracle upper bound up to polylogarithmic
factors~\citep{protasov1996algorithms,pmlr-v75-lee18a}. 
Thus the result determines, up to logarithmic factors, the joint
polynomial dependence on dimension and accuracy on both sides of the
transition \(\epsilon\asymp d^{-1/2}\). To our knowledge, this is the
first near-optimal lower bound that isolates the scalar feedback
dimension penalty for arbitrary adaptive \textit{randomized} exact value
algorithms.

The proof requires a mechanism that is neither span based nor
noise based. Exact scalar feedback is low dimensional in format but not
necessarily low information: a value at a dense query may depend on all
hidden directions, and several values may be combined into
finite difference estimates. Consequently, first-order zero chain and
resisting oracle arguments~\citep{nesterov2018lectures,woodworth2017lower,carmon2020lower,10.1007/s10107-019-01431-x}, which constrain the support or span of
returned derivative vectors, do not control the information revealed by
exact values.
A different geometric resisting oracle construction can establish the
\textit{deterministic} high-accuracy endpoint~\citep{kerger2026closing}. However, that construction is specific to a deterministic transcript
and does not provide the fixed hard distribution required for a
\textit{randomized} lower bound through Yao's principle. 
At the same time, the transcript contains no exogenous
observation noise from which to derive a per query statistical
indistinguishability bound~\citep{NIPS2012_e555ebe0,duchi2015optimal}. Our approach instead tracks directly how much posterior information an adaptive sequence of exact values reveals.

We use a random support function hard family
$f_\Xi(\mathbf x)
        =
        \max_{1\le i\le k}
        \langle \rvec a_i,\mathbf x\rangle,
        \ 
        \mathbf x\in B_2^d$, 
where the hidden vectors \(\rvec a_1,\ldots,\rvec a_k\) are
independent normalized truncated Gaussians. 
Let $\rvec s^{\mathrm{sc}}
:=
\sum_{i=1}^k\rvec a_i$ 
be their aggregate direction. The construction has two complementary
geometric properties. First, with constant probability, $ \min_{\mathbf x\in B_2^d} f_\Xi(\mathbf x)
        \lesssim 
        -k^{-1/2}$.
Second, for every \(\mathbf x\in B_2^d\), $f_\Xi(\mathbf x)
\ge
\frac1k
\langle\rvec s^{\mathrm{sc}},\mathbf x\rangle$. 
Hence an approximately optimal output must be substantially aligned with
\(-\rvec  s^{\mathrm{sc}}\). The parameter \(k\) therefore sets the
accuracy scale, and choosing $k\asymp\min\{d,\epsilon^{-2}\}$ 
produces the two regimes of the main theorem.

The information analysis tracks the posterior mean of
\(\rvec  s^{\mathrm{sc}}\). For this purpose, we condition on an
augmented transcript that records, in addition to each observed maximum
value, the identity of a deterministic active linear piece. The actual
algorithm still uses only function values; the augmentation is an
analytical device that exposes the posterior structure. Given the
augmented transcript, a winner block is restricted to an affine slice
and each loser block is restricted by halfspace inequalities. As a
result, the conditional law continues to factor across blocks, and each
block remains strongly log-concave on an affine support.

The main technical result is a one step selection inequality. For
independent strongly log-concave blocks on affine supports, one augmented
maximum observation moves the posterior mean of their unscaled sum by at
most \(O(\log k)\) in expected squared norm. Its proof decomposes the
selection event into a winner slice and loser lower tails. Winner-slice
and loser-tail regression estimates control the corresponding conditional
barycenters, while a two dimensional no spike principle rules out large
typical movements along thin slices.

Applying the one step inequality conditionally at each adaptive query and
using martingale orthogonality yields $\mathbb E
\left\|
\mathbb E[
\rvec s^{\mathrm{sc}}
\mid\mathcal G_T
]
\right\|^2
\lesssim
\frac{T\log(k)}{d}$,
where \(\mathcal G_T=\sigma(\mathsf T_T)\) is the sigma field generated
by the augmented transcript. A conditional
subgaussian bound controls the posterior residual in the algorithm's
output direction. Together with the negative optimum event, these
estimates show that $T
\lesssim
\frac{dk}{\log(k)}$ 
queries cannot produce a sufficiently aligned output with constant
probability. Finally, setting
\(k\asymp\min\{d,\epsilon^{-2}\}\) and applying Yao's minimax principle
gives the lower bound for arbitrary \textit{randomized} algorithms.

\subsection{Our Contributions}
Our contributions are twofold. 
\begin{itemize}
\item  First, we establish a near-optimal lower bound for arbitrary
adaptive \textit{randomized} algorithms minimizing  Lipschitz convex functions with exact scalar value access.  To our knowledge, this
is the first full scale result for \textit{randomized} algorithms in this model:
it captures the scalar feedback dimension penalty and determines, up to
logarithmic factors, the joint polynomial dependence on dimension and
accuracy from
\(d\epsilon^{-2}\) to \(d^2\).

    \item Second, we develop a posterior mean energy framework for
    adaptive noiseless oracle lower bounds. Its core is a
    one-step selection bound that controls exact scalar valued observations through posterior geometry, rather than
    through first-order zero chain growth or noise induced transcript
    contraction.
\end{itemize}

\section{Related Work}

\paragraph{Zeroth-order optimization.} 
Randomized smoothing, finite differences, and random directional
estimators form a classical approach to zeroth-order optimization~\citep{matyas1965random,conn2009introduction,Nesterov2015RandomGM,JMLR:v18:16-632}.
Across the four standard convex regimes, $L$-Lipschitz convex, $\beta$-smooth
convex, $L$-Lipschitz $\mu$-strongly convex, and $\beta$-smooth $\mu$-strongly convex, a recurring phenomenon is that zeroth-order methods retain the accuracy
and condition parameter dependence of comparable first-order methods
while paying an additional dimension dependent cost.  At a schematic
level, the corresponding first-order dependencies are $O\!\left(\frac{L^2}{\epsilon^2}\right)$, $O\!\left(\frac{\beta }{\epsilon}\right)$, $O\!\left(\frac{L^2}{\mu\epsilon}\right)$, and $O\!\left(\frac{\beta}{\mu}\log\frac1\epsilon\right)$ 
respectively~\citep{nesterov2018lectures}.  Under standard randomized zeroth-order estimators, these
rates commonly acquire an additional dimension factor.  
\citet{Nesterov2015RandomGM} develop a broad Gaussian random search
framework covering smooth and nonsmooth convex objectives, strong
convexity, stochastic optimization, and nonconvex stationarity.  Their
framework contains both finite difference function value estimators and
a directional derivative estimator.  In particular, they obtain the sharper
\(O(dL^2/\epsilon^2)\) nonsmooth rate under directional-derivative feedback, whereas their finite-difference analysis incurs an additional dimension factor.  The
\(O(dL^2/\epsilon^2)\) exact scalar value upper bound relevant to our
comparison follows instead from the symmetric two-point method of \citet{JMLR:v18:16-632}. 
The two-point and multi-point literature gives a complementary stochastic
and minimax perspective.  \citet{duchi2015optimal} study stochastic and
nonstochastic convex objectives and establish dimension dependent rates
for smooth and nonsmooth problems using paired or multiple function
evaluations.  Zeroth-order methods for smooth nonconvex stochastic
optimization were developed by \citet{ghadimi2013stochastic}, while
more recent work studies nonsmooth nonconvex stochastic objectives and
the dimension dependence of stationarity guarantees~\citep{lin2022gradient,chen2023faster,kornowski2024algorithm}. 
Exact evaluation oracles also arise in general convex optimization.
For \(L\)-Lipschitz functions on \(B_2^d\),
\citet{protasov1996algorithms} give a deterministic value only algorithm requiring $O\left(
d^2\log(d+1)
\log\frac{4L}{\epsilon}
\right)$ 
evaluations to achieve error \(\epsilon\). 
\citet{pmlr-v75-lee18a} later give a more general randomized algorithm
for minimizing a convex function given an evaluation oracle for the
objective and a membership oracle for the feasible set.  Their method
uses
\(\widetilde O(d^2)\)
oracle calls and
\(\widetilde O(d^3)\)
additional arithmetic operations.  On the explicit Euclidean ball, it
also gives a high-accuracy upper bound in our model.  We use Protasov's
result as the primary comparison because it is 
value only, and stated directly for a known convex domain. 
Together with the two-point upper bound, this yields the
two scale upper bound landscape $O(d\epsilon^{-2})
\wedge
\widetilde O(d^2)$ 
for the Lipschitz convex class studied here. Randomized coordinate and block coordinate methods provide a related
partial information viewpoint
\citep{nesterov2012efficiency,bubeck2015convex}.  These methods replace a
full dimensional update by a randomly selected coordinate or block,
trading less information and cheaper computation per iteration for a
dimension dependent iteration complexity.  This is not the same oracle
model as exact scalar value access: coordinate methods receive selected
components of derivative information, whereas one exact function value
may mix information from many directions.  Nevertheless, they provide a
useful comparison for the broader role of dimension under partial
information.

\paragraph{Lower bounds.}
Oracle lower bound theory is substantially more developed for first- and
higher-order information.  Resisting oracle, zero chain, span based, and
hard instance constructions establish lower bounds for convex and
strongly convex optimization~\citep{nesterov2018lectures,woodworth2017lower,arjevani2019oracle},
finite-sum and composite objectives~\citep{woodworth2016tight,bai2024complexity}, smooth functions satisfying
the Polyak--\L{}ojasiewicz condition
\citep{yue2023lower}, and nonconvex stationary point problems for first-
and higher-order algorithms~\citep{arjevani2020second,arjevani2023lower,fang2018spider,carmon2020lower,10.1007/s10107-019-01431-x}. 
The standard mechanisms, however, do not directly control exact
scalar value feedback.  A zero chain argument restricts the support or
span of derivative vectors revealed to the algorithm.  By contrast, a
value at a dense query may depend on all hidden directions, and several
exact values may be combined into finite difference estimates.  Information theoretic lower bounds are available for stochastic
zeroth-order optimization.  In these models, an oracle returns random
values such as \(F(\mathbf x;\xi)\), and lower bounds compare the induced
observation laws through statistical testing or divergence arguments
\citep{NIPS2012_e555ebe0,duchi2015optimal}.  These results sharply capture
dimension dependence in stochastic convex models, but do not directly
settle the exact value setting, where the oracle
introduces no exogenous observation noise. Noiseless \textit{randomized} baselines are also known.
\citet{7919238,10411928} establish lower bounds for arbitrary local
oracles in nonsmooth convex optimization.  Since exact scalar evaluation
is a local oracle, their Euclidean large scale result implies $\Omega(\epsilon^{-2})$ for  $\epsilon\ge d^{-1/2}$.  This result applies to a broader and potentially more
informative oracle class, but consequently does not isolate the additional
factor \(d\) associated specifically with scalar only feedback. 
Related lower bound questions also arise outside convex optimization. In the noisy stochastic nonsmooth nonconvex setting,
\citet{kornowski2024algorithm} obtain an
\(O(d\delta^{-1}\epsilon^{-3})\) zeroth-order algorithm for finding a
\((\delta,\epsilon)\)-Goldstein-stationary point, where \(\delta\) is the
localization radius in the Goldstein subdifferential and \(\epsilon\) is
the stationarity tolerance. They observe that, although the dependence
on each parameter is separately optimal, no lower bound jointly matching
\(d\), \(\delta\), and \(\epsilon\) was known.
They conjecture that such a result could be obtained by adapting smooth
first-order lower bound analyses to zeroth-order oracles.  Although these
nonconvex stochastic models and stationarity criteria differ from ours,
they illustrate the broader difficulty of establishing zeroth-order lower
bounds that are simultaneously sharp in dimension and accuracy. 
Our work addresses this difficulty in the exact scalar value convex setting.  The posterior mean energy method controls
the information revealed by adaptive noiseless real valued observations
through the geometry of posterior distributions, rather than through
first-order span growth or noise induced transcript contraction.

\paragraph{Independent concurrent \textit{deterministic} exact value lower bound.}
Independent concurrent work by
\citet{kerger2026closing}
proves that arbitrary \textit{deterministic} exact value algorithms require $\Omega\left(
\frac{d^2}{\log(d+1)}
\right)$ 
queries at accuracy
\(\beta d^{-1/2}\)
for a sufficiently small universal constant \(\beta>0\).  
This
agrees with the \textit{deterministic} high-accuracy endpoint implied by
Theorem~\ref{thm:main-exact-zo-lb}. 
The scopes of the two results differ in two essential respects.
First, the result of
\citet{kerger2026closing}
is \textit{deterministic} and is tied to the high-accuracy endpoint, whereas our
theorem applies to arbitrary adaptive \textit{randomized} algorithms and gives
the full interpolation from
\(\widetilde\Omega(d\epsilon^{-2})\)
to
\(\widetilde\Omega(d^2)\). 
Second, the proof mechanisms are different.
Their argument maintains Cartesian products of convex uncertainty sets
under a transcript specific resisting oracle and controls codimension,
intrinsic volume, and aggregate width in order to construct two
transcript compatible objectives with separated minimizers.
Our argument starts from a fixed product distribution over objectives,
constructs measurable product posterior kernels, controls the
posterior mean energy generated by the actual exact observations, and
then invokes Yao's principle.
As explicitly noted by
\citet{kerger2026closing},
their resisting oracle does not yield a hard distribution for \textit{randomized}
algorithms. 
\citet{kerger2026closing} also derives a mixed-integer extension, which
is outside the scope of the present paper.

\paragraph{Basic notation.}
Let $B_2^d
:=
\left\{
\mathbf x\in\mathbb R^d:
\|\mathbf x\|_2\le1
\right\}$ and $\mathbb S^{d-1}
:=
\left\{
\mathbf x\in\mathbb R^d:
\|\mathbf x\|_2=1
\right\}$ 
denote the Euclidean unit ball and sphere.  For \(k\in\mathbb N\), write $[k]:=\{1,\ldots,k\}$. 
Unless explicitly declared otherwise, symbols without a breve denote
deterministic numerical quantities.  Deterministic scalars and parameters
are written with ordinary Roman or Greek letters, in lowercase or
uppercase according to their conventional roles, such as $x$, $\epsilon$ and $T$. Deterministic vectors are written in
bold lowercase Latin letters, such as $\mathbf{x}$ and $\mathbf{y}$, and
deterministic matrices in bold uppercase  Latin letters, such as $\mathbf{A}$ and $\mathbf{B}$.  Deterministic linear maps are
written in calligraphic letters, such as $\mathcal A$ and $\mathcal B$. 
Unless explicitly declared otherwise, named scalar-, vector-, and
matrix-valued quantities without a breve are deterministic.  A breve
marks a named random quantity:
\(\breve X\) denotes a scalar random variable,
\(\breve{\mathbf x}\) a random vector, and
\(\breve{\mathbf A}\) a random matrix.
Random seeds, transcript variables, and other abstract random elements
are declared explicitly when introduced.
Functional and operator notation, such as
\(\mathbb E\), \(\mathbb P\), and
\(\operatorname{err}_T(\mathsf r,f)\), is not separately decorated with a
breve; its randomness is determined by its arguments. 
The Euclidean norm is denoted by \(\|\cdot\|\), and all Lipschitz
constants are with respect to this norm.  The symbols
\(\mathbf I_d\) and \(\mathbf 1_d\) denote, respectively, the
\(d\times d\) identity matrix and the all ones vector in
\(\mathbb R^d\). 
The symbols
\(\mathbf 0_d\) and \(\mathbf O_{m\times n}\)
denote, respectively, the zero vector in \(\mathbb R^d\) and the
\(m\times n\) zero matrix.
For a square zero matrix, we write
\(\mathbf O_d:=\mathbf O_{d\times d}\).
Dimension subscripts are omitted only when they are clear from context. 
For an event or set \(E\), the symbol
\(\mathbf 1_E\) denotes its indicator function. 
For nonnegative quantities \(a,b\), the notation $a\lesssim b$ means $a\le Cb$ 
for a universal positive constant \(C\).  We write
\(a\gtrsim b\) for \(b\lesssim a\), and \(a\asymp b\) when both
comparisons hold. 
For nonnegative functions \(a,b\) of the relevant asymptotic parameters,
the relations $a=O(b)$, $a=\Omega(b)$, and $a=\Theta(b)$ 
mean, respectively, that $a\le Cb$, $a\ge cb$, and $cb\le a\le Cb$, 
throughout the asymptotic regime under consideration, for constants 
\(c,C>0\) independent of the displayed asymptotic parameters. The notation \(\widetilde O(\cdot)\), \(\widetilde\Omega(\cdot)\), and
\(\widetilde\Theta(\cdot)\) hides polylogarithmic factors in the relevant
parameters.

\section{Problem Setup}
\label{sec:setup}
An exact zeroth-order oracle for a function \(f:B_2^d\to\R\) returns the
exact real value $\mathsf O_f(\mathbf x)=f(\mathbf x)$, $\mathbf x\in B_2^d$. 
Each oracle call returns a single scalar function value. Thus an algorithmic
step that evaluates both \(f(\mathbf x+\delta\mathbf u)\) and
\(f(\mathbf x-\delta\mathbf u)\) uses two oracle calls, and a full central
coordinate finite difference estimate uses \(2d\) oracle calls. Throughout
this paper, oracle complexity always counts individual scalar evaluations.

\subsection{Algorithms and Complexity Measures}
\label{subsec:algorithm-classes}

We define \textit{deterministic} and \textit{randomized} algorithms separately at each
fixed query budget.  This fixed-budget formulation is the one used in the
oracle-complexity definition below.

\paragraph{Deterministic algorithms.}

Fix \(T\in\mathbb N\).  A \textit{deterministic} adaptive \(T\)-query exact
zeroth-order algorithm is a tuple $\mathsf d
=
\bigl(
\mathbf x_1,
\pi_2,\ldots,\pi_T,
\pi_{\mathrm{out}}
\bigr)$, 
where $\mathbf x_1\in B_2^d$, \(\pi_t:\mathbb R^{t-1}\to B_2^d\) for  $t=2,\ldots,T$ 
and $\pi_{\mathrm{out}}:
\mathbb R^T\to B_2^d$ 
are Borel measurable.

When \(\mathsf d\) is run on a function
\(f:B_2^d\to\mathbb R\), its first query is the prescribed point
\(\mathbf x_1\).  After observing $y_s
:=
f(\mathbf x_s)$, $s=1,\ldots,t-1$
the algorithm chooses $\mathbf x_t
=
\pi_t(y_1,\ldots,y_{t-1})$, $t=2,\ldots,T$. 
After the \(T\)-th query, it returns
\[
\widehat{\mathbf x}_T(\mathsf d,f)
:=
\pi_{\mathrm{out}}(y_1,\ldots,y_T).
\]
Denote the class of \textit{deterministic} adaptive \(T\)-query exact
zeroth-order algorithms by $\mathcal A_{\mathrm{det}}^{(T)}$.
\paragraph{Randomized algorithms.}
A \textit{randomized} adaptive \(T\)-query exact zeroth-order algorithm
\(\mathsf r\) consists of a probability space $(
\Omega_{\mathsf r},
\mathcal H_{\mathsf r},
\mathbb P_{\mathsf r}
)$ 
and jointly measurable maps
\[
\Pi_1:
\Omega_{\mathsf r}\to B_2^d,\quad
\Pi_t:
\Omega_{\mathsf r}\times\mathbb R^{t-1}
\to
B_2^d,
\quad
t=2,\ldots,T, \quad\text{and }\quad \Pi_{\mathrm{out}}:
\Omega_{\mathsf r}\times\mathbb R^T
\to
B_2^d.
\]
For a fixed seed
\(\omega\in\Omega_{\mathsf r}\), define
\[
\mathbf x_1^\omega
:=
\Pi_1(\omega),\quad
\pi_t^\omega(\mathbf y)
:=
\Pi_t(\omega,\mathbf y),
\quad
t=2,\ldots,T,\quad  \text{and }\quad\pi_{\mathrm{out}}^\omega(\mathbf y)
:=
\Pi_{\mathrm{out}}(\omega,\mathbf y).
\]
Then $\mathsf d_\omega
:=
\bigl(
\mathbf x_1^\omega,
\pi_2^\omega,\ldots,\pi_T^\omega,
\pi_{\mathrm{out}}^\omega
\bigr)
\in
\mathcal A_{\mathrm{det}}^{(T)}$ 
is the \textit{deterministic} algorithm obtained by fixing the seed. 
For a fixed Borel measurable objective
\(f:B_2^d\to\mathbb R\) and a seed
\(\omega\in\Omega_{\mathsf r}\), define the realized query sequence
recursively by $\mathbf x_1(\mathsf r,f;\omega)
:=
\Pi_1(\omega)$, 
and, for \(t=2,\ldots,T\),
\[
\mathbf x_t(\mathsf r,f;\omega)
:=
\Pi_t\left(
\omega,
f(\mathbf x_1(\mathsf r,f;\omega)),
\ldots,
f(\mathbf x_{t-1}(\mathsf r,f;\omega))
\right).
\]
The corresponding realized output is
\[
\begin{aligned}
\widehat{\mathbf x}_T(\mathsf r,f;\omega)
:=
\Pi_{\mathrm{out}}\Bigl(
\omega,\,
f(\mathbf x_1(\mathsf r,f;\omega)),
\ldots, f(\mathbf x_T(\mathsf r,f;\omega))
\Bigr).
\end{aligned}
\]
Equivalently, $\widehat{\mathbf x}_T(\mathsf r,f;\omega)
=
\widehat{\mathbf x}_T(\mathsf d_\omega,f)$. 
The random output of \(\mathsf r\) on \(f\) is the
\(B_2^d\)-valued random vector $\breve{\mathbf x}^{\mathrm{out}}_T(\mathsf r,f)
:
\Omega_{\mathsf r}\to B_2^d$ 
defined by $\breve{\mathbf x}^{\mathrm{out}}_T(\mathsf r,f)(\omega)
:=
\widehat{\mathbf x}_T(\mathsf r,f;\omega)$. 
We denote the class of \textit{randomized} adaptive \(T\)-query exact zeroth-order
algorithms by $\mathcal A_{\mathrm{rand}}^{(T)}$.  
A \textit{deterministic} algorithm is identified with a \textit{randomized} algorithm whose
seed space consists of a single point.  Thus $\mathcal A_{\mathrm{det}}^{(T)}
\subseteq
\mathcal A_{\mathrm{rand}}^{(T)}$.  
For later use, define $\mathcal A_{\mathrm{det}}
:=
\bigcup_{T\ge1}
\mathcal A_{\mathrm{det}}^{(T)}$ and $\mathcal A_{\mathrm{rand}}
:=
\bigcup_{T\ge1}
\mathcal A_{\mathrm{rand}}^{(T)}$.

\paragraph{Function class.}

For \(L>0\), let
\[
\mathcal F_L
:=
\left\{
f:B_2^d\to\mathbb R:
\begin{array}{l}
f\text{ is convex, and}\\[1mm]
|f(\mathbf x)-f(\mathbf y)|
\le
L\|\mathbf x-\mathbf y\|
\text{ for all }
\mathbf x,\mathbf y\in B_2^d
\end{array}
\right\}.
\]
We regard the objective functions as elements of $\mathcal C(B_2^d)$,  
the Banach space of real valued continuous functions on
\(B_2^d\), equipped with the uniform norm $\|f\|_\infty
:=
\sup_{\mathbf x\in B_2^d}
|f(\mathbf x)|$ 
and its Borel sigma field. 
For every \(L>0\), the class
\(\mathcal F_L\)
is a closed, and hence Borel, subset of
\(\mathcal C(B_2^d)\).
Indeed, convexity and the \(L\)-Lipschitz inequality are preserved under
uniform limits. 
No differentiability or gradient-Lipschitz assumption is imposed.
The hard functions constructed below belong to
\(\mathcal F_2\).

\paragraph{Optimization error.}

For \(f\in\mathcal F_L\), write
\[
f^\star
:=
\min_{\mathbf x\in B_2^d}
f(\mathbf x).
\]
The minimum is attained because \(f\) is continuous and
\(B_2^d\) is compact.

For $\mathsf d
\in
\mathcal A_{\mathrm{det}}^{(T)}$, 
define the \textit{deterministic} optimization error
\[
\operatorname{err}_T(\mathsf d,f)
:=
f\left(
\widehat{\mathbf x}_T(\mathsf d,f)
\right)
-
f^\star.
\]

For $\mathsf r
\in
\mathcal A_{\mathrm{rand}}^{(T)}$ 
and a seed
\(\omega\in\Omega_{\mathsf r}\), define the realized optimization error
\[
\operatorname{err}_T(\mathsf r,f;\omega)
:=
f\left(
\widehat{\mathbf x}_T(\mathsf r,f;\omega)
\right)
-
f^\star.
\]
Equivalently,
\[
\operatorname{err}_T(\mathsf r,f;\omega)
=
\operatorname{err}_T(\mathsf d_\omega,f).
\]

The random optimization error of \(\mathsf r\) on \(f\) is the
measurable map
\[
\operatorname{err}_T(\mathsf r,f):
\Omega_{\mathsf r}
\to
[0,\infty)
\]
defined by
\[
\bigl[
\operatorname{err}_T(\mathsf r,f)
\bigr](\omega)
:=
\operatorname{err}_T(\mathsf r,f;\omega).
\]
Equivalently, as an equality of random variables,
\[
\operatorname{err}_T(\mathsf r,f)
=
f\left(
\breve{\mathbf x}^{\mathrm{out}}_T(\mathsf r,f)
\right)
-
f^\star.
\]
\begin{lemma}[Measurability of adaptive exact value algorithms]
\label{lem:algorithm-error-measurability}
The evaluation map $\operatorname{ev}:
\mathcal C(B_2^d)
\times
B_2^d
\to
\mathbb R$ where $\operatorname{ev}(f,\mathbf x)
:=
f(\mathbf x)$ 
is continuous.  Moreover, the optimal value map $\operatorname{val}:
\mathcal C(B_2^d)
\to
\mathbb R$, where  $\operatorname{val}(f)
:=
\min_{\mathbf x\in B_2^d}f(\mathbf x)$, 
is \(1\)-Lipschitz with respect to the uniform norm. 
Consequently, for every
\(\mathsf d\in\mathcal A_{\mathrm{det}}^{(T)}\),
the maps $f
\longmapsto
\widehat{\mathbf x}_T(\mathsf d,f)$ 
and $f
\longmapsto
\Err_T(\mathsf d,f)$ 
are Borel measurable on
\(\mathcal C(B_2^d)\). 
Likewise, for every
\(\mathsf r\in\mathcal A_{\mathrm{rand}}^{(T)}\),
the maps $(\omega,f)
\longmapsto
\widehat{\mathbf x}_T(\mathsf r,f;\omega)$ 
and $(\omega,f)
\longmapsto
\Err_T(\mathsf r,f;\omega)$ 
are measurable with respect to $\mathcal H_{\mathsf r}
\otimes
\mathcal B\bigl(
\mathcal C(B_2^d)
\bigr)$.  
\end{lemma}

\begin{proof}
To prove continuity of the evaluation map, suppose that $f_n\to f$ uniformly and  $\mathbf x_n\to\mathbf x$. 
Then
\[
\begin{aligned}
|f_n(\mathbf x_n)-f(\mathbf x)|
&\le
\|f_n-f\|_\infty
+
|f(\mathbf x_n)-f(\mathbf x)|
\longrightarrow
0,
\end{aligned}
\]
because \(f\) is continuous. 

For the optimal value map, for every 
\(f,g\in\mathcal C(B_2^d)\), $\min_{B_2^d}f
\le
\min_{B_2^d}g
+
\|f-g\|_\infty$. 
Interchanging \(f\) and \(g\) gives $\left|
\min_{B_2^d}f
-
\min_{B_2^d}g
\right|
\le
\|f-g\|_\infty$.

Now fix
\(\mathsf d\in\mathcal A_{\mathrm{det}}^{(T)}\).
The first query is constant and hence Borel measurable as a function of
\(f\).  Suppose inductively that $f
\longmapsto
\mathbf x_s(\mathsf d,f)$ and  $f
\longmapsto
f\bigl(
\mathbf x_s(\mathsf d,f)
\bigr)$ 
are Borel measurable for
\(s<t\).
Since \(\pi_t\) is Borel measurable,
\[
f
\longmapsto
\mathbf x_t(\mathsf d,f)
=
\pi_t\left(
f(\mathbf x_1(\mathsf d,f)),
\ldots,
f(\mathbf x_{t-1}(\mathsf d,f))
\right)
\]
is Borel measurable.  Continuity of the evaluation map then implies that $f
\longmapsto
f\bigl(
\mathbf x_t(\mathsf d,f)
\bigr)$ 
is Borel measurable.  Induction proves that the entire transcript and
the final output are Borel measurable functions of \(f\).

Therefore $f
\longmapsto
f\left(
\widehat{\mathbf x}_T(\mathsf d,f)
\right)
-
\min_{\mathbf x\in B_2^d}f(\mathbf x)$ 
is Borel measurable. 
The randomized assertion follows from the same induction on the product
space $\Omega_{\mathsf r}
\times
\mathcal C(B_2^d)$, 
using the joint measurability of
\(\Pi_1,\ldots,\Pi_T,\Pi_{\mathrm{out}}\).
\end{proof}

\paragraph{Constant-success oracle complexity.}

For a function class \(\mathcal F\), define the \textit{randomized}
constant-success exact value oracle complexity by
\begin{equation}
\begin{aligned}
T_\epsilon
\left(
\mathcal A_{\mathrm{rand}},
\mathcal F
\right)
:=
\inf
\Biggl\{
T\in\mathbb N:
\  \exists\ 
\mathsf r
\in
\mathcal A_{\mathrm{rand}}^{(T)}
\text{ s.t. }
\sup_{f\in\mathcal F}
\mathbb P_{\mathsf r}
\left(
\Err_T(\mathsf r,f)
>
\epsilon
\right)
\le
\frac12
\Biggr\}.
\end{aligned}
\label{eq:oracle-complexity-definition}
\end{equation}
Here
\(\mathbb P_{\mathsf r}\)
denotes probability with respect to the algorithmic seed on $(\Omega_{\mathsf r},
\mathcal H_{\mathsf r},
\mathbb P_{\mathsf r}
)$. 
We use the convention that the infimum of the empty set is
\(+\infty\).

\subsection{The Hard Distribution}
\label{sec:hard-family}

We now define the hard distribution used in the lower bound.  For every
pair \((d,k)\), the distribution is supported on the same function class
\(\mathcal F_2\), while the number \(k\) of linear pieces determines the
accuracy scale of the resulting hard instances.

Let $\rvec{g}
\sim
N(\mathbf 0,\mathbf I_d),$ 
and define $p_d^{\mathrm{tr}}
:=
\mathbb P\left(
\|\rvec{g}\|_2
\le
2\sqrt d
\right)$. 
Since the standard Gaussian density is strictly positive, $p_d^{\mathrm{tr}}>0.$ 
Define the truncated Gaussian probability measure
\(\gamma_d^{\mathrm{tr}}\) on \(\mathbb R^d\) by
\begin{equation}
\gamma_d^{\mathrm{tr}}(E)
:=
\frac{
\mathbb P\left(
\rvec{g}\in E,\,
\|\rvec{g}\|_2\le2\sqrt d
\right)
}{
p_d^{\mathrm{tr}}
},
\qquad
E\in\mathcal B(\mathbb R^d).
\label{eq:truncated-gaussian-law}
\end{equation}
Equivalently,
\begin{equation}
\gamma_d^{\mathrm{tr}}(\mathrm d\mathbf b)
=
\frac{
(2\pi)^{-d/2}
e^{-\|\mathbf b\|_2^2/2}
\mathbf 1_{\{\|\mathbf b\|_2\le2\sqrt d\}}
}{
p_d^{\mathrm{tr}}
}
\,\mathrm d\mathbf b.
\label{eq:truncated-gaussian-density}
\end{equation} 
In conditional-law notation, $\gamma_d^{\mathrm{tr}}
=
\Law\left(
\rvec{g}
\;\middle|\;
\|\rvec{g}\|_2\le2\sqrt d
\right)$. 
Here the vertical bar denotes conditioned on the event
\(\{\|\rvec{g}\|_2\le2\sqrt d\}\); explicitly,
\eqref{eq:truncated-gaussian-law} is the definition of this conditional
law. 
On a probability space $(
\Omega_{\mathrm{fun}},
\mathcal H_{\mathrm{fun}},
\mathbb P_{\mathrm{fun}}
)$,  
let $\rvec{b}_1,\ldots,\rvec{b}_k
\overset{\mathrm{i.i.d.}}{\sim}
\gamma_d^{\mathrm{tr}}$.  
This function side randomness is assumed to be independent of any
algorithmic seed.  We use  $\Xi
:=
\bigl(
\rvec{b}_1,\ldots,\rvec{b}_k
\bigr)$ 
as the joint function seed, and write $\mathbb P_\Xi
=
\left(
\gamma_d^{\mathrm{tr}}
\right)^{\otimes k}$ 
for its law.

Define the scaled random vectors $\rvec{a}_i
:=
\frac{\rvec{b}_i}{\sqrt d}$ for  $i=1,\ldots,k$. 
Then $\|\rvec{a}_i\|_2
\le
2$, $\mathbb P_\Xi$-almost surely. 
Let $\mathsf S_{d,k}
:=
\left(
2\sqrt d\,B_2^d
\right)^k
\subseteq
(\mathbb R^d)^k$. 
By construction, $\mathbb P_\Xi(
\mathsf S_{d,k})
=
1$.  
For a realization $\xi
=
\bigl(
\mathbf b_1,\ldots,\mathbf b_k
\bigr)
\in
\mathsf S_{d,k}$ 
of the seed \(\Xi\), set $\mathbf a_i(\xi)
:=
\frac{\mathbf b_i}{\sqrt d}$, for $i=1,\ldots,k$,  
and define the corresponding \textit{deterministic} objective by
\begin{equation}
f_\xi(\mathbf x)
:=
\max_{1\le i\le k}
\left\langle
\mathbf a_i(\xi),
\mathbf x
\right\rangle,
\qquad
\mathbf x\in B_2^d.
\label{eq:hard-function}
\end{equation}
The random hard objective is denoted by \(f_\Xi\). 
For every realization \(\xi\), the function \(f_\xi\) is convex as the
maximum of finitely many linear functions.  Moreover, for
\(\mathbf x,\mathbf y\in B_2^d\),
\[
\begin{aligned}
|f_\xi(\mathbf x)-f_\xi(\mathbf y)|
\le
\max_{1\le i\le k}
\left|
\left\langle
\mathbf a_i(\xi),
\mathbf x-\mathbf y
\right\rangle
\right|
\le
2\|\mathbf x-\mathbf y\|_2.
\end{aligned}
\]
For every
\(\xi\in\mathsf S_{d,k}\),
the preceding argument gives $f_\xi\in\mathcal F_2$. 
Define $\Phi_{d,k}:
(\mathbb R^d)^k
\to
\mathcal C(B_2^d)$, where  $\Phi_{d,k}(\xi)
:=
f_\xi$.  
For $\xi
=
(\mathbf b_1,\ldots,\mathbf b_k)$ and  $\xi'
=
(\mathbf b_1',\ldots,\mathbf b_k')$, 
the elementary inequality $\left|
\max_i u_i-\max_i v_i
\right|
\le
\max_i|u_i-v_i|$ 
gives
\[
\begin{aligned}
\|f_\xi-f_{\xi'}\|_\infty
&=
\sup_{\mathbf x\in B_2^d}
\left|
\max_{1\le i\le k}
\frac{
\langle\mathbf b_i,\mathbf x\rangle
}{
\sqrt d
}
-
\max_{1\le i\le k}
\frac{
\langle\mathbf b_i',\mathbf x\rangle
}{
\sqrt d
}
\right|
\\
&\le
\frac1{\sqrt d}
\max_{1\le i\le k}
\sup_{\mathbf x\in B_2^d}
\left|
\left\langle
\mathbf b_i-\mathbf b_i',
\mathbf x
\right\rangle
\right|\le
\frac1{\sqrt d}
\max_{1\le i\le k}
\|\mathbf b_i-\mathbf b_i'\|_2.
\end{aligned}
\]
Hence
\(\Phi_{d,k}\)
is continuous, and in particular Borel measurable. 
Therefore $f_\Xi
=
\Phi_{d,k}(\Xi)$ 
is a
\(\mathcal C(B_2^d)\)-valued Borel random element.  Define $\mathfrak D_{d,k}
:=
(\Phi_{d,k})_\sharp
\mathbb P_\Xi
=
\Law(f_\Xi).$ 
Thus
\(\mathfrak D_{d,k}\)
is a well defined Borel probability measure on
\(\mathcal C(B_2^d)\).
Since $\Phi_{d,k}(
\mathsf S_{d,k})
\subseteq
\mathcal F_2$ 
and
\(\mathbb P_\Xi(\mathsf S_{d,k})=1\),
we have $\mathfrak D_{d,k}(\mathcal F_2)
=
1$.  

Define the random matrix $\rmat{A}
:=
\begin{pmatrix}
\rvec{a}_1,
\ldots,
\rvec{a}_k
\end{pmatrix}^\top
\in
\mathbb R^{k\times d}$.  
For a \textit{deterministic} realization $\mathbf A
=
\begin{pmatrix}
\mathbf a_1,
\ldots,
\mathbf a_k
\end{pmatrix}^\top$, 
we also write $f_{\mathbf A}(\mathbf x)
:=
\max_{1\le i\le k}
\langle
\mathbf a_i,
\mathbf x
\rangle$.  
Finally, define the random aggregate directions $\rvec{s}
:=
\sum_{i=1}^k
\rvec{b}_i$ and  $\rvec{s}^{\mathrm{sc}}
:=
\sum_{i=1}^k
\rvec{a}_i
=
\frac{\rvec{s}}{\sqrt d}$. 
For every \(\mathbf x\in B_2^d\), almost surely,
\begin{equation}
\begin{aligned}
f_\Xi(\mathbf x)
=
\max_{1\le i\le k}
\langle
\rvec{a}_i,
\mathbf x
\rangle
\ge
\frac1k
\sum_{i=1}^k
\langle
\rvec{a}_i,
\mathbf x
\rangle
=
\frac1k
\left\langle
\rvec{s}^{\mathrm{sc}},
\mathbf x
\right\rangle.
\end{aligned}
\label{eq:average-lower-bound}
\end{equation}
This elementary inequality is the geometric link between optimization
error and posterior information about the hidden aggregate direction
\(\rvec{s}^{\mathrm{sc}}\).

\subsection{Yao's Reduction}
\label{subsec:yao}

We first prove a distributional lower bound for \textit{deterministic} algorithms
against the random hard distribution
\(\mathfrak D_{d,k}\).  We then invoke the standard lower bound direction
of Yao's minimax principle~\citep{yao1977probabilistic} to obtain a
worst case lower bound for \textit{randomized} algorithms.  We state the resulting
fixed budget implication in the notation of this paper.  By
Lemma~\ref{lem:algorithm-error-measurability}
and the construction in
Subsection~\ref{sec:hard-family},
\(\mathfrak D_{d,k}\)
is a Borel probability measure on
\(\mathcal C(B_2^d)\), and all failure events used below are Borel
measurable.  The corresponding joint measurability for \textit{randomized}
algorithms also permits the applications of Tonelli's theorem in Yao's
reduction.

\begin{proposition}[Yao's lower-bound principle, fixed-budget form]
\label{prop:yao-constant-success}
Let
\(\mathcal F\subseteq\mathcal C(B_2^d)\)
be Borel, and let
\(\mathfrak D\)
be a Borel probability measure on
\(\mathcal C(B_2^d)\)
such that $\mathfrak D(\mathcal F)=1$. Fix $T\in\mathbb N$,  $\epsilon>0$ and $p>\frac12$.  
Suppose that every \textit{deterministic} \(T\)-query algorithm $\mathsf d
\in
\mathcal A_{\mathrm{det}}^{(T)}$
satisfies
\begin{equation}
\mathbb P_{f\sim\mathfrak D}
\left(
\Err_T(\mathsf d,f)>\epsilon
\right)
\ge
p.
\label{eq:yao-distributional-assumption}
\end{equation}
Then $T_\epsilon
\left(
\mathcal A_{\mathrm{rand}},
\mathcal F
\right)
>
T$. 
\end{proposition}

\subsection{Known Upper Bounds for Comparison}
\label{subsec:upper-bounds}

We record two known upper bounds after translating them into the exact
scalar evaluation model and the constant success criterion used in this
paper.  Neither result is used in the lower bound proof.

The first bound follows from the Euclidean two point regret estimate of
\citet[Corollary~2]{JMLR:v18:16-632}.  Setting
\(f_t\equiv f\) in every round and returning the average of the iterates
converts the regret bound into an expected optimization error bound by
convexity.  Each round uses two exact scalar evaluations.  Running the
method on a slight contraction of \(B_2^d\) keeps both perturbed query
points in \(B_2^d\), while Lipschitz continuity controls the error caused
by the contraction.  Finally, running the expected error guarantee at
accuracy \(\epsilon/2\) and applying Markov's inequality yields success
probability at least \(1/2\).

\begin{proposition}[Two-point exact value upper bound]
\label{prop:upper-two-point}
There exists a universal constant \(C>0\) such that, for every
\(d\ge1\), \(L>0\), and \(0<\epsilon\le L\), $T_\epsilon
\left(
\mathcal A_{\mathrm{rand}},
\mathcal F_L
\right)
\le
C
\frac{dL^2}{\epsilon^2}$. 
\end{proposition}
The second bound is the \textit{deterministic} value only method of
\citet{protasov1996algorithms}.  Their accuracy parameter is relative to the
objective range.  We translate it into the absolute error convention
used here.
\begin{proposition}[\textit{deterministic} exact-value upper bound]
\label{prop:upper-protasov}
There exists a universal constant \(C>0\) such that, for every
\(d\ge1\), \(L>0\), and \(0<\epsilon\le L\), 
\begin{equation}
   T_\epsilon
\left(
\mathcal A_{\mathrm{rand}},
\mathcal F_L
\right)
\le
C d^2
\log(d+1)
\log\left(
\frac{4L}{\epsilon}
\right).  
\end{equation}
\end{proposition}
Combining
Propositions~\ref{prop:upper-two-point}
and~\ref{prop:upper-protasov}, up to polylogarithmic
factors, the upper bound scale is $d
\min\left\{
d,
\left(\frac{L}{\epsilon}\right)^2
\right\}$.  
For \(L=\Theta(1)\), this becomes $d\min\{d,\epsilon^{-2}\}$. 

\section{Main Results}
\label{sec:main-results}

We now state the main lower bound results. 
The first result is a fixed budget
distributional lower bound for \textit{deterministic} algorithms.  It is
parameterized by the number \(k\) of linear pieces in the random support
function.  With constant probability, the hard objective has a negative
optimum on the scale \(k^{-1/2}\), while fewer than order $\frac{dk}{\log(ek)}$ 
exact scalar evaluations do not reveal enough information to attain this
scale.  The \textit{randomized} worst case lower bound follows by taking $k
\asymp
\min\{d,\epsilon^{-2}\}$ 
and applying
Proposition~\ref{prop:yao-constant-success}.

\begin{theorem}[Distributional lower bound for \textit{deterministic} algorithms]
\label{thm:deterministic-hard-lb}
There exist universal positive constants $c_{\mathrm{dim}}$, $c_{\mathrm{prob}}$, and $c_{\mathrm{err}}$  
and an integer
\(d_0\in\mathbb N\)
such that the following holds.  For every
\(d\ge d_0\), every integer \(k\) satisfying $1
\le
k
\le
c_{\mathrm{dim}}d$,  
every integer \(T\) satisfying $T
\le
c_{\mathrm{prob}}
\frac{dk}{\log(ek)}$,  
and every \textit{deterministic} algorithm $\mathsf d
\in
\mathcal A_{\mathrm{det}}^{(T)}$,  
we have
\begin{equation}
\mathbb P_\Xi
\left(
\Err_T(
\mathsf d,
f_\Xi)
>
\frac{c_{\mathrm{err}}}{\sqrt k}
\right)
\ge
\frac34.
\label{eq:distributional-prob}
\end{equation}
\end{theorem}
The probability in
\eqref{eq:distributional-prob}
is taken over the function seed \(\Xi\), or equivalently over the random
hard objective $f_\Xi
\sim
\mathfrak D_{d,k}$. 
Since
\(\mathfrak D_{d,k}\)
is supported on
\(\mathcal F_2\),
Theorem~\ref{thm:deterministic-hard-lb}
is precisely a distributional lower bound of the form required by
Proposition~\ref{prop:yao-constant-success}.
\begin{theorem}[Exact scalar-value oracle lower bound]
\label{thm:main-exact-zo-lb} Set \(L_0:=2\). 
There exist universal constants
\(c>0\),
\(0<\epsilon_0\le L_0\),
and \(d_0\in\mathbb N\)
such that, for every \(d\ge d_0\) and every
\(0<\epsilon\le\epsilon_0\),
\begin{equation}
T_\epsilon
\left(
\mathcal A_{\mathrm{rand}},
\mathcal F_{L_0}
\right)
\ge
c
\frac{
d\min\{d,\epsilon^{-2}\}
}{
\log\left(
\min\{d,\epsilon^{-2}\}
\right)
}.
\label{eq:main-full-scale}
\end{equation}
\end{theorem}
To our knowledge,
Theorem~\ref{thm:main-exact-zo-lb}
is the first full scale near-optimal lower bound for arbitrary adaptive
\textit{randomized} algorithms in the exact noiseless scalar value model.
It captures the dimension penalty specific to scalar feedback and
determines the joint polynomial dependence on dimension and accuracy on
both sides of the transition
\(\epsilon=d^{-1/2}\), up to logarithmic factors. 
Independent concurrent work by
\citet{kerger2026closing}
obtains the same high-accuracy lower bound scale for \textit{deterministic}
algorithms at a universal constant multiple of
\(d^{-1/2}\).
Our theorem additionally applies to \textit{randomized} algorithms and covers the
entire low-accuracy regime through the tunable parameter
\(k\asymp\epsilon^{-2}\).
The two accuracy regimes follow directly from
\eqref{eq:main-full-scale}.
\begin{corollary}[Low-accuracy regime]
\label{cor:low-accuracy}
Under the assumptions of
Theorem~\ref{thm:main-exact-zo-lb}, if $\epsilon
\ge
d^{-1/2}$, 
then
\begin{equation}
T_\epsilon
\left(
\mathcal A_{\mathrm{rand}},
\mathcal F_{L_0}
\right)
\ge
c
\frac{
d
}{
\epsilon^2
\log(\epsilon^{-2})
}.
\label{eq:low-accuracy-lower-bound}
\end{equation}
\end{corollary}

\begin{corollary}[High-accuracy regime]
\label{cor:high-accuracy}
Under the assumptions of
Theorem~\ref{thm:main-exact-zo-lb}, if $\epsilon
\le
d^{-1/2}$,  
then
\begin{equation}
T_\epsilon
\left(
\mathcal A_{\mathrm{rand}},
\mathcal F_{L_0}
\right)
\ge
c
\frac{
d^2
}{
\log(d)
}.
\label{eq:high-accuracy-lower-bound}
\end{equation}
In particular, the lower bound in this regime has no remaining dependence
on \(\epsilon\).
\end{corollary}
Ignoring logarithmic factors, Corollaries~\ref{cor:low-accuracy}
and~\ref{cor:high-accuracy} yield the lower-bound scales $\widetilde\Omega(d\epsilon^{-2})$  and  $\widetilde\Omega(d^2)$ 
in the low- and high-accuracy regimes, respectively.  Both are
near-optimal: the former matches the \(O(d\epsilon^{-2})\) two-point upper
bound up to a logarithmic factor, and the latter matches the
\(\widetilde O(d^2)\) evaluation-oracle upper bound up to polylogarithmic
factors.  The high-accuracy lower bound is independent of
\(\epsilon\).

\begin{table}[t]
\centering
\small
\caption{
Comparison of exact scalar value oracle bounds for convex
\(L_0\)-Lipschitz functions over \(B_2^d\), with \(L_0=2\).
All bounds count individual exact scalar evaluations.
}
\label{tab:full-scale-comparison}

\begin{tabular}{
@{}
>{\raggedright\arraybackslash}p{0.20\textwidth}
>{\centering\arraybackslash}p{0.32\textwidth}
>{\centering\arraybackslash}p{0.42\textwidth}
@{}
}
\toprule

Regime
&
Low accuracy~
\(\epsilon\gtrsim d^{-1/2}\)
&
High accuracy~
\(\epsilon\lesssim d^{-1/2}\)
\\

\midrule

Known upper bounds
&

\(\displaystyle O(d\epsilon^{-2})\)
\par\smallskip
\citep{JMLR:v18:16-632}
&
\(\displaystyle
O\!\left(
d^2\log(d+1)
\log\frac{1}{\epsilon}
\right)
\)
\par\smallskip
\citep{protasov1996algorithms}
\\[4mm]

Lower bounds
&
\begin{minipage}[t]{\linewidth}
\vspace{0pt}
\centering

\strut\textbf{This paper, \textit{randomized}:}
\par\smallskip

\(\displaystyle
\Omega\!\left(
\frac{
d
}{
\epsilon^2
\log(\epsilon^{-2})
}
\right)
\)

\end{minipage}
&
\begin{minipage}[t]{\linewidth}
\vspace{0pt}

{
\setlength{\tabcolsep}{3pt}
\setlength{\dashlinedash}{2.5pt}
\setlength{\dashlinegap}{2pt}
\setlength{\arrayrulewidth}{0.4pt}

\begin{tabular}[t]{
@{}
>{\centering\arraybackslash}p{0.46\linewidth}
:
>{\centering\arraybackslash}p{0.46\linewidth}
@{}
}

\strut
\textbf{\citet{kerger2026closing}, \textit{deterministic}:}
\par\smallskip

\(\displaystyle
\Omega\!\left(
\frac{d^2}{\log(d+1)}
\right)
\)

&

\strut
\textbf{This paper, \textit{randomized}:}
\par\smallskip

\(\displaystyle
\Omega\!\left(
\frac{d^2}{\log(d)}
\right)
\)

\end{tabular}
}

\end{minipage}
\\[2mm]

\bottomrule
\end{tabular}
\end{table}

\section{Proofs of the Main Lower Bounds}
\label{sec:proof-main}

This section proves
Theorems~\ref{thm:deterministic-hard-lb}
and~\ref{thm:main-exact-zo-lb}.  The distributional lower bound combines
three ingredients.

First, Proposition~\ref{prop:augmented-energy} controls the posterior mean
energy of the hidden aggregate direction under the augmented information
generated by exact maximum values and their active indices:
\[
\mathbb E_\Xi
\left\|
\mathbb E_\Xi
\left[
\rvec{s}^{\mathrm{sc}}
\mid
\mathcal G_T
\right]
\right\|_2^2
\le
C_{\mathrm{en}}
\frac{T\log(ek)}{d}.
\]
Here $\rvec{s}^{\mathrm{sc}}
=
\sum_{i=1}^k
\rvec{a}_i$ 
is the scaled random aggregate direction,
\(\mathsf T_T\) is the augmented transcript, and $\mathcal G_T
=
\sigma(\mathsf T_T)$ 
is the sigma-field generated by that transcript. 
Second, Lemma~\ref{lem:conditional-residual-tail} shows that, conditioned
on \(\mathcal G_T\), the posterior residual of
\(\rvec{s}^{\mathrm{sc}}\) is subgaussian in every
\(\mathcal G_T\)-measurable direction in \(B_2^d\). 
Third, Lemma~\ref{lem:negative-optimum-event} below shows that, with
constant probability, the random support function has minimum value on
the scale $-k^{-1/2}$.  
Together, these estimates imply that a deterministic algorithm using $T
\lesssim
\frac{dk}{\log(ek)}$  
exact scalar evaluations cannot return a point whose value is within
order \(k^{-1/2}\) of the optimum with constant probability. 
Finally, choosing $k
\asymp
\min\{d,\epsilon^{-2}\}$ 
and applying the fixed-budget form of Yao's lower-bound principle in
Proposition~\ref{prop:yao-constant-success}
gives the \textit{randomized} worst-case lower bound.

The proof below uses the augmented posterior energy estimate from
Proposition~\ref{prop:augmented-energy}, because the conditional residual
analysis also requires the product posterior structure available under
\(\mathcal G_T\). Although the same energy bound also holds after conditioning down to the
actual value filtration, the proof below works with
\(\mathcal G_T\), since the residual bound relies on the augmented product
posterior.

The following lemma supplies the geometric component of the proof.
\begin{lemma}[Negative optimum event]
\label{lem:negative-optimum-event}
There exist universal constants
\(c_{\mathrm{dim}},\gamma_{\mathrm{opt}}>0\)
and an integer
\(d_{\mathrm{opt}}\in\mathbb N\)
such that, for every
\(d\ge d_{\mathrm{opt}}\)
and every integer $1
\le
k
\le
c_{\mathrm{dim}}d$,  
we have
\[
\mathbb P_\Xi
\left(
f_\Xi^\star
\le
-\frac{\gamma_{\mathrm{opt}}}{\sqrt k}
\right)
\ge
\frac78,
\]
where $f_\Xi^\star
:=
\min_{\mathbf x\in B_2^d}
f_\Xi(\mathbf x)$.  
\end{lemma}
\begin{proof}[Proof of Lemma~\ref{lem:negative-optimum-event}]
Let $c_{\mathrm{sv}}
:=
\frac14$. 
By Theorem~\ref{thm:sv-truncation}, there exists
\(d_{\mathrm{opt}}\in\mathbb N\)
such that, for every
\(d\ge d_{\mathrm{opt}}\)
and every
\(k\le d/16\),
\[
\mathbb P_\Xi
\left(
\rmat{A}\rmat{A}^{\top}
\succeq
c_{\mathrm{sv}}\mathbf I_k
\right)
\ge
\frac78.
\]
Set $c_{\mathrm{dim}}
:=
\frac1{16}$.  

Fix a deterministic realization $\mathbf A
=
\begin{pmatrix}
\mathbf a_1^\top,
\ldots,
\mathbf a_k^\top
\end{pmatrix}$ 
satisfying $\mathbf A\mathbf A^\top
\succeq
c_{\mathrm{sv}}\mathbf I_k$.  
Define $\gamma_{\mathrm{opt}}
:=
\frac12\sqrt{c_{\mathrm{sv}}}
=
\frac14$,  
and set
\[
\mathbf x^\star(\mathbf A)
:=
-\frac{\gamma_{\mathrm{opt}}}{\sqrt k}
\mathbf A^\top
(\mathbf A\mathbf A^\top)^{-1}
\mathbf 1_k.
\]
Then $\mathbf A\mathbf x^\star(\mathbf A)
=
-\frac{\gamma_{\mathrm{opt}}}{\sqrt k}
\mathbf 1_k$.
Moreover,
\[
\begin{aligned}
\|\mathbf x^\star(\mathbf A)\|_2^2
&=
\frac{\gamma_{\mathrm{opt}}^2}{k}
\mathbf 1_k^\top
(\mathbf A\mathbf A^\top)^{-1}
\mathbf 1_k
\le
\frac{\gamma_{\mathrm{opt}}^2}{k}
\cdot
\frac1{c_{\mathrm{sv}}}
\|\mathbf 1_k\|_2^2
=
\frac{\gamma_{\mathrm{opt}}^2}{c_{\mathrm{sv}}}
=
\frac14.
\end{aligned}
\]
Hence $\mathbf x^\star(\mathbf A)
\in
B_2^d$.  
Since every coordinate of
\(\mathbf A\mathbf x^\star(\mathbf A)\)
equals
\(-\gamma_{\mathrm{opt}}/\sqrt k\),
\[
\begin{aligned}
f_{\mathbf A}^\star
:=
\min_{\mathbf x\in B_2^d}
f_{\mathbf A}(\mathbf x)
\le
f_{\mathbf A}
\bigl(
\mathbf x^\star(\mathbf A)
\bigr)
=
-\frac{\gamma_{\mathrm{opt}}}{\sqrt k}.
\end{aligned}
\]
Therefore
\[
\left\{
\rmat{A}\rmat{A}^{\top}
\succeq
c_{\mathrm{sv}}\mathbf I_k
\right\}
\subseteq
\left\{
f_\Xi^\star
\le
-\frac{\gamma_{\mathrm{opt}}}{\sqrt k}
\right\}.
\]
Taking probabilities proves the lemma.
\end{proof}

\begin{proof}[Proof of Theorem~\ref{thm:deterministic-hard-lb}]
Let $c_{\mathrm{dim}}^{\mathrm{opt}}$,  $\gamma_{\mathrm{opt}}>0$ and  $d_{\mathrm{opt}}\in\mathbb N$ 
be the universal constants supplied by
Lemma~\ref{lem:negative-optimum-event}.  Thus, whenever $d\ge d_{\mathrm{opt}}$ and  $1\le k\le c_{\mathrm{dim}}^{\mathrm{opt}}d$,  
we have
\begin{equation}
\Pp_\Xi
\left(
f_\Xi^\star
\le
-\frac{\gamma_{\mathrm{opt}}}{\sqrt k}
\right)
\ge
\frac78.
\label{eq:det-lb-opt-event}
\end{equation}

Let \(C_{\mathrm{en}}>0\) be the universal constant in
Proposition~\ref{prop:augmented-energy}.  Set
\[
\eta
:=
\frac{\gamma_{\mathrm{opt}}}{16},\qquad
c_{\mathrm{prob}}
:=
\frac{\eta^2}{16C_{\mathrm{en}}},
\qquad
c_{\mathrm{err}}
:=
\frac{7\gamma_{\mathrm{opt}}}{16},\qquad \text{and} \qquad d_{\mathrm{tail}}
:=
\left\lceil
\frac{2\log 16}{\eta^2}
\right\rceil.
\]
Finally, define $c_{\mathrm{dim}}
:=
c_{\mathrm{dim}}^{\mathrm{opt}}$ and $d_0
:=
\max\{
d_{\mathrm{opt}},
d_{\mathrm{tail}}
\}$. 
Fix $d\ge d_0$, $1\le k\le c_{\mathrm{dim}}d$, 
and an integer \(T\) satisfying $T
\le
c_{\mathrm{prob}}
\frac{dk}{\log(ek)}$. 
Let $\mathsf d
\in
\mathcal A_{\mathrm{det}}^{(T)}$ 
be arbitrary deterministic algorithm.

Since the hard objective \(f_\Xi\) is random, the output of the fixed
deterministic algorithm is a random vector.  Define
\[
\rvec{x}^{\mathrm{out}}_T
:=
\widehat{\mathbf x}_T(
\mathsf d,
f_\Xi).
\]
By the definition of a deterministic exact-value algorithm,
\(\rvec{x}^{\mathrm{out}}_T\) is measurable with respect to the
value-only sigma-field \(\mathcal F_T\).  Moreover, $\rvec{x}^{\mathrm{out}}_T
\in
B_2^d$ almost surely. 
Since $\mathcal F_T
\subseteq
\mathcal G_T$,  
the output is also
\(\mathcal G_T\)-measurable. 
Define the scaled augmented posterior mean by
\[
\rvec{m}^{\mathrm{sc}}_T
:=
\E_\Xi
\left[
\rvec{s}^{\mathrm{sc}}
\mid
\mathcal G_T
\right].
\]
By the scaled conclusion of
Proposition~\ref{prop:augmented-energy},
\begin{equation}
\E_\Xi
\left\|
\rvec{m}^{\mathrm{sc}}_T
\right\|_2^2
\le
C_{\mathrm{en}}
\frac{
T\log(ek)
}{
d
}.
\label{eq:det-lb-mean-energy}
\end{equation}
Define the posterior-mean event
\[
\mathcal E_{\mathrm{mean}}
:=
\left\{
\left\|
\rvec{m}^{\mathrm{sc}}_T
\right\|_2
\le
\eta\sqrt k
\right\}.
\]
By Markov's inequality and
\eqref{eq:det-lb-mean-energy},
\[
\begin{aligned}
\Pp_\Xi(
\mathcal E_{\mathrm{mean}}^c)
&=
\Pp_\Xi
\left(
\left\|
\rvec{m}^{\mathrm{sc}}_T
\right\|_2^2
>
\eta^2k
\right)
\\
&\le
\frac{
\E_\Xi
\left\|
\rvec{m}^{\mathrm{sc}}_T
\right\|_2^2
}{
\eta^2k
}
\\
&\le
\frac{
C_{\mathrm{en}}
T\log(ek)
}{
\eta^2dk
}.
\end{aligned}
\]
Using the assumed upper bound on \(T\), we obtain
\begin{equation}\label{eq:det-lb-mean-prob}
\begin{aligned}
\Pp_\Xi(
\mathcal E_{\mathrm{mean}}^c)
\le
\frac{
C_{\mathrm{en}}
}{
\eta^2
}
c_{\mathrm{prob}}
=
\frac1{16}.
\end{aligned}
\end{equation}
Next define the posterior residual
\[
\rvec{r}_T
:=
\rvec{s}^{\mathrm{sc}}
-
\rvec{m}^{\mathrm{sc}}_T.
\]
Because
\(\rvec{x}^{\mathrm{out}}_T\)
is \(\mathcal G_T\)-measurable and belongs to \(B_2^d\),
Lemma~\ref{lem:conditional-residual-tail} gives, for every \(u>0\),
\[
\Pp_\Xi
\left(
\left\langle
\rvec{r}_T,
\rvec{x}^{\mathrm{out}}_T
\right\rangle
\le
-u
\;\middle|\;
\mathcal G_T
\right)
\le
\exp\left(
-\frac{du^2}{2k}
\right)
\qquad
\text{almost surely}.
\]
Taking expectations and choosing $u
=
\eta\sqrt k$, 
we obtain
\[
\Pp_\Xi
\left(
\left\langle
\rvec{r}_T,
\rvec{x}^{\mathrm{out}}_T
\right\rangle
\le
-\eta\sqrt k
\right)
\le
\exp\left(
-\frac{d\eta^2}{2}
\right).
\]
Since $d
\ge
d_{\mathrm{tail}}
\ge
\frac{2\log16}{\eta^2}$,  
the right-hand side is at most \(1/16\).  Define
\[
\mathcal E_{\mathrm{res}}
:=
\left\{
\left\langle
\rvec{r}_T,
\rvec{x}^{\mathrm{out}}_T
\right\rangle
\ge
-\eta\sqrt k
\right\}.
\]
Then $\Pp_\Xi(
\mathcal E_{\mathrm{res}}^c)
\le
\frac1{16}$. 
Finally, define the negative-optimum event
\[
\mathcal E_{\mathrm{opt}}
:=
\left\{
f_\Xi^\star
\le
-\frac{\gamma_{\mathrm{opt}}}{\sqrt k}
\right\}.
\]
By
\eqref{eq:det-lb-opt-event},
\begin{equation}
\Pp_\Xi(
\mathcal E_{\mathrm{opt}}^c)
\le
\frac18.
\label{eq:det-lb-opt-prob}
\end{equation}
Consider an outcome in $\mathcal E_{\mathrm{opt}}
\cap
\mathcal E_{\mathrm{mean}}
\cap
\mathcal E_{\mathrm{res}}$. 
Using $\rvec{s}^{\mathrm{sc}}
=
\rvec{m}^{\mathrm{sc}}_T
+
\rvec{r}_T$,
we obtain
\[
\begin{aligned}
\left\langle
\rvec{s}^{\mathrm{sc}},
\rvec{x}^{\mathrm{out}}_T
\right\rangle
&=
\left\langle
\rvec{m}^{\mathrm{sc}}_T,
\rvec{x}^{\mathrm{out}}_T
\right\rangle
+
\left\langle
\rvec{r}_T,
\rvec{x}^{\mathrm{out}}_T
\right\rangle
\\
&\ge
-
\left\|
\rvec{m}^{\mathrm{sc}}_T
\right\|_2
\left\|
\rvec{x}^{\mathrm{out}}_T
\right\|_2
-
\eta\sqrt k
\\
&\ge
-2\eta\sqrt k.
\end{aligned}
\]
The last inequality uses $\left\|
\rvec{x}^{\mathrm{out}}_T
\right\|_2
\le1$  
and the definition of
\(\mathcal E_{\mathrm{mean}}\).

Applying the pointwise average lower bound
\eqref{eq:average-lower-bound} at
\(\rvec{x}^{\mathrm{out}}_T\), we get
\[
\begin{aligned}
f_\Xi
\left(
\rvec{x}^{\mathrm{out}}_T
\right)
\ge
\frac1k
\left\langle
\rvec{s}^{\mathrm{sc}},
\rvec{x}^{\mathrm{out}}_T
\right\rangle
\ge
-\frac{2\eta}{\sqrt k}
=
-\frac{\gamma_{\mathrm{opt}}}{8\sqrt k}.
\end{aligned}
\]
On the event
\(\mathcal E_{\mathrm{opt}}\), $f_\Xi^\star
\le
-\frac{\gamma_{\mathrm{opt}}}{\sqrt k}$. 
Consequently,
\[
\begin{aligned}
\Err_T(
\mathsf d,
f_\Xi)
=
f_\Xi
\left(
\rvec{x}^{\mathrm{out}}_T
\right)
-
f_\Xi^\star
\ge
-\frac{\gamma_{\mathrm{opt}}}{8\sqrt k}
+
\frac{\gamma_{\mathrm{opt}}}{\sqrt k}
=
\frac{
7\gamma_{\mathrm{opt}}
}{
8\sqrt k
}
>
\frac{
c_{\mathrm{err}}
}{
\sqrt k
}.
\end{aligned}
\]

No independence between the three events is required.  By the union bound,
\[
\begin{aligned}
&
\Pp_\Xi
\left(
\mathcal E_{\mathrm{opt}}
\cap
\mathcal E_{\mathrm{mean}}
\cap
\mathcal E_{\mathrm{res}}
\right)
\\
&\qquad\ge
1
-
\Pp_\Xi(
\mathcal E_{\mathrm{opt}}^c)
-
\Pp_\Xi(
\mathcal E_{\mathrm{mean}}^c)
-
\Pp_\Xi(
\mathcal E_{\mathrm{res}}^c)
\\
&\qquad\ge
1
-
\frac18
-
\frac1{16}
-
\frac1{16}
\\
&\qquad=
\frac34.
\end{aligned}
\]

Therefore $\Pp_\Xi
\left(
\Err_T(
\mathsf d,
f_\Xi)
>
\frac{
c_{\mathrm{err}}
}{
\sqrt k
}
\right)
\ge
\frac34$. 
Since
\(\mathsf d\in
\mathcal A_{\mathrm{det}}^{(T)}\)
was arbitrary, this proves the theorem.
\end{proof}
\begin{proof}[Proof of Theorem~\ref{thm:main-exact-zo-lb}]
Let $m_\epsilon
:=
\min\{d,\epsilon^{-2}\}$.
Choose $a_0
:=
\min\left\{
\frac{c_{\mathrm{dim}}}{2},
\frac{c_{\mathrm{err}}^2}{16},
\frac14
\right\}$, 
where
\(c_{\mathrm{dim}}\)
and
\(c_{\mathrm{err}}\)
are the constants in
Theorem~\ref{thm:deterministic-hard-lb}. 
By decreasing \(\epsilon_0\) and increasing \(d_0\), if necessary, we may
assume that $a_0m_\epsilon
\ge
2$. 
Set $k
:=
\left\lfloor
a_0m_\epsilon
\right\rfloor$.
Then
\[
\frac{a_0}{2}m_\epsilon
\le
k
\le
a_0m_\epsilon
\le
c_{\mathrm{dim}}d.
\]
Moreover,
\[
\frac{c_{\mathrm{err}}}{\sqrt k}
\ge
\frac{c_{\mathrm{err}}}{\sqrt{a_0m_\epsilon}}
\ge
4m_\epsilon^{-1/2}
\ge
4\epsilon.
\]
Also, $\log(ek)
\le
\log(em_\epsilon)$, 
and therefore
\[
\frac{dk}{\log(ek)}
\ge
\frac{a_0}{2}
\frac{
dm_\epsilon
}{
\log(em_\epsilon)
}.
\]

Define $c_0
:=
\frac{
a_0c_{\mathrm{prob}}
}{
4
}$
and $T_0
:=
\left\lfloor
c_0
\frac{
dm_\epsilon
}{
\log(em_\epsilon)
}
\right\rfloor.$ 
Then
\[
T_0
\le
c_{\mathrm{prob}}
\frac{dk}{\log(ek)}.
\]

Hence, for every $\mathsf d
\in
\mathcal A_{\mathrm{det}}^{(T_0)}$,  
Theorem~\ref{thm:deterministic-hard-lb} gives
\[
\mathbb P_\Xi
\left(
\Err_{T_0}(
\mathsf d,
f_\Xi)
\ge
\frac{c_{\mathrm{err}}}{\sqrt k}
\right)
\ge
\frac34.
\]
Since $\frac{c_{\mathrm{err}}}{\sqrt k}
>
\epsilon$, 
we obtain
\[
\mathbb P_\Xi
\left(
\Err_{T_0}(
\mathsf d,
f_\Xi)
>
\epsilon
\right)
\ge
\frac34.
\]
Equivalently, because $\mathfrak D_{d,k}
=
\Law(f_\Xi)$, 
\[
\mathbb P_{f\sim\mathfrak D_{d,k}}
\left(
\Err_{T_0}(
\mathsf d,
f)
>
\epsilon
\right)
\ge
\frac34
\]
for every
\(\mathsf d\in
\mathcal A_{\mathrm{det}}^{(T_0)}\). 
Applying
Proposition~\ref{prop:yao-constant-success}
with $\mathcal F
=
\mathcal F_2$, $\mathfrak D
=
\mathfrak D_{d,k}$, and $p
=
\frac34$, 
yields
\[
T_\epsilon
\left(
\mathcal A_{\mathrm{rand}},
\mathcal F_2
\right)
>
T_0.
\]

By increasing \(d_0\) and decreasing \(\epsilon_0\) once more, we may
ensure that $c_0
\frac{
dm_\epsilon
}{
\log(em_\epsilon)
}
\ge
2$. 
Consequently,
\[
T_0
\ge
\frac{c_0}{2}
\frac{
dm_\epsilon
}{
\log(em_\epsilon)
}.
\]
After renaming the universal constant and recalling that $L_0=2$, 
we obtain
\[
T_\epsilon
\left(
\mathcal A_{\mathrm{rand}},
\mathcal F_{L_0}
\right)
\ge
c
\frac{
d\min\{d,\epsilon^{-2}\}
}{
\log\left(
e\min\{d,\epsilon^{-2}\}
\right)
}.
\]
\end{proof}

\subsection{Adaptive Posterior Energy under Augmented Information}
\label{sec:adaptive-energy}

This subsection proves the augmented posterior mean energy bound in
Proposition~\ref{prop:augmented-energy}, which is the information estimate
used in the distributional lower bound proof.  We also record its
immediate consequence for the value history actually observed by the
algorithm.

Throughout this subsection, fix a deterministic algorithm
\[
\mathsf d
=
\bigl(
\mathbf x_1,
\pi_2,\ldots,\pi_T,
\pi_{\mathrm{out}}
\bigr)
\in
\mathcal A_{\mathrm{det}}^{(T)}.
\]
The randomness below comes only from the function seed \(\Xi\). 
When the fixed deterministic algorithm \(\mathsf d\) is run on the random
objective \(f_\Xi\), define its random query sequence recursively by $\rvec{x}_1
:=
\mathbf x_1$  
and, for \(t=2,\ldots,T\), $\rvec{x}_t
:=
\pi_t\left(
\rsc{Y}^{\mathrm{sc}}_1,
\ldots,
\rsc{Y}^{\mathrm{sc}}_{t-1}
\right)$,  
where $\rsc{Y}^{\mathrm{sc}}_s
:=
f_\Xi(\rvec{x}_s)$.  
Define its random output by
\[
\rvec{x}^{\mathrm{out}}_T
:=
\pi_{\mathrm{out}}\left(
\rsc{Y}^{\mathrm{sc}}_1,
\ldots,
\rsc{Y}^{\mathrm{sc}}_T
\right)
=
\widehat{\mathbf x}_T(
\mathsf d,
f_\Xi).
\]

We work primarily with the unscaled random blocks $\rvec{b}_1,\ldots,\rvec{b}_k$  
and their aggregate $\rvec{s}
:=
\sum_{i=1}^k
\rvec{b}_i$.  
Recall that $\rvec{s}^{\mathrm{sc}}
=
\frac{\rvec{s}}{\sqrt d}
=
\sum_{i=1}^k
\rvec{a}_i$.  

Define the unscaled random support function
\[
F_\Xi(\mathbf x)
:=
\max_{1\le i\le k}
\left\langle
\rvec{b}_i,
\mathbf x
\right\rangle,
\qquad
\mathbf x\in B_2^d.
\]
Then $F_\Xi(\mathbf x)
=
\sqrt d\,f_\Xi(\mathbf x)$.  

Define the deterministic smallest-index selector
\begin{equation}
\iota(\mathbf v)
:=
\min
\left\{
i\in[k]:
v_i
=
\max_{1\le j\le k}v_j
\right\},
\qquad
\mathbf v\in\mathbb R^k.
\label{eq:active-index-selector}
\end{equation}

For \(t=1,\ldots,T\), define the random active index and the unscaled
observed maximum by
\[
\rsc{I}_t
:=
\iota\left(
\left(
\left\langle
\rvec{b}_i,
\rvec{x}_t
\right\rangle
\right)_{i=1}^k
\right), \quad\text{and} \quad  \rsc{Y}_t
:=
F_\Xi(\rvec{x}_t)
=
\max_{1\le i\le k}
\left\langle
\rvec{b}_i,
\rvec{x}_t
\right\rangle.
\]
Thus $\rsc{Y}_t
=
\left\langle
\rvec{b}_{\rsc{I}_t},
\rvec{x}_t
\right\rangle$.  
The actual oracle returns the scaled value $\rsc{Y}^{\mathrm{sc}}_t
:=
f_\Xi(\rvec{x}_t)
=
\frac{\rsc{Y}_t}{\sqrt d}$. 
Since multiplication by \(\sqrt d\) is a deterministic Borel bijection,
the scaled and unscaled value histories generate the same sigma-field.
We therefore define
\begin{equation}
\mathcal F_t
:=
\sigma\left(
\rsc{Y}^{\mathrm{sc}}_1,
\ldots,
\rsc{Y}^{\mathrm{sc}}_t
\right)
=
\sigma\left(
\rsc{Y}_1,
\ldots,
\rsc{Y}_t
\right),
\qquad
t=1,\ldots,T,
\label{eq:value-filtration-equivalence}
\end{equation}
and let \(\mathcal F_0\) be the trivial sigma-field.

By the deterministic algorithmic recursion, $\rvec{x}_{t+1}$ 
is \(\mathcal F_t\)-measurable for \(t<T\), and 
$\rvec{x}^{\mathrm{out}}_T$ 
is \(\mathcal F_T\)-measurable. The selector \(\iota:\mathbb R^k\to[k]\) is Borel measurable, since
\[
\{\mathbf v:\iota(\mathbf v)=i\}
=
\bigcap_{j<i}\{v_j<v_i\}
\cap
\bigcap_{j>i}\{v_j\le v_i\}.
\]
Consequently, by induction, every
\(\rsc I_t,\rsc Y_t,\rvec x_t\), and \(\mathsf T_t\)
is Borel measurable.

For \(t\ge1\), define the augmented transcript space
\[
\mathsf X_t
:=
\bigl(
[k]\times\mathbb R\times B_2^d
\bigr)^t,
\qquad
\mathsf X_0
:=
\{\ast\},
\]
equipped with its product Borel sigma-field, where \(\ast\) denotes the
empty transcript.

Define the random augmented transcript by
\[
\mathsf T_0
:=
\ast,
\qquad
\mathsf T_t
:=
\bigl(
(\rsc{I}_s,\rsc{Y}_s,\rvec{x}_s)
\bigr)_{s=1}^t
\in
\mathsf X_t,
\qquad
t\ge1,
\]
and define the augmented filtration by
\[
\mathcal G_t
:=
\sigma(\mathsf T_t),
\qquad
t=0,\ldots,T.
\]

We write $\tau_t
=
\bigl(
(i_s,y_s,\mathbf z_s)
\bigr)_{s=1}^t
\in
\mathsf X_t$ 
for a deterministic transcript value.  Here
\(\mathbf z_s\) is the query-coordinate of a generic deterministic
transcript, whereas \(\rvec{x}_s\) is the random query produced when the
fixed deterministic algorithm is run against the random hard objective. 
Since the augmented transcript contains the complete value history, $\mathcal F_t
\subseteq
\mathcal G_t$.  

Define the unscaled and scaled augmented posterior means by
\[
\rvec{m}_t
:=
\mathbb E_\Xi
\left[
\rvec{s}
\mid
\mathcal G_t
\right],\quad \text{and} \quad \rvec{m}^{\mathrm{sc}}_t
:=
\mathbb E_\Xi
\left[
\rvec{s}^{\mathrm{sc}}
\mid
\mathcal G_t
\right]
=
\frac{\rvec{m}_t}{\sqrt d}.
\]
The value only posterior is difficult to analyze directly.  Conditioned
on an exact maximum value without revealing its active block produces a
mixture over possible winners and generally does not preserve a product
structure over the hidden blocks.  Revealing the active index
\(\rsc{I}_t\) converts the observation into separate constraints on the
individual blocks and restores a blockwise posterior description.

If $(\rsc{I}_t,\rsc{Y}_t)
=
(i,y)$,  
then the winner block satisfies the affine equality $\left\langle
\rvec{b}_i,
\rvec{x}_t
\right\rangle
=
y$,  
whereas each loser block satisfies a halfspace constraint.  After strict
lower index loser constraints are replaced by their closed versions on
null boundaries, each posterior block is a Gaussian potential restricted
to a closed convex subset of an affine support.  Hence it remains
\(1\)-strongly log-concave relative to that affine support. 
There is also a measure-theoretic issue: the winner equality usually
defines a probability zero event.  The posterior must therefore be
constructed through jointly measurable regular conditional kernels,
rather than through elementary conditioned on a positive probability
event. 
The preceding discussion gives the geometric form expected of the
posterior.  To make this description rigorous, we first construct
conditional kernels that are jointly measurable in the past transcript
and the newly observed value.  We then identify these kernels with affine
slice and halfspace truncation measures and use them recursively to
construct the full posterior.
\begin{lemma}[Parameterized disintegration]
\label{lem:parameterized-disintegration}
Let
\(\Theta\), \(\mathsf X\), and \(\mathsf Y\) be standard Borel spaces.  Let $K(\theta,\mathrm d x)$ 
be a probability kernel from
\(\Theta\) to \(\mathsf X\), and let $T:
\Theta\times\mathsf X
\to
\mathsf Y$ 
be Borel measurable. Define the image kernel $K_T(\theta,A)
:=
\int_{\mathsf X}
\mathbf 1_A(T(\theta, x))
K(\theta,\mathrm d x)$ for $A\in\mathcal B(\mathsf Y)$. 
Then there exists a probability kernel $Q(\theta,y,\mathrm d x)$ 
from \(\Theta\times\mathsf Y\) to \(\mathsf X\) such that, for every nonnegative $\mathcal B(\Theta)
\otimes
\mathcal B(\mathsf Y)
\otimes
\mathcal B(\mathsf X)$-measurable  
function $h:
\Theta\times\mathsf Y\times\mathsf X
\to
[0,\infty]$, we have
\[
\int_{\mathsf X}
h\bigl(\theta,T(\theta,x),x\bigr)
K(\theta,\mathrm{d} x)
=
\int_{\mathsf Y}
\int_{\mathsf X}
h(\theta,y,x)
Q(\theta,y,\mathrm{d} x)
K_T(\theta,\mathrm{d} y).
\]
In particular, for every \(\theta\),
\(Q(\theta,\cdot,\cdot)\) is a regular conditional distribution of
\(K(\theta,\cdot)\) given \(T(\theta,\cdot)\), and $Q\left(
\theta,y,
\{x:T(\theta,x)=y\}
\right)
=
1$ 
for \(K_T(\theta,\cdot)\)-almost every \(y\).
The kernel \(Q(\theta,y,\cdot)\) is unique only
\(K_T(\theta,\cdot)\)-almost everywhere in \(y\), for each fixed
\(\theta\).
\end{lemma}

\begin{proof}[Proof of Lemma~\ref{lem:parameterized-disintegration}]
Define a probability kernel
\(\rho\) from \(\Theta\) to
\(\mathsf Y\times\mathsf X\) by
\[
\rho(\theta,\mathrm{d} y,\mathrm{d} x)
:=
\delta_{T(\theta,x)}(\mathrm{d} y)
K(\theta,\mathrm{d} x).
\]
Its \(\mathsf Y\)-marginal is \(K_T\).
By the disintegration-of-kernels theorem
\citep[Theorem~1.25]{kallenberg2017random}, there exists a probability
kernel \(Q\) from
\(\Theta\times\mathsf Y\) to \(\mathsf X\) such that
\[
\rho(\theta,\mathrm{d} y,\mathrm{d} x)
=
K_T(\theta,\mathrm{d} y)Q(\theta,y,\mathrm{d} x).
\]
Integrating a nonnegative Borel test function gives the stated identity.

To prove the fiber-support assertion, apply the preceding identity with
\[
h(\theta,y,x)
:=
\mathbf 1_{\{y=T(\theta,x)\}}.
\]
Since $h\bigl(
\theta,T(\theta,x),x
\bigr)
=
1$,  
we obtain
\[
1
=
\int_{\mathsf Y}
Q\left(
\theta,y,
\{x:T(\theta,x)=y\}
\right)
K_T(\theta,\mathrm d y).
\]
The integrand takes values in \([0,1]\).  Hence it must equal \(1\)
for \(K_T(\theta,\cdot)\)-almost every \(y\).
\end{proof}
This lemma provides a jointly measurable abstract conditional kernel.
To identify its affine-slice geometry uniformly in the transcript
parameter, we use the parameterized affine-slice result below.

\begin{lemma}[Borel orthonormal frames]
\label{lem:borel-orthonormal-frame}
Let \(\Theta\) be a standard Borel space, and let $\mathbf P:
\Theta
\to
\mathbb R^{d\times d}$ 
be Borel measurable.  Suppose that, for every
\(\theta\in\Theta\), $\mathbf P(\theta)^\top
=
\mathbf P(\theta)$, $\mathbf P(\theta)^2
=
\mathbf P(\theta)$,  
and that $\operatorname{rank}\mathbf P(\theta)
=
r$ 
is constant on \(\Theta\).  Then there exists a Borel map $\mathbf U:
\Theta
\to
\mathbb R^{d\times r}$  
such that $\mathbf U(\theta)^\top
\mathbf U(\theta)
=
\mathbf I_r$  and $\mathbf U(\theta)
\mathbf U(\theta)^\top
=
\mathbf P(\theta)$ 
for every \(\theta\in\Theta\).
For \(r=0\), \(\mathbf U(\theta)\) is the unique
\(d\times0\) matrix.
\end{lemma}

\begin{proof}[Proof of Lemma~\ref{lem:borel-orthonormal-frame}] 
For every ordered \(r\)-tuple $J=(j_1,\ldots,j_r)$ 
of distinct indices in \([d]\), define
\[
D_J(\theta)
:=
\det
\left(
\left[
\left\langle
\mathbf P(\theta)\mathbf e_{j_a},
\mathbf P(\theta)\mathbf e_{j_b}
\right\rangle
\right]_{a,b=1}^r
\right).
\]
Each \(D_J\) is Borel measurable.  Since
\(\operatorname{rank}\mathbf P(\theta)=r\), for every \(\theta\)
there is at least one \(J\) with $D_J(\theta)>0$. 

Order the finitely many tuples \(J\) lexicographically and let
\(\Theta_J\) be the Borel set on which \(J\) is the first tuple with
\(D_J(\theta)>0\).  The sets \(\Theta_J\) form a finite Borel
partition of \(\Theta\). 
On \(\Theta_J\), apply the Gram--Schmidt procedure to $\mathbf P(\theta)\mathbf e_{j_1},
\ldots,
\mathbf P(\theta)\mathbf e_{j_r}$.  
All denominators in this procedure are strictly positive on
\(\Theta_J\), so the resulting orthonormal vectors are Borel functions
of \(\theta\).  Their columns form the required matrix
\(\mathbf U(\theta)\).  Patching these definitions over the finite
Borel partition gives the result.
\end{proof}

\begin{lemma}[Jointly measurable affine-slice disintegration]
\label{lem:linear-disintegration}
Let \(\Theta\) be a standard Borel space, and let
\(m\ge0\).  Let $\mathbf M:
\Theta\to\mathbb R^{m\times d}$, $\mathbf c:
\Theta\to\mathbb R^m$, and $\mathbf x:
\Theta\to\mathbb R^d$ 
be Borel maps. 
Let $\mathcal C
\subseteq
\Theta\times\mathbb R^d$ 
be Borel, and write $C_\theta
:=
\left\{
\mathbf b\in\mathbb R^d:
(\theta,\mathbf b)\in\mathcal C
\right\}$.  
Let $V:
\Theta\times\mathbb R^d
\to
\mathbb R\cup\{+\infty\}$ 
be Borel measurable. 
For \(\theta\in\Theta\), define $H_\theta
:=
\left\{
\mathbf b\in\mathbb R^d:
\mathbf M(\theta)\mathbf b
=
\mathbf c(\theta)
\right\}$. 
Let $\mu(\theta,\mathrm d\mathbf b)$ 
be a probability kernel from
\(\Theta\) to \(\mathbb R^d\). 
Suppose that there is a Borel set $\Theta_0\subseteq\Theta$ 
such that, for every \(\theta\in\Theta_0\),

\begin{enumerate}[label=(\alph*)]

\item \(H_\theta\) is nonempty;
\item \(C_\theta\subseteq H_\theta\); 

\item
\begin{equation}
\mu(\theta,\mathrm d\mathbf b)
=
Z(\theta)^{-1}
e^{-V(\theta,\mathbf b)}
\mathbf 1_{C_\theta}(\mathbf b)
\,
\sigma_{H_\theta}(\mathrm d\mathbf b), \qquad0<Z(\theta)<\infty.
\label{eq:parameterized-affine-prior} 
\end{equation}
\end{enumerate}
Define the image kernel
\[
\lambda(\theta,A)
:=
\int_{\mathbb R^d}
\mathbf 1_A
\left(
\left\langle
\mathbf b,\mathbf x(\theta)
\right\rangle
\right)
\mu(\theta,\mathrm d\mathbf b),
\qquad
A\in\mathcal B(\mathbb R).
\]
Then there exist $\widetilde Q:
\Theta\times\mathbb R
\times
\mathcal B(\mathbb R^d)
\to
[0,1]$,  
a Borel map $J:
\Theta\times\mathbb R
\to
[0,+\infty],$
and a Borel set $\mathsf G
\subseteq
\Theta_0\times\mathbb R$
such that:
\begin{enumerate}[label=(\roman*)]

\item
\(\widetilde Q\) is a probability kernel from
\(\Theta\times\mathbb R\)
to \(\mathbb R^d\);

\item
for every \(\theta\in\Theta_0\),
\(\widetilde Q(\theta,\cdot,\cdot)\)
is a regular conditional distribution of
\(\mu(\theta,\cdot)\)
given the map $\mathbf b
\longmapsto
\left\langle
\mathbf b,\mathbf x(\theta)
\right\rangle;$

\item
for every \(\theta\in\Theta_0\), $\lambda
\left(
\theta,
\mathsf G_\theta
\right)
=
1$, where  $\mathsf G_\theta
:=
\left\{
y\in\mathbb R:
(\theta,y)\in\mathsf G
\right\};$ 

\item
for every
\((\theta,y)\in\mathsf G\),
the affine fiber $H_{\theta,y}
:=
H_\theta
\cap
\left\{
\mathbf b:
\left\langle
\mathbf b,\mathbf x(\theta)
\right\rangle
=
y
\right\}$
is nonempty, $0<J(\theta,y)<\infty$,  
and
\begin{equation}
\widetilde Q
\left(
\theta,y,\mathrm d\mathbf b
\right)
=
J(\theta,y)^{-1}
e^{-V(\theta,\mathbf b)}
\mathbf 1_{C_\theta}(\mathbf b)
\,
\sigma_{H_{\theta,y}}(\mathrm d\mathbf b).
\label{eq:joint-affine-slice-density}
\end{equation}

\end{enumerate}
\end{lemma}

\begin{proof}[Proof of Lemma~\ref{lem:linear-disintegration}]
The Moore--Penrose pseudoinverse map $\mathbf A
\longmapsto
\mathbf A^\dagger$ 
is Borel measurable on finite dimensional matrix spaces.  Indeed, it is
continuous on every fixed rank stratum, and there are only finitely many
possible ranks. 
Define
\[
\mathbf b_0(\theta)
:=
\mathbf M(\theta)^\dagger
\mathbf c(\theta)\quad \text{and} \quad \mathbf P(\theta)
:=
\mathbf I_d
-
\mathbf M(\theta)^\dagger
\mathbf M(\theta).
\]
Both maps are Borel measurable. 
For every \(\theta\),
\(\mathbf P(\theta)\)
is the orthogonal projection onto $\ker\mathbf M(\theta)$.  
Fix
\(\theta\in\Theta_0\).
Since \(H_\theta\neq\varnothing\),
we have $\mathbf c(\theta)
\in
\operatorname{range}\mathbf M(\theta)$.  
Therefore $\mathbf M(\theta)
\mathbf M(\theta)^\dagger
\mathbf c(\theta)
=
\mathbf c(\theta)$ 
and hence $\mathbf M(\theta)\mathbf b_0(\theta)
=
\mathbf c(\theta)$.  
Consequently,
\begin{equation}
H_\theta
=
\mathbf b_0(\theta)
+
\ker\mathbf M(\theta).
\label{eq:affine-support-b0-kernel}
\end{equation} 
For $r=0,\ldots,\min\{m,d\}$,  
define the good rank-\(r\) stratum
\begin{equation}
\mathsf A_r
:=
\left\{
\theta\in\Theta_0:
\operatorname{rank}\mathbf M(\theta)=r
\right\}.
\label{eq:good-rank-stratum}
\end{equation}
Each \(\mathsf A_r\) is Borel, and $\Theta_0
=
\bigsqcup_{r=0}^{\min\{m,d\}}
\mathsf A_r$.   
On \(\mathsf A_r\), $\operatorname{rank}\mathbf P(\theta)
=
d-r$.  
Applying
Lemma~\ref{lem:borel-orthonormal-frame}
on the standard Borel space
\(\mathsf A_r\),
there exists a Borel matrix $\mathbf U_r(\theta)
\in
\mathbb R^{d\times(d-r)}$ 
such that
\[
\mathbf U_r(\theta)^\top
\mathbf U_r(\theta)
=
\mathbf I_{d-r}\quad \text{and} \quad \mathbf U_r(\theta)
\mathbf U_r(\theta)^\top
=
\mathbf P(\theta).\]
Thus the columns of
\(\mathbf U_r(\theta)\)
form an orthonormal basis of
\(\ker\mathbf M(\theta)\). 
For
\(\theta\in\mathsf A_r\),
define
\begin{equation}
\begin{aligned}
\widehat Z_r(\theta)
:=
\int_{\mathbb R^{d-r}}
\exp\left(
-
V\left(
\theta,
\mathbf b_0(\theta)
+
\mathbf U_r(\theta)\mathbf z
\right)
\right)\cdot
\mathbf 1_{\mathcal C}
\left(
\theta,
\mathbf b_0(\theta)
+
\mathbf U_r(\theta)\mathbf z
\right)
\,
\mathrm d\mathbf z.
\end{aligned}
\label{eq:measurable-affine-normalizer}
\end{equation}
When \(d-r=0\), integration over
\(\mathbb R^0\)
means evaluation at its unique point. 
The integrand in
\eqref{eq:measurable-affine-normalizer}
is jointly Borel measurable in
\((\theta,\mathbf z)\).
Hence $\theta
\longmapsto
\widehat Z_r(\theta)$ 
is Borel measurable on
\(\mathsf A_r\). 
The map
\[
\mathbf z
\longmapsto
\mathbf b_0(\theta)
+
\mathbf U_r(\theta)\mathbf z
\]
is an affine isometry from
\(\mathbb R^{d-r}\)
onto \(H_\theta\).  Therefore
\[
\widehat Z_r(\theta)
=
\int_{H_\theta}
e^{-V(\theta,\mathbf b)}
\mathbf 1_{C_\theta}(\mathbf b)
\,
\sigma_{H_\theta}(\mathrm d\mathbf b).
\]
Integrating
\eqref{eq:parameterized-affine-prior}
over \(\mathbb R^d\) gives $1
=
Z(\theta)^{-1}
\widehat Z_r(\theta)$,  
and consequently
\begin{equation}
\widehat Z_r(\theta)
=
Z(\theta)
\in
(0,\infty),
\qquad
\theta\in\mathsf A_r.
\label{eq:measurable-normalizer-equals-Z}
\end{equation}
Define the Borel maps
\[
\mathbf q(\theta)
:=
\mathbf P(\theta)\mathbf x(\theta) \quad \text{and} \quad y_0(\theta)
:=
\left\langle
\mathbf b_0(\theta),
\mathbf x(\theta)
\right\rangle. 
\]
For
\(\theta\in\mathsf A_r\)
and
\(\boldsymbol\ell\in\ker\mathbf M(\theta)\),
we have
\[
\left\langle
\boldsymbol\ell,
\mathbf x(\theta)
\right\rangle
=
\left\langle
\boldsymbol\ell,
\mathbf P(\theta)\mathbf x(\theta)
\right\rangle
=
\left\langle
\boldsymbol\ell,
\mathbf q(\theta)
\right\rangle.
\]
Thus the functional $\mathbf b
\longmapsto
\left\langle
\mathbf b,\mathbf x(\theta)
\right\rangle$ 
is nonconstant on \(H_\theta\) exactly when $\mathbf q(\theta)\neq\mathbf 0_d$.  
For every \(r\), split
\(\mathsf A_r\)
into the two Borel sets
\[
\mathsf A_r^{\mathrm{var}}
:=
\left\{
\theta\in\mathsf A_r:
\|\mathbf q(\theta)\|_2>0
\right\}\quad \text{and} \quad \mathsf A_r^{\mathrm{const}}
:=
\left\{
\theta\in\mathsf A_r:
\mathbf q(\theta)=\mathbf 0_d
\right\}. 
\]
\paragraph{The nonconstant functional branch.}

Fix \(r\), and consider
\(\mathsf A_r^{\mathrm{var}}\).
For
\(\theta\in\mathsf A_r^{\mathrm{var}}\),
define
\[
\mathbf e(\theta)
:=
\frac{
\mathbf q(\theta)
}{
\|\mathbf q(\theta)\|_2
}\quad \text{and} \quad \mathbf P_0(\theta)
:=
\mathbf P(\theta)
-
\mathbf e(\theta)\mathbf e(\theta)^\top.
\]
Since $\mathbf e(\theta)
\in
\ker\mathbf M(\theta)$,  
the matrix
\(\mathbf P_0(\theta)\)
is the orthogonal projection onto
\[
L_0(\theta)
:=
\ker\mathbf M(\theta)
\cap
\mathbf q(\theta)^\perp.
\]
Moreover, $\operatorname{rank}\mathbf P_0(\theta)
=
d-r-1$.
Applying
Lemma~\ref{lem:borel-orthonormal-frame}
on
\(\mathsf A_r^{\mathrm{var}}\),
there exists a Borel matrix $\mathbf W_r(\theta)
\in
\mathbb R^{d\times(d-r-1)}$ 
whose columns form an orthonormal basis of
\(L_0(\theta)\). 
For $(\theta,y)
\in
\mathsf A_r^{\mathrm{var}}
\times
\mathbb R$,  
define
\begin{equation}
\mathbf h_r(\theta,y)
:=
\mathbf b_0(\theta)
+
\frac{
y-y_0(\theta)
}{
\|\mathbf q(\theta)\|_2^2
}
\mathbf q(\theta).
\label{eq:affine-fiber-basepoint}
\end{equation}
Since $\mathbf q(\theta)
\in
\ker\mathbf M(\theta)$,  
we have $\mathbf h_r(\theta,y)
\in
H_\theta$.  
Furthermore,
\[
\left\langle
\mathbf q(\theta),
\mathbf x(\theta)
\right\rangle
=
\left\langle
\mathbf P(\theta)\mathbf x(\theta),
\mathbf x(\theta)
\right\rangle
=
\|\mathbf q(\theta)\|_2^2.
\]
Hence $\left\langle
\mathbf h_r(\theta,y),
\mathbf x(\theta)
\right\rangle
=
y$.  
It follows that
\begin{equation}
H_{\theta,y}
=
\mathbf h_r(\theta,y)
+
L_0(\theta).
\label{eq:affine-fiber-parametrization}
\end{equation}
For a Borel set
\(E\subseteq\mathbb R^d\),
define
\begin{equation}
\begin{aligned}
N_{r,E}(\theta,y)
:=
\int_{\mathbb R^{d-r-1}}
&
\mathbf 1_E
\left(
\mathbf h_r(\theta,y)
+
\mathbf W_r(\theta)\mathbf z
\right)
\\
&\cdot
\mathbf 1_{\mathcal C}
\left(
\theta,
\mathbf h_r(\theta,y)
+
\mathbf W_r(\theta)\mathbf z
\right)
\\
&\cdot
\exp\left(
-
V\left(
\theta,
\mathbf h_r(\theta,y)
+
\mathbf W_r(\theta)\mathbf z
\right)
\right)
\,
\mathrm d\mathbf z.
\end{aligned}
\label{eq:parameterized-fiber-integral}
\end{equation}
When
\(d-r-1=0\),
the integral over
\(\mathbb R^0\)
means evaluation at its unique point.
The integrand is jointly Borel measurable in
\((\theta,y,\mathbf z)\).
Consequently, $(\theta,y)
\longmapsto
N_{r,E}(\theta,y)$ 
is Borel measurable on
\(\mathsf A_r^{\mathrm{var}}\times\mathbb R\). 
Since $\mathbf z
\longmapsto
\mathbf h_r(\theta,y)
+
\mathbf W_r(\theta)\mathbf z$ 
is an affine isometry from
\(\mathbb R^{d-r-1}\)
onto \(H_{\theta,y}\),
we also have
\begin{equation}
\begin{aligned}
N_{r,E}(\theta,y)
=
\int_{H_{\theta,y}}
&
\mathbf 1_E(\mathbf b)
e^{-V(\theta,\mathbf b)}
\mathbf 1_{C_\theta}(\mathbf b)
\,
\sigma_{H_{\theta,y}}(\mathrm d\mathbf b).
\end{aligned}
\label{eq:fiber-integral-intrinsic-form}
\end{equation}
Set $J_r^{\mathrm{var}}(\theta,y)
:=
N_{r,\mathbb R^d}(\theta,y)$. 
Define
\[
\widetilde Q_r^{\mathrm{var}}
(\theta,y,E)
:=
\begin{cases}
\dfrac{
N_{r,E}(\theta,y)
}{
J_r^{\mathrm{var}}(\theta,y)
},
&
0<
J_r^{\mathrm{var}}(\theta,y)
<
\infty,
\\[3mm]
\delta_{\mathbf h_r(\theta,y)}(E),
&
\text{otherwise}.
\end{cases}
\]
Then $\widetilde Q_r^{\mathrm{var}}$ 
is a probability kernel from
\(\mathsf A_r^{\mathrm{var}}\times\mathbb R\)
to \(\mathbb R^d\).

We next derive the disintegration identity.  The columns of $\left[
\mathbf e(\theta)
\ \ 
\mathbf W_r(\theta)
\right]$
form an orthonormal basis of
\(\ker\mathbf M(\theta)\).
Thus every
\(\mathbf b\in H_\theta\)
has a unique representation
\[
\mathbf b
=
\mathbf b_0(\theta)
+
s\mathbf e(\theta)
+
\mathbf W_r(\theta)\mathbf z.
\]
Moreover, $\left\langle
\mathbf e(\theta),
\mathbf x(\theta)
\right\rangle
=
\|\mathbf q(\theta)\|_2$.
Therefore $\left\langle
\mathbf b,\mathbf x(\theta)
\right\rangle
=
y_0(\theta)
+
s\|\mathbf q(\theta)\|_2$. 
Changing variables $y
=
y_0(\theta)
+
s\|\mathbf q(\theta)\|_2$ 
gives, for every nonnegative Borel function \(f\),
\begin{equation}
\begin{aligned}
\int_{H_\theta}
f(\mathbf b)
\,\sigma_{H_\theta}(\mathrm d\mathbf b)
=
\int_{\mathbb R}
\frac1{
\|\mathbf q(\theta)\|_2
}
\int_{H_{\theta,y}}
f(\mathbf b)
\,\sigma_{H_{\theta,y}}(\mathrm d\mathbf b)
\,
\mathrm dy.
\end{aligned}
\label{eq:affine-coarea-explicit}
\end{equation} 
Applying
\eqref{eq:affine-coarea-explicit}
to the density in
\eqref{eq:parameterized-affine-prior}
shows that, for
\(\theta\in\mathsf A_r^{\mathrm{var}}\),
\begin{equation}
\lambda(\theta,\mathrm dy)
=
\frac{
J_r^{\mathrm{var}}(\theta,y)
}{
\widehat Z_r(\theta)
\|\mathbf q(\theta)\|_2
}
\,\mathrm dy.
\label{eq:image-density-nonconstant}
\end{equation}
In particular,
\[
\int_{\mathbb R}
J_r^{\mathrm{var}}(\theta,y)
\,\mathrm dy
=
\widehat Z_r(\theta)
\|\mathbf q(\theta)\|_2
<
\infty.
\]
Hence $0<
J_r^{\mathrm{var}}(\theta,y)
<
\infty$
for
\(\lambda(\theta,\cdot)\)-almost every \(y\). 
Let
\(A\in\mathcal B(\mathbb R)\)
and
\(E\in\mathcal B(\mathbb R^d)\).
Using
\eqref{eq:image-density-nonconstant},
the definition of
\(\widetilde Q_r^{\mathrm{var}}\),
and
\eqref{eq:affine-coarea-explicit},
we obtain
\begin{equation}\label{eq:nonconstant-rcd-identity}
\begin{aligned}
\int_A
\widetilde Q_r^{\mathrm{var}}
(\theta,y,E)
\,
\lambda(\theta,\mathrm dy)
&=
\frac1{
\widehat Z_r(\theta)
\|\mathbf q(\theta)\|_2
}
\int_A
N_{r,E}(\theta,y)
\,\mathrm dy\\
&=
\int_{\mathbb R^d}
\mathbf 1_E(\mathbf b)
\mathbf 1_A
\left(
\left\langle
\mathbf b,\mathbf x(\theta)
\right\rangle
\right)
\mu(\theta,\mathrm d\mathbf b).
\end{aligned}
\end{equation}
Thus
\(\widetilde Q_r^{\mathrm{var}}\)
is a regular conditional distribution on the nonconstant branch.

\paragraph{The constant functional branch.}

Fix \(r\), and consider
\(\mathsf A_r^{\mathrm{const}}\). 
Let
\(\theta\in\mathsf A_r^{\mathrm{const}}\).
For every
\(\mathbf b\in H_\theta\),
we may write
\[
\mathbf b
=
\mathbf b_0(\theta)
+
\boldsymbol\ell,
\qquad
\boldsymbol\ell
\in
\ker\mathbf M(\theta).
\]
Since $\mathbf P(\theta)\mathbf x(\theta)
=
\mathbf q(\theta)
=
\mathbf 0_d$,  
we have $\left\langle
\boldsymbol\ell,
\mathbf x(\theta)
\right\rangle
=
\left\langle
\boldsymbol\ell,
\mathbf P(\theta)\mathbf x(\theta)
\right\rangle
=
0$. 
Consequently,
\begin{equation}
\left\langle
\mathbf b,\mathbf x(\theta)
\right\rangle
=
y_0(\theta)
\qquad
\text{for every }\mathbf b\in H_\theta.
\label{eq:constant-functional-on-H}
\end{equation}
Because $C_\theta
\subseteq
H_\theta$  
and
\(\mu(\theta,\cdot)\)
has density
\eqref{eq:parameterized-affine-prior},
we have $\mu(\theta,H_\theta)=1$.  
Combining this with
\eqref{eq:constant-functional-on-H}
gives
\begin{equation}
\lambda(\theta,\cdot)
=
\delta_{y_0(\theta)}.
\label{eq:constant-image-delta}
\end{equation}
Define
\[
\widetilde Q_r^{\mathrm{const}}
(\theta,y,E)
:=
\begin{cases}
\mu(\theta,E),
&
y=y_0(\theta),
\\[1mm]
\delta_{\mathbf b_0(\theta)}(E),
&
y\neq y_0(\theta).
\end{cases}
\]
This is a probability kernel from
\(\mathsf A_r^{\mathrm{const}}\times\mathbb R\)
to \(\mathbb R^d\).
Indeed, the graph $\left\{
(\theta,y):
y=y_0(\theta)
\right\}$ 
is Borel. 
For
\(A\in\mathcal B(\mathbb R)\)
and
\(E\in\mathcal B(\mathbb R^d)\),
equation
\eqref{eq:constant-image-delta}
gives
\[
\begin{aligned}
\int_A
\widetilde Q_r^{\mathrm{const}}
(\theta,y,E)
\,
\lambda(\theta,\mathrm dy)
&=
\mathbf 1_A
\left(
y_0(\theta)
\right)
\mu(\theta,E)
\\
&=
\int_{\mathbb R^d}
\mathbf 1_E(\mathbf b)
\mathbf 1_A
\left(
\left\langle
\mathbf b,\mathbf x(\theta)
\right\rangle
\right)
\mu(\theta,\mathrm d\mathbf b).
\end{aligned}
\]
Thus
\(\widetilde Q_r^{\mathrm{const}}\)
is a regular conditional distribution on the constant branch.

\paragraph{Definition of the global kernel and normalizer.}

Define
\[
\widetilde Q
\left(
\theta,y,E
\right)
:=
\begin{cases}
\widetilde Q_r^{\mathrm{var}}
(\theta,y,E),
&
\theta\in\mathsf A_r^{\mathrm{var}}
\text{ for some }r,
\\[1mm]
\widetilde Q_r^{\mathrm{const}}
(\theta,y,E),
&
\theta\in\mathsf A_r^{\mathrm{const}}
\text{ for some }r,
\\[1mm]
\delta_{\mathbf 0_d}(E),
&
\theta\notin\Theta_0.
\end{cases}
\label{eq:global-affine-slice-kernel}
\]
The finitely many strata in this definition are Borel and pairwise
disjoint.  Hence
\(\widetilde Q\)
is a probability kernel from
\(\Theta\times\mathbb R\)
to \(\mathbb R^d\). 
Define
\[
J(\theta,y)
:=
\begin{cases}
J_r^{\mathrm{var}}(\theta,y),
&
\theta\in\mathsf A_r^{\mathrm{var}}
\text{ for some }r,
\\[1mm]
\widehat Z_r(\theta),
&
\theta\in\mathsf A_r^{\mathrm{const}}
\text{ for some }r
\text{ and }
y=y_0(\theta),
\\[1mm]
0,
&
\text{otherwise}.
\end{cases}
\label{eq:global-affine-slice-normalizer}
\]
The map \(J\) is Borel measurable.

Equations
\eqref{eq:nonconstant-rcd-identity}
and
\eqref{eq:constant-image-delta}
show that, for every
\(\theta\in\Theta_0\),
\(\widetilde Q(\theta,\cdot,\cdot)\)
is a regular conditional distribution of
\(\mu(\theta,\cdot)\)
given
\[
\mathbf b
\longmapsto
\left\langle
\mathbf b,\mathbf x(\theta)
\right\rangle.
\]

\paragraph{Definition of the full measure geometric set.}

Define
\[
\begin{aligned}
\mathsf G
:={}&
\bigcup_{r=0}^{\min\{m,d\}}
\left\{
(\theta,y):
\theta\in\mathsf A_r^{\mathrm{var}},
\quad
0<
J_r^{\mathrm{var}}(\theta,y)
<
\infty
\right\}
\\
&\quad\cup
\bigcup_{r=0}^{\min\{m,d\}}
\left\{
(\theta,y_0(\theta)):
\theta\in\mathsf A_r^{\mathrm{const}}
\right\}.
\end{aligned}
\label{eq:affine-slice-good-set}
\]
This is a Borel subset of
\(\Theta_0\times\mathbb R\). 
If
\(\theta\in\mathsf A_r^{\mathrm{var}}\),
equation
\eqref{eq:image-density-nonconstant}
shows that
\[
\lambda
\left(
\theta,
\left\{
y:
0<
J_r^{\mathrm{var}}(\theta,y)
<
\infty
\right\}
\right)
=
1.
\]
If
\(\theta\in\mathsf A_r^{\mathrm{const}}\),
then
\[
\lambda(\theta,\cdot)
=
\delta_{y_0(\theta)}.
\]
Therefore
\[
\lambda
\left(
\theta,
\mathsf G_\theta
\right)
=
1
\qquad
\text{for every }\theta\in\Theta_0.
\]

Now fix
\((\theta,y)\in\mathsf G\).
If
\(\theta\in\mathsf A_r^{\mathrm{var}}\),
then
\eqref{eq:affine-fiber-parametrization}
shows that $H_{\theta,y}
=
\mathbf h_r(\theta,y)
+
L_0(\theta)$ 
is nonempty.  Moreover,
\eqref{eq:fiber-integral-intrinsic-form}
and the definition of
\(\widetilde Q_r^{\mathrm{var}}\)
give
\[
\widetilde Q
\left(
\theta,y,\mathrm d\mathbf b
\right)
=
J(\theta,y)^{-1}
e^{-V(\theta,\mathbf b)}
\mathbf 1_{C_\theta}(\mathbf b)
\,
\sigma_{H_{\theta,y}}(\mathrm d\mathbf b).
\]
If
\(\theta\in\mathsf A_r^{\mathrm{const}}\),
then
\(y=y_0(\theta)\),
and $H_{\theta,y_0(\theta)}
=
H_\theta.$
Furthermore, $J
\left(
\theta,y_0(\theta)
\right)
=
\widehat Z_r(\theta)
=
Z(\theta)$,  
and
\[
\begin{aligned}
\widetilde Q
\left(
\theta,y_0(\theta),
\mathrm d\mathbf b
\right)
&=
\mu(\theta,\mathrm d\mathbf b)
\\
&=
J
\left(
\theta,y_0(\theta)
\right)^{-1}
e^{-V(\theta,\mathbf b)}
\mathbf 1_{C_\theta}(\mathbf b)
\,
\sigma_{H_{\theta,y_0(\theta)}}(\mathrm d\mathbf b).
\end{aligned}
\]
Thus, for every
\((\theta,y)\in\mathsf G\),
the fiber \(H_{\theta,y}\) is nonempty, $0<J(\theta,y)<\infty$,  
and
\eqref{eq:joint-affine-slice-density}
holds.
\end{proof}

\paragraph{Query reconstruction and consistent transcript extensions.}

Define $\chi_1:
\mathsf X_0
\to
B_2^d$  
by $\chi_1(\ast)
:=
\mathbf x_1$.  
For \(t=2,\ldots,T\), define $\chi_t:
\mathsf X_{t-1}
\to
B_2^d$
by
\[
\chi_t
\left(
\bigl(
(i_s,y_s,\mathbf z_s)
\bigr)_{s=1}^{t-1}
\right)
:=
\pi_t
\left(
\frac{y_1}{\sqrt d},
\ldots,
\frac{y_{t-1}}{\sqrt d}
\right).
\]
The map \(\chi_1\) is constant and hence Borel measurable.  For
\(t\ge2\), the coordinate projection $\bigl(
(i_s,y_s,\mathbf z_s)
\bigr)_{s=1}^{t-1}
\longmapsto
(y_1,\ldots,y_{t-1})$ 
is Borel measurable, the scaling map $(y_1,\ldots,y_{t-1})
\longmapsto
\left(
\frac{y_1}{\sqrt d},
\ldots,
\frac{y_{t-1}}{\sqrt d}
\right)$ 
is continuous, and \(\pi_t\) is Borel measurable.  Therefore every
\(\chi_t\) is Borel measurable.

The map \(\chi_t\) ignores the active-index and recorded-query
coordinates because the fixed deterministic algorithm selects its next
query using only the preceding scalar oracle values.

The algorithmic recursion gives
\begin{equation}
\rvec{x}_t
=
\chi_t(\mathsf T_{t-1}),
\qquad
t=1,\ldots,T,
\quad
\mathbb P_\Xi\text{-almost surely}.
\label{eq:query-reconstruction}
\end{equation}

For $t=1,\ldots,T$,  $\tau_{t-1}\in\mathsf X_{t-1}$,  $i\in[k]$, and $y\in\mathbb R$ 
define the consistent transcript extension
\[
\tau_{t-1}\oplus_t(i,y)
:=
\bigl(
\tau_{t-1},
(i,y,\chi_t(\tau_{t-1}))
\bigr)
\in
\mathsf X_t.
\]
The map $(\tau_{t-1},i,y)
\longmapsto
\tau_{t-1}\oplus_t(i,y)$ 
is Borel measurable.  Moreover,
\begin{equation}
\mathsf T_t
=
\mathsf T_{t-1}
\oplus_t
(\rsc{I}_t,\rsc{Y}_t)
\qquad
\mathbb P_\Xi\text{-almost surely}.
\label{eq:transcript-recursion}
\end{equation} 
The preceding disintegration lemmas provide measurable conditional
kernels for affine equalities, while the maps
\(\chi_t\) and \(\oplus_t\) encode the adaptive query recursion.  We can
now construct the posterior kernels on the entire transcript space.
Their geometric identification will be asserted only
\(\mathbb P_{\mathsf T_t}\)-almost everywhere.

\begin{lemma}[Measurable product posterior]
\label{lem:regular-product-posterior}
For every $t\in\{0,1,\ldots,T\}$,  
there exist probability kernels $\mu_{j,t}:
\mathsf X_t\times\mathcal B(\mathbb R^d)
\to[0,1]$, $j\in[k]$ 
and a Borel set $\mathsf X_t^{\mathrm{post}}
\subseteq
\mathsf X_t$ 
such that $\mathbb P_{\mathsf T_t}
\left(
\mathsf X_t^{\mathrm{post}}
\right)
=1$.  
Define the product kernel
\[
K_t
\left(
\tau_t,
\mathrm d\mathbf b_1,\ldots,\mathrm d\mathbf b_k
\right)
:=
\bigotimes_{j=1}^k
\mu_{j,t}
\left(
\tau_t,
\mathrm d\mathbf b_j
\right).
\]
Then \(K_t\) is a regular conditional distribution of $(\rvec{b}_1,\ldots,\rvec{b}_k)$ 
given \(\mathsf T_t\). Equivalently, for every $E\in\mathcal B\bigl((\mathbb R^d)^k\bigr)$ and $A\in\mathcal B(\mathsf X_t)$ 
we have
\begin{equation}
\mathbb P_\Xi
\left(
(\rvec{b}_1,\ldots,\rvec{b}_k)\in E,\,
\mathsf T_t\in A
\right)
=
\int_A
K_t(\tau_t,E)
\,\mathbb P_{\mathsf T_t}(\mathrm d\tau_t).
\label{eq:posterior-kernel-disintegration}
\end{equation}
Since $\mathcal G_t
=
\sigma(\mathsf T_t)$, 
it follows that
\begin{equation}
\Law\left(
\rvec{b}_1,\ldots,\rvec{b}_k
\mid
\mathcal G_t
\right)(\omega,\cdot)
=
K_t\left(
\mathsf T_t(\omega),\cdot
\right)
\qquad
\mathbb P_\Xi\text{-almost surely}.
\label{eq:posterior-product-kernel}
\end{equation}
For every deterministic transcript value $\tau_t
=
\bigl((i_s,y_s,\mathbf z_s)\bigr)_{s=1}^t
\in
\mathsf X_t$, 
define
\[
H_{j,t}(\tau_t)
:=
\left\{
\mathbf b\in\mathbb R^d:
\langle\mathbf b,\mathbf z_s\rangle=y_s
\text{ whenever }i_s=j
\right\},
\]
and
\[
C_{j,t}(\tau_t)
:=
\left\{
\mathbf b\in H_{j,t}(\tau_t):
\|\mathbf b\|_2\le2\sqrt d,\quad
\langle\mathbf b,\mathbf z_s\rangle\le y_s
\text{ whenever }i_s\ne j
\right\}.
\]
For \(t=0\), these constraints are interpreted vacuously, so that
\[
H_{j,0}(\ast)=\mathbb R^d,
\qquad
C_{j,0}(\ast)
=
\{\mathbf b:\|\mathbf b\|_2\le2\sqrt d\}.
\] 
For every $\tau_t\in\mathsf X_t^{\mathrm{post}}$ 
and every \(j\in[k]\), the set
\(H_{j,t}(\tau_t)\) is a nonempty affine subspace,
\(C_{j,t}(\tau_t)\) is a nonempty closed convex subset of
\(H_{j,t}(\tau_t)\), and
\begin{equation}
\mu_{j,t}
\left(
\tau_t,\mathrm d\mathbf b
\right)
=
Z_{j,t}(\tau_t)^{-1}
e^{-\|\mathbf b\|_2^2/2}
\mathbf 1_{C_{j,t}(\tau_t)}(\mathbf b)
\,
\sigma_{H_{j,t}(\tau_t)}
\left(
\mathrm d\mathbf b
\right),
\label{eq:posterior-block-density}
\end{equation}
where $0<Z_{j,t}(\tau_t)<\infty$. 

Consequently, every posterior block is \(1\)-strongly log-concave
relative to its affine support and has a closed convex effective domain. 
If $\dim H_{j,t}(\tau_t)=0$,  
then the corresponding posterior block is a Dirac measure.
\end{lemma}

\begin{proof}[Proof of Lemma~\ref{lem:regular-product-posterior}]
We prove simultaneously by induction on \(t\) that:

\begin{enumerate}[label=(\roman*)]

\item there exist probability kernels
\[
\mu_{j,t}:
\mathsf X_t
\times
\mathcal B(\mathbb R^d)
\to
[0,1],
\qquad
j\in[k],
\]
whose product is a regular conditional distribution of the hidden blocks
given \(\mathsf T_t\);

\item there exists a Borel set $\mathsf X_t^{\mathrm{post}}
\subseteq
\mathsf X_t$ 
of full
\(\mathbb P_{\mathsf T_t}\)-measure;

\item for every
\(\tau_t\in\mathsf X_t^{\mathrm{post}}\),
the block kernels have the geometric density
\eqref{eq:posterior-block-density}.

\end{enumerate}

\paragraph{Base case.}

Set $\mathsf X_0^{\mathrm{post}}
:=
\{\ast\}$. 
Clearly, $\mathbb P_{\mathsf T_0}
\left(
\mathsf X_0^{\mathrm{post}}
\right)
=
1$.  
Define
\[
Z_0
:=
\int_{\mathbb R^d}
e^{-\|\mathbf b\|_2^2/2}
\mathbf 1_{\{
\|\mathbf b\|_2\le2\sqrt d
\}}(\mathbf b)
\,
\mathrm d\mathbf b,
\]
and, for every \(j\in[k]\), define
\[
\mu_{j,0}
\left(
\ast,
\mathrm d\mathbf b
\right)
:=
Z_0^{-1}
e^{-\|\mathbf b\|_2^2/2}
\mathbf 1_{\{
\|\mathbf b\|_2\le2\sqrt d
\}}(\mathbf b)
\,
\mathrm d\mathbf b.
\]
By construction of the hard distribution,
\[
\Law
\left(
\rvec b_1,\ldots,\rvec b_k
\right)
=
\bigotimes_{j=1}^k
\mu_{j,0}(\ast,\cdot).
\]
Since
\(\mathsf T_0=\ast\)
is deterministic, the product kernel
\[
K_0
\left(
\ast,
\mathrm d\mathbf b_1,\ldots,
\mathrm d\mathbf b_k
\right)
:=
\bigotimes_{j=1}^k
\mu_{j,0}
\left(
\ast,
\mathrm d\mathbf b_j
\right)
\]
is a regular conditional distribution of $(\rvec b_1,\ldots,\rvec b_k)$ 
given \(\mathsf T_0\).

Moreover,
\[
H_{j,0}(\ast)
=
\mathbb R^d,
\qquad
C_{j,0}(\ast)
=
\left\{
\mathbf b:
\|\mathbf b\|_2\le2\sqrt d
\right\},
\]
so
\eqref{eq:posterior-block-density}
holds at time \(0\).

\paragraph{Induction hypothesis.}

Fix $t\in\{1,\ldots,T\}$,  
and suppose the assertion holds at time \(t-1\).  Thus there are
probability kernels $\mu_{j,t-1}$, $j\in[k]$,  
and a Borel set $\mathsf X_{t-1}^{\mathrm{post}}
\subseteq
\mathsf X_{t-1}$ 
such that $\mathbb P_{\mathsf T_{t-1}}
\left(
\mathsf X_{t-1}^{\mathrm{post}}
\right)
=
1,$
the product kernel
\[
K_{t-1}
\left(
\tau,
\mathrm d\mathbf b_1,\ldots,
\mathrm d\mathbf b_k
\right)
:=
\bigotimes_{j=1}^k
\mu_{j,t-1}
\left(
\tau,
\mathrm d\mathbf b_j
\right)
\]
is a regular conditional distribution of the hidden blocks given
\(\mathsf T_{t-1}\), and the geometric density representation holds at
every $\tau
\in
\mathsf X_{t-1}^{\mathrm{post}}$.  

For $\tau\in\mathsf X_{t-1}$,  
write
\[
\mu_j^\tau
:=
\mu_{j,t-1}(\tau,\cdot),
\qquad
\mathbf x(\tau)
:=
\chi_t(\tau).
\]
The map $\tau
\longmapsto
\mathbf x(\tau)$ 
is Borel measurable.

\paragraph{Parameterized affine supports.}

Write a deterministic past transcript as $\tau
=
\bigl(
(i_s,y_s,\mathbf z_s)
\bigr)_{s=1}^{t-1}$.  
For every block \(j\in[k]\), define $\mathbf M_{j,t-1}(\tau)
\in
\mathbb R^{(t-1)\times d}$ 
rowwise by
\[
\left[
\mathbf M_{j,t-1}(\tau)
\right]_{s,\cdot}
:=
\mathbf 1_{\{i_s=j\}}
\mathbf z_s^\top,
\qquad
s=1,\ldots,t-1,
\]
and define $\mathbf c_{j,t-1}(\tau)
\in
\mathbb R^{t-1}$ 
by $\left[
\mathbf c_{j,t-1}(\tau)
\right]_s
:=
\mathbf 1_{\{i_s=j\}}
y_s$.  
When \(t=1\), these are the unique empty matrix and empty vector.

Both maps $\tau
\longmapsto
\mathbf M_{j,t-1}(\tau)$ 
and $\tau
\longmapsto
\mathbf c_{j,t-1}(\tau)$ 
are Borel measurable, and
\[
H_{j,t-1}(\tau)
=
\left\{
\mathbf b:
\mathbf M_{j,t-1}(\tau)\mathbf b
=
\mathbf c_{j,t-1}(\tau)
\right\}.
\]

Define the graph of the current closed constraint set by
\[
\mathcal C_{j,t-1}
:=
\left\{
(\tau,\mathbf b)
\in
\mathsf X_{t-1}\times\mathbb R^d:
\mathbf b\in C_{j,t-1}(\tau)
\right\}.
\]
This is a Borel subset of
\(\mathsf X_{t-1}\times\mathbb R^d\), since it is defined by finitely
many Borel equalities and inequalities.

\paragraph{Winner-slice kernels.}

For each \(j\in[k]\), define the Borel projection map $R_j:
\mathsf X_{t-1}
\times
\mathbb R^d
\to
\mathbb R$ 
by
\[
R_j(\tau,\mathbf b)
:=
\left\langle
\mathbf b,
\mathbf x(\tau)
\right\rangle.
\]
Define its image kernel
\[
\lambda_j:
\mathsf X_{t-1}
\times
\mathcal B(\mathbb R)
\to
[0,1]
\]
by
\[
\lambda_j(\tau,A)
:=
\int_{\mathbb R^d}
\mathbf 1_A
\left(
R_j(\tau,\mathbf b)
\right)
\mu_{j,t-1}
\left(
\tau,
\mathrm d\mathbf b
\right),
\qquad
A\in\mathcal B(\mathbb R).
\]
First apply
Lemma~\ref{lem:parameterized-disintegration}.
It gives a jointly Borel probability kernel
\[
Q_{j,\mathrm{abs}}^=
\left(
\tau,y,
\mathrm d\mathbf b
\right)
\]
such that, for every fixed \(\tau\), $Q_{j,\mathrm{abs}}^=
(\tau,\cdot,\cdot)$ 
is a regular conditional distribution of
\(\mu_j^\tau\)
given $\mathbf b
\longmapsto
R_j(\tau,\mathbf b)$.   

Next apply 
Lemma~\ref{lem:linear-disintegration}, with parameter space $\Theta
=
\mathsf X_{t-1}$;  
good parameter set $\Theta_0
=
\mathsf X_{t-1}^{\mathrm{post}}$; 
affine data $\mathbf M
=
\mathbf M_{j,t-1}$, $\mathbf c
=
\mathbf c_{j,t-1}$;  
query direction $\mathbf x
=
\chi_t$; 
potential $V(\tau,\mathbf b)
=
\frac12
\|\mathbf b\|_2^2$,  
and constraint graph $\mathcal C
=
\mathcal C_{j,t-1}$.   
The lemma gives a jointly Borel probability kernel
\[
\widetilde Q_{j,t}^=
\left(
\tau,y,
\mathrm d\mathbf b
\right)
\]
and a Borel set
\[
\mathsf G_{j,t}^=
\subseteq
\mathsf X_{t-1}^{\mathrm{post}}
\times
\mathbb R
\]
such that, for every $\tau
\in
\mathsf X_{t-1}^{\mathrm{post}}$,  
\[
\lambda_j
\left(
\tau,
(\mathsf G_{j,t}^=)_\tau
\right)
=
1,
\]
where
\[
(\mathsf G_{j,t}^=)_\tau
:=
\left\{
y:
(\tau,y)
\in
\mathsf G_{j,t}^=
\right\}.
\]
Moreover, for every $(\tau,y)
\in
\mathsf G_{j,t}^=,$
the affine slice
\[
H_{j,t-1}(\tau)
\cap
\left\{
\mathbf b:
R_j(\tau,\mathbf b)=y
\right\}
\]
is nonempty and
\begin{equation}
\begin{aligned}
&
\widetilde Q_{j,t}^=
\left(
\tau,y,
\mathrm d\mathbf b
\right)
\\=&
\left[
Z_{j,t}^{=}(\tau,y)
\right]^{-1}
e^{-\|\mathbf b\|_2^2/2}
\mathbf 1_{
C_{j,t-1}(\tau)
\cap
\{
R_j(\tau,\mathbf b)=y
\}
}(\mathbf b)
\cdot
\sigma_{
H_{j,t-1}(\tau)
\cap
\{
R_j(\tau,\mathbf b)=y
\}
}
\left(
\mathrm d\mathbf b
\right),
\end{aligned}
\label{eq:parameterized-winner-slice-density}
\end{equation}
where $0
<
Z_{j,t}^{=}(\tau,y)
<
\infty$.

Define a global winner-slice kernel by
\begin{equation}
Q_j^=
\left(
\tau,y,E
\right)
:=
\begin{cases}
\widetilde Q_{j,t}^=
\left(
\tau,y,E
\right),
&
\tau
\in
\mathsf X_{t-1}^{\mathrm{post}},
\\[1mm]
Q_{j,\mathrm{abs}}^=
\left(
\tau,y,E
\right),
&
\tau
\notin
\mathsf X_{t-1}^{\mathrm{post}}.
\end{cases}
\label{eq:patched-winner-kernel}
\end{equation}
Since
\(\mathsf X_{t-1}^{\mathrm{post}}\)
is Borel, \(Q_j^=\) is a probability kernel.  For every fixed \(\tau\),
it is a regular conditional distribution of
\(\mu_j^\tau\)
given \(R_j(\tau,\cdot)\), and on the good past-transcript set it has
the jointly measurable geometric form
\eqref{eq:parameterized-winner-slice-density}.

\paragraph{Loser truncation kernels.}

For every $E
\in
\mathcal B(\mathbb R^d)$,  
define
\[
N_j^<(\tau,y,E)
:=
\int_{\mathbb R^d}
\mathbf 1_E(\mathbf b)
\mathbf 1_{\{
R_j(\tau,\mathbf b)<y
\}}
\mu_{j,t-1}
\left(
\tau,
\mathrm d\mathbf b
\right),
\]
and
\[
N_j^\le(\tau,y,E)
:=
\int_{\mathbb R^d}
\mathbf 1_E(\mathbf b)
\mathbf 1_{\{
R_j(\tau,\mathbf b)\le y
\}}
\mu_{j,t-1}
\left(
\tau,
\mathrm d\mathbf b
\right).
\]
The integrands are jointly Borel in
\((\tau,y,\mathbf b)\).  Hence, by parameterized integration,
\[
(\tau,y)
\longmapsto
N_j^<(\tau,y,E)
\]
and
\[
(\tau,y)
\longmapsto
N_j^\le(\tau,y,E)
\]
are Borel measurable.

Set
\[
F_j^<(\tau,y)
:=
N_j^<(\tau,y,\mathbb R^d),
\qquad
F_j^\le(\tau,y)
:=
N_j^\le(\tau,y,\mathbb R^d).
\]
Define
\[
Q_j^<(\tau,y,E)
:=
\begin{cases}
\dfrac{
N_j^<(\tau,y,E)
}{
F_j^<(\tau,y)
},
&
F_j^<(\tau,y)>0,
\\[3mm]
\mu_{j,t-1}(\tau,E),
&
F_j^<(\tau,y)=0,
\end{cases}
\]
and
\[
Q_j^\le(\tau,y,E)
:=
\begin{cases}
\dfrac{
N_j^\le(\tau,y,E)
}{
F_j^\le(\tau,y)
},
&
F_j^\le(\tau,y)>0,
\\[3mm]
\mu_{j,t-1}(\tau,E),
&
F_j^\le(\tau,y)=0.
\end{cases}
\]
These are probability kernels from $\mathsf X_{t-1}
\times
\mathbb R$ 
to
\(\mathbb R^d\).
The definitions on zero-denominator parameter values are immaterial,
because the observation law introduced below assigns zero mass to those
values. 
Define also the boundary-mass function
\begin{equation}
\begin{aligned}
\beta_j(\tau,y)
:=
F_j^\le(\tau,y)
-
F_j^<(\tau,y)
=
\mu_{j,t-1}
\left(
\tau,
\left\{
\mathbf b:
R_j(\tau,\mathbf b)=y
\right\}
\right).
\end{aligned}
\label{eq:posterior-boundary-mass}
\end{equation}
The map $(\tau,y)
\longmapsto
\beta_j(\tau,y)$
is Borel measurable.

\paragraph{Conditional law of the next observation.}

For every possible winner \(i\in[k]\), define
\[
w_i(\tau,y)
:=
\prod_{j<i}
F_j^<(\tau,y)
\prod_{j>i}
F_j^\le(\tau,y),
\]
and define the subprobability kernel
\[
\Lambda_i:
\mathsf X_{t-1}
\times
\mathcal B(\mathbb R)
\to
[0,1]
\]
by
\[
\Lambda_i(\tau,A)
:=
\int_A
w_i(\tau,y)
\lambda_i(\tau,\mathrm dy).
\]
Equivalently,
\[
\Lambda_i(\tau,\mathrm dy)
=
w_i(\tau,y)
\lambda_i(\tau,\mathrm dy).
\]

The strict factors for
\(j<i\)
and the non-strict factors for
\(j>i\)
are precisely those imposed by the deterministic smallest-index
tie-breaking rule.

For $(\tau,i,y)
\in
\mathsf X_{t-1}
\times
[k]
\times
\mathbb R$, 
define the updated block kernels
\[
Q_j^{\tau,i,y}
:=
\begin{cases}
Q_j^<(\tau,y,\cdot),
&
j<i,
\\
Q_i^=(\tau,y,\cdot),
&
j=i,
\\
Q_j^\le(\tau,y,\cdot),
&
j>i.
\end{cases}
\]
Define the product kernel
\[
Q_t^i
\left(
\tau,y,
\mathrm d\mathbf b_1,\ldots,
\mathrm d\mathbf b_k
\right)
:=
\bigotimes_{j=1}^k
Q_j^{\tau,i,y}
\left(
\mathrm d\mathbf b_j
\right).
\]
A finite product of probability kernels is again a probability kernel.

For fixed \(\tau\), define $\mathbb P_\tau
:=
\bigotimes_{j=1}^k
\mu_j^\tau$. 
On
\((\mathbb R^d)^k\),
define
\[
\mathcal Y_\tau
\left(
\mathbf b_1,\ldots,\mathbf b_k
\right)
:=
\max_{1\le j\le k}
R_j(\tau,\mathbf b_j),
\]
and
\[
\mathcal I_\tau
\left(
\mathbf b_1,\ldots,\mathbf b_k
\right)
:=
\min
\argmax_{1\le j\le k}
R_j(\tau,\mathbf b_j).
\]
Fix a product rectangle
\[
E
=
E_1\times\cdots\times E_k
\subseteq
(\mathbb R^d)^k
\]
and a Borel set $A
\subseteq
\mathbb R$.  
Independence under
\(\mathbb P_\tau\),
followed by disintegration of the \(i\)-th block, gives
\[
\begin{aligned}
&
\mathbb P_\tau
\left(
(\mathbf b_1,\ldots,\mathbf b_k)\in E,\,
\mathcal I_\tau=i,\,
\mathcal Y_\tau\in A
\right)
\\
&\quad=
\int_A
Q_i^=(\tau,y,E_i)
\left[
\prod_{j<i}
N_j^<(\tau,y,E_j)
\right]
\left[
\prod_{j>i}
N_j^\le(\tau,y,E_j)
\right]
\lambda_i(\tau,\mathrm dy).
\end{aligned}
\]
By construction,
\[
N_j^<(\tau,y,E_j)
=
F_j^<(\tau,y)
Q_j^<(\tau,y,E_j),
\]
and
\[
N_j^\le(\tau,y,E_j)
=
F_j^\le(\tau,y)
Q_j^\le(\tau,y,E_j).
\]
These identities remain true when the corresponding denominator is zero,
because then both sides vanish.  Hence
\begin{equation}
\begin{aligned}
\mathbb P_\tau
\left(
(\mathbf b_1,\ldots,\mathbf b_k)\in E,\,
\mathcal I_\tau=i,\,
\mathcal Y_\tau\in A
\right)
=
\int_A
Q_t^i(\tau,y,E)
\Lambda_i(\tau,\mathrm dy).
\end{aligned}
\label{eq:one-step-kernel-disintegration}
\end{equation}
A monotone-class argument extends
\eqref{eq:one-step-kernel-disintegration}
from product rectangles to all $E
\in
\mathcal B\bigl((\mathbb R^d)^k\bigr)$. 
Taking $E
=
(\mathbb R^d)^k$ 
and summing over \(i\in[k]\) gives
\[
\sum_{i=1}^k
\Lambda_i(\tau,\mathbb R)
=
1.
\]
Define a probability kernel $\Lambda:
\mathsf X_{t-1}
\times
\mathcal B([k]\times\mathbb R)
\to
[0,1]$ 
by
\begin{equation}
\Lambda(\tau,D)
:=
\sum_{i=1}^k
\int_{\mathbb R}
\mathbf 1_D(i,y)
\Lambda_i(\tau,\mathrm dy),
\qquad
D\in\mathcal B([k]\times\mathbb R).
\label{eq:next-observation-kernel}
\end{equation}
We write $\Lambda
\left(
\tau,
\mathrm d\ell,
\mathrm dy
\right)$ 
for this kernel, where \(\ell\in[k]\) denotes its discrete coordinate. 
Equation
\eqref{eq:one-step-kernel-disintegration}
shows that $\Lambda(\tau,\cdot)
=
(\mathcal I_\tau,\mathcal Y_\tau)_\sharp
\mathbb P_\tau.$
Thus \(\Lambda(\tau,\cdot)\) is the law of the next augmented observation
under the conditional product law associated with the past transcript
\(\tau\).

\paragraph{Definition of the updated posterior kernels.}

Every $\tau_t\in\mathsf X_t$ 
has a unique decomposition $\tau_t
=
\bigl(
\tau,
(i,y,\mathbf z)
\bigr)$,  
where $\tau\in\mathsf X_{t-1}$, $i\in[k]$, $y\in\mathbb R$, and  $\mathbf z\in B_2^d$. 
Define
\[
\mu_{j,t}(\tau_t,E)
:=
Q_j^{\tau,i,y}(E),
\qquad
E\in\mathcal B(\mathbb R^d).
\]
This definition ignores the final recorded query coordinate
\(\mathbf z\) on transcripts that are inconsistent with the fixed
algorithm.  The coordinate map $\tau_t
\longmapsto
(\tau,i,y)$ 
is Borel measurable.  Hence every $\mu_{j,t}:
\mathsf X_t
\times
\mathcal B(\mathbb R^d)
\to
[0,1]$
is a probability kernel. 
Define
\[
K_t
\left(
\tau_t,
\mathrm d\mathbf b_1,\ldots,
\mathrm d\mathbf b_k
\right)
:=
\bigotimes_{j=1}^k
\mu_{j,t}
\left(
\tau_t,
\mathrm d\mathbf b_j
\right).
\]

\paragraph{Global regular conditional distribution identity.}

The induction hypothesis,
\eqref{eq:one-step-kernel-disintegration},
and
\eqref{eq:query-reconstruction}
imply that $\Lambda(\mathsf T_{t-1},\cdot)$ 
is a version of the conditional law of $(\rsc I_t,\rsc Y_t)$ 
given \(\mathsf T_{t-1}\).  Combining this fact with
\eqref{eq:transcript-recursion},
we obtain the transition identity: for every nonnegative Borel function
\[
\varphi:
\mathsf X_t
\to
[0,\infty],
\]
\begin{equation}
\begin{aligned}
\int_{\mathsf X_t}
\varphi(\tau_t)
\,
\mathbb P_{\mathsf T_t}
\left(
\mathrm d\tau_t
\right)=
\int_{\mathsf X_{t-1}}
\int_{[k]\times\mathbb R}
\varphi
\left(
\tau\oplus_t(\ell,y)
\right)
\Lambda
\left(
\tau,
\mathrm d\ell,
\mathrm dy
\right)
\mathbb P_{\mathsf T_{t-1}}
\left(
\mathrm d\tau
\right).
\end{aligned}
\label{eq:transcript-transition-identity}
\end{equation}
Let $A
\in
\mathcal B(\mathsf X_t)$ 
and $E
\in
\mathcal B\bigl((\mathbb R^d)^k\bigr)$.  
Conditioning on
\(\mathsf T_{t-1}\)
and applying
\eqref{eq:one-step-kernel-disintegration}
gives
\[
\begin{aligned}
&
\mathbb P_\Xi
\left(
(\rvec b_1,\ldots,\rvec b_k)\in E,\,
\mathsf T_t\in A
\right)
\\
&\quad=
\int_{\mathsf X_{t-1}}
\int_{[k]\times\mathbb R}
\mathbf 1_A
\left(
\tau\oplus_t(\ell,y)
\right)
Q_t^\ell(\tau,y,E)\cdot
\Lambda
\left(
\tau,
\mathrm d\ell,
\mathrm dy
\right)
\mathbb P_{\mathsf T_{t-1}}
\left(
\mathrm d\tau
\right).
\end{aligned}
\]
For every consistent extension,
\[
K_t
\left(
\tau\oplus_t(\ell,y),
E
\right)
=
Q_t^\ell(\tau,y,E).
\]
Hence the preceding expression is
\[
\begin{aligned}
\int_{\mathsf X_{t-1}}
\int_{[k]\times\mathbb R}
\mathbf 1_A
\left(
\tau\oplus_t(\ell,y)
\right)
K_t
\left(
\tau\oplus_t(\ell,y),
E
\right)\cdot
\Lambda
\left(
\tau,
\mathrm d\ell,
\mathrm dy
\right)
\mathbb P_{\mathsf T_{t-1}}
\left(
\mathrm d\tau
\right).
\end{aligned}
\]
Applying
\eqref{eq:transcript-transition-identity}
to
\[
\varphi(\tau_t)
:=
\mathbf 1_A(\tau_t)
K_t(\tau_t,E)
\]
gives
\[
\mathbb P_\Xi
\left(
(\rvec b_1,\ldots,\rvec b_k)\in E,\,
\mathsf T_t\in A
\right)
=
\int_A
K_t(\tau_t,E)
\,
\mathbb P_{\mathsf T_t}
\left(
\mathrm d\tau_t
\right).
\]
Thus \(K_t\) is a regular conditional distribution of the hidden blocks
given \(\mathsf T_t\), proving
\eqref{eq:posterior-kernel-disintegration}
and
\eqref{eq:posterior-product-kernel}.

\paragraph{Borel good observation sets.}

We now construct a Borel full-measure set on which the geometric
description is valid.

For every possible winner \(i\in[k]\), define
\begin{equation}
\begin{aligned}
\mathsf G_{t,i}
:=
\Bigl\{
(\tau,y)
\in
\mathsf X_{t-1}\times\mathbb R:
\;&
\tau
\in
\mathsf X_{t-1}^{\mathrm{post}},
\ 
(\tau,y)
\in
\mathsf G_{i,t}^=,
\\
&
F_j^<(\tau,y)>0
\text{ and }
\beta_j(\tau,y)=0
\quad
\text{for every }j<i,
\\
&
F_j^\le(\tau,y)>0
\quad
\text{for every }j>i
\Bigr\}.
\end{aligned}
\label{eq:good-one-step-observation-set}
\end{equation}
This is a Borel set. 
We claim that, for every $\tau
\in
\mathsf X_{t-1}^{\mathrm{post}}$, 
\begin{equation}
\Lambda_i
\left(
\tau,
(\mathsf G_{t,i})_\tau
\right)
=
\Lambda_i(\tau,\mathbb R),
\qquad
i\in[k].
\label{eq:good-observation-full-mass}
\end{equation}
First, $\lambda_i
\left(
\tau,
(\mathsf G_{i,t}^=)_\tau
\right)
=
1$ and  $\Lambda_i(\tau,\cdot)
\ll
\lambda_i(\tau,\cdot)$. 
Hence the winner-slice good set has full
\(\Lambda_i(\tau,\cdot)\)-measure. 
Second, if \(j<i\), then the factor $F_j^<(\tau,y)$ 
appears in \(w_i(\tau,y)\).  Therefore
\[
\Lambda_i
\left(
\tau,
\{y:F_j^<(\tau,y)=0\}
\right)
=
0.
\]
Similarly, if \(j>i\), then $F_j^\le(\tau,y)$ 
appears in \(w_i(\tau,y)\), so
\[
\Lambda_i
\left(
\tau,
\{y:F_j^\le(\tau,y)=0\}
\right)
=
0.
\]
It remains to consider
\(\beta_j(\tau,y)\)
for \(j<i\).
Fix such a block \(j\). 
If the functional $\mathbf b
\longmapsto
R_j(\tau,\mathbf b)$ 
is nonconstant on
\(H_{j,t-1}(\tau)\),
then, for every \(y\),
\[
H_{j,t-1}(\tau)
\cap
\left\{
\mathbf b:
R_j(\tau,\mathbf b)=y
\right\}
\]
is a proper affine hyperplane in
\(H_{j,t-1}(\tau)\).  It has zero
\(\sigma_{H_{j,t-1}(\tau)}\)-measure, and therefore
\[
\beta_j(\tau,y)
=
0
\qquad
\text{for every }y.
\]
If the functional is constant on
\(H_{j,t-1}(\tau)\), say $R_j(\tau,\mathbf b)
\equiv c$,  
then
\[
F_j^<(\tau,y)
=
\mathbf 1_{\{c<y\}},
\qquad
\beta_j(\tau,y)
=
\mathbf 1_{\{y=c\}}.
\]
Since \(F_j^<(\tau,y)\) is a factor of
\(w_i(\tau,y)\), the measure
\(\Lambda_i(\tau,\cdot)\)
assigns zero mass to \(y=c\).  Hence
\[
\beta_j(\tau,y)=0
\]
for
\(\Lambda_i(\tau,\cdot)\)-almost every \(y\). 
This proves
\eqref{eq:good-observation-full-mass}.

\paragraph{Borel full-measure transcript set.}

Define
\begin{equation}
\begin{aligned}
\mathsf X_t^{\mathrm{post}}
:=
\bigcup_{i=1}^k
\Bigl\{
\bigl(
\tau,
(i,y,\mathbf z)
\bigr)
\in
\mathsf X_t:
\;
\tau
\in
\mathsf X_{t-1}^{\mathrm{post}},\ 
\mathbf z
=
\chi_t(\tau),
\ 
(\tau,y)
\in
\mathsf G_{t,i}
\Bigr\}.
\end{aligned}
\label{eq:posterior-good-transcript-set}
\end{equation}
The predecessor-coordinate maps, the maps
\(\chi_t\), and all sets
\(\mathsf G_{t,i}\)
are Borel.  Hence
\(\mathsf X_t^{\mathrm{post}}\)
is Borel. 
Using
\eqref{eq:transcript-transition-identity},
\eqref{eq:good-observation-full-mass},
and the induction hypothesis,
\[
\begin{aligned}
\mathbb P_{\mathsf T_t}
\left(
\mathsf X_t^{\mathrm{post}}
\right)
&=
\int_{\mathsf X_{t-1}^{\mathrm{post}}}
\sum_{i=1}^k
\Lambda_i
\left(
\tau,
(\mathsf G_{t,i})_\tau
\right)
\mathbb P_{\mathsf T_{t-1}}
\left(
\mathrm d\tau
\right)
\\
&=
\int_{\mathsf X_{t-1}^{\mathrm{post}}}
\sum_{i=1}^k
\Lambda_i(\tau,\mathbb R)
\mathbb P_{\mathsf T_{t-1}}
\left(
\mathrm d\tau
\right)
\\
&=
\mathbb P_{\mathsf T_{t-1}}
\left(
\mathsf X_{t-1}^{\mathrm{post}}
\right)
\\
&=
1.
\end{aligned}
\]

\paragraph{Geometric identification on the good transcript set.}

Fix $\tau_t
\in
\mathsf X_t^{\mathrm{post}}$.
Then, for a unique $i\in[k]$,  
we can write
\[
\tau_t
=
\tau\oplus_t(i,y)
=
\bigl(
\tau,
(i,y,\chi_t(\tau))
\bigr),
\]
where $\tau
\in
\mathsf X_{t-1}^{\mathrm{post}}$ 
and $(\tau,y)
\in
\mathsf G_{t,i}$.  
Write
\[
H_j
:=
H_{j,t-1}(\tau),
\qquad
C_j
:=
C_{j,t-1}(\tau),
\qquad
\mathbf x
:=
\chi_t(\tau).
\]

For the winner block \(j=i\),
\eqref{eq:parameterized-winner-slice-density}
gives
\[
\begin{aligned}
&
\mu_{i,t}
\left(
\tau_t,
\mathrm d\mathbf b
\right)
\\
&=
\left[
Z_{i,t}^{=}(\tau,y)
\right]^{-1}
e^{-\|\mathbf b\|_2^2/2}
\mathbf 1_{
C_i
\cap
\{
\langle\mathbf b,\mathbf x\rangle=y
\}
}(\mathbf b)\cdot
\sigma_{
H_i
\cap
\{
\langle\mathbf b,\mathbf x\rangle=y
\}
}
\left(
\mathrm d\mathbf b
\right).
\end{aligned}
\]
The updated affine support is therefore
\[
H_{i,t}(\tau_t)
=
H_i
\cap
\left\{
\mathbf b:
\langle\mathbf b,\mathbf x\rangle=y
\right\}, \quad \text{and} \quad C_{i,t}(\tau_t)
=
C_i
\cap
\left\{
\mathbf b:
\langle\mathbf b,\mathbf x\rangle=y
\right\}. 
\]

Now consider a lower-index loser
\(j<i\).
Since $(\tau,y)
\in
\mathsf G_{t,i}$,  
we have
\[
F_j^<(\tau,y)>0, \quad \text{and} \quad \beta_j(\tau,y)=0. 
\]
Thus
\[
N_j^\le(\tau,y,E)
-
N_j^<(\tau,y,E)
=
\mu_{j,t-1}
\left(
\tau,
E
\cap
\{
R_j(\tau,\mathbf b)=y
\}
\right)
=
0
\]
for every Borel \(E\).  Consequently,
\[
N_j^\le(\tau,y,E)
=
N_j^<(\tau,y,E), \quad F_j^\le(\tau,y)
=
F_j^<(\tau,y),\quad \text{and}\quad Q_j^<(\tau,y,\cdot)
=
Q_j^\le(\tau,y,\cdot).
\]
Using the induction hypothesis,
\[
\begin{aligned}
&
\mu_{j,t}
\left(
\tau_t,
\mathrm d\mathbf b
\right)
\\
=&
Q_j^\le
\left(
\tau,y,
\mathrm d\mathbf b
\right)=
\left[
Z_{j,t-1}(\tau)
F_j^\le(\tau,y)
\right]^{-1}
e^{-\|\mathbf b\|_2^2/2}
\mathbf 1_{
C_j
\cap
\{
\langle\mathbf b,\mathbf x\rangle\le y
\}
}(\mathbf b)
\,
\sigma_{H_j}
\left(
\mathrm d\mathbf b
\right).
\end{aligned}
\]

For an upper-index loser
\(j>i\),
we have $F_j^\le(\tau,y)>0$,  
and directly obtain
\[
\begin{aligned}
&\mu_{j,t}
\left(
\tau_t,
\mathrm d\mathbf b
\right)\\
=&
Q_j^\le
\left(
\tau,y,
\mathrm d\mathbf b
\right)=
\left[
Z_{j,t-1}(\tau)
F_j^\le(\tau,y)
\right]^{-1}
e^{-\|\mathbf b\|_2^2/2}
\mathbf 1_{
C_j
\cap
\{
\langle\mathbf b,\mathbf x\rangle\le y
\}
}(\mathbf b)
\,
\sigma_{H_j}
\left(
\mathrm d\mathbf b
\right).
\end{aligned}
\]

Thus, for every loser \(j\neq i\),
\[
H_{j,t}(\tau_t)
=
H_j, \quad \text{and} \quad C_{j,t}(\tau_t)
=
C_j
\cap
\left\{
\mathbf b:
\langle\mathbf b,\mathbf x\rangle\le y
\right\}.
\]
Iterating these updates from time \(0\) gives exactly
\[
H_{j,t}(\tau_t)
=
\left\{
\mathbf b:
\langle\mathbf b,\mathbf z_s\rangle=y_s
\text{ whenever }i_s=j
\right\},
\]
and
\[
C_{j,t}(\tau_t)
=
\left\{
\mathbf b\in H_{j,t}(\tau_t):
\|\mathbf b\|_2\le2\sqrt d,\quad
\langle\mathbf b,\mathbf z_s\rangle\le y_s
\text{ whenever }i_s\neq j
\right\}.
\]
The set
\(H_{j,t}(\tau_t)\)
is an affine subspace, being an intersection of finitely many affine
hyperplanes.  The set
\(C_{j,t}(\tau_t)\)
is closed and convex, being the intersection within that affine subspace
of a closed Euclidean ball and finitely many closed halfspaces.

The relevant normalizing constants are strictly positive by the
definition of the good observation sets.  They are finite because
\[
C_{j,t}(\tau_t)
\subseteq
\left\{
\mathbf b:
\|\mathbf b\|_2\le2\sqrt d
\right\}
\]
is bounded and
\(\sigma_{H_{j,t}(\tau_t)}\)
is locally finite.  Therefore $H_{j,t}(\tau_t)$ 
and $C_{j,t}(\tau_t)$ 
are nonempty and
\[
0
<
Z_{j,t}(\tau_t)
:=
\int_{H_{j,t}(\tau_t)}
e^{-\|\mathbf b\|_2^2/2}
\mathbf 1_{C_{j,t}(\tau_t)}(\mathbf b)
\,
\sigma_{H_{j,t}(\tau_t)}
\left(
\mathrm d\mathbf b
\right)
<
\infty.
\]
We have proved
\[
\mu_{j,t}
\left(
\tau_t,
\mathrm d\mathbf b
\right)
=
Z_{j,t}(\tau_t)^{-1}
e^{-\|\mathbf b\|_2^2/2}
\mathbf 1_{C_{j,t}(\tau_t)}(\mathbf b)
\,
\sigma_{H_{j,t}(\tau_t)}
\left(
\mathrm d\mathbf b
\right).
\]

\paragraph{Strong log-concavity.}

For a set \(C\), define its extended-valued convex indicator by
\[
\operatorname{ind}_C(\mathbf b)
:=
\begin{cases}
0,
&
\mathbf b\in C,
\\
+\infty,
&
\mathbf b\notin C.
\end{cases}
\]
The negative log-density of the \(j\)-th posterior block relative to
\(\sigma_{H_{j,t}(\tau_t)}\) is
\[
\mathcal V_{j,t}^{\tau_t}(\mathbf b)
=
\frac12
\|\mathbf b\|_2^2
+
\operatorname{ind}_{C_{j,t}(\tau_t)}(\mathbf b)
+
\log Z_{j,t}(\tau_t).
\]
Since
\(C_{j,t}(\tau_t)\)
is nonempty, closed, and convex, $\operatorname{ind}_{C_{j,t}(\tau_t)}$ 
is proper, lower semicontinuous, and convex.  Hence $\mathcal V_{j,t}^{\tau_t}(\mathbf b)
-
\frac12
\|\mathbf b\|_2^2$ 
is convex on
\(H_{j,t}(\tau_t)\).
Thus the posterior block is \(1\)-strongly log-concave relative to its
affine support.

If $\dim H_{j,t}(\tau_t)=0$,  
then
\(H_{j,t}(\tau_t)\)
consists of a single point.  Since
\(C_{j,t}(\tau_t)\)
is nonempty and contained in this singleton, the posterior block is the
Dirac measure at that point.

This completes the induction.
\end{proof}

\begin{theorem}[One-step selection bound on affine supports]
\label{thm:one-step}
Let \(d,k\ge1\), and let $\rvec{u}_1,\ldots,\rvec{u}_k$ 
be independent random vectors in \(\mathbb R^d\).  For each
\(i\in[k]\), let $H_i
\subseteq
\mathbb R^d$ 
be a nonempty affine subspace, possibly of dimension zero, and suppose
that the law \(\mu_i\) of \(\rvec{u}_i\) is
\[
\mu_i(\mathrm d\mathbf u)
=
\mathcal Z_i^{-1}
e^{-\mathcal V_i(\mathbf u)}
\,\sigma_{H_i}(\mathrm  d\mathbf u),
\]
where
\[
0
<
\mathcal Z_i
:=
\int_{H_i}
e^{-\mathcal V_i(\mathbf u)}
\,\mathrm  d\sigma_{H_i}(\mathbf u)
<
\infty,
\]
and $\mathcal V_i:
H_i
\to
\mathbb R\cup\{+\infty\}$ 
is proper and lower semicontinuous.  Assume that $\mathbf u
\to
\mathcal V_i(\mathbf u)
-
\frac12\|\mathbf u\|_2^2$ 
is convex on \(H_i\). 
Here
\(\sigma_{H_i}\)
denotes the Lebesgue--Hausdorff measure on \(H_i\); when
\(\dim H_i=0\), it is the unit point mass on \(H_i\). 
Fix $\mathbf v
\in
\mathbb S^{d-1}$, 
and define $\rsc{R}_i
:=
\left\langle
\rvec{u}_i,
\mathbf v
\right\rangle$ for $i=1,\ldots,k$,  $\rsc{Y}
:=
\max_{1\le i\le k}
\rsc{R}_i$, $\rsc{I}
:=
\min
\argmax_{1\le i\le k}
\rsc{R}_i$ and $\rvec{s}
:=
\sum_{i=1}^k
\rvec{u}_i.$ 
Then there exists a universal constant \(C>0\) such that
\begin{equation}
\mathbb E
\left\|
\mathbb E
\left[
\rvec{s}
\mid
\sigma(\rsc{I},\rsc{Y})
\right]
-
\mathbb E\rvec{s}
\right\|_2^2
\le
C\log(ek).
\label{eq:one-step-selection}
\end{equation}
Equivalently, using the standard conditional-expectation shorthand,
\[
\mathbb E
\left\|
\mathbb E
\left[
\rvec{s}
\mid
\rsc{I},\rsc{Y}
\right]
-
\mathbb E\rvec{s}
\right\|_2^2
\le
C\log(ek).
\]
\end{theorem}
The proof of Theorem~\ref{thm:one-step} is given in
Section~\ref{sec:one-step}.  We now apply this local estimate to the
adaptive posterior process.  Conditionally on the past augmented
transcript, the hidden blocks are independent and \(1\)-strongly
log-concave on affine supports, while the next query direction is fixed.
The theorem therefore controls the one-step movement of the posterior
mean of their sum.
\begin{proposition}[Augmented posterior energy]
\label{prop:augmented-energy}
There exists a universal constant
\(C_{\mathrm{en}}>0\)
such that, for every $\mathsf d
\in
\mathcal A_{\mathrm{det}}^{(T)}$,  
the augmented posterior means satisfy
\begin{equation}
\mathbb E_\Xi
\left\|
\rvec{m}_T
\right\|_2^2
\le
C_{\mathrm{en}}
T\log(ek).
\label{eq:augmented-energy-unscaled}
\end{equation}
Consequently,
\begin{equation}
\mathbb E_\Xi
\left\|
\rvec{m}^{\mathrm{sc}}_T
\right\|_2^2
\le
C_{\mathrm{en}}
\frac{T\log(ek)}{d}.
\label{eq:augmented-energy-scaled}
\end{equation}
\end{proposition}
The proof uses the affine support one step result Theorem~\ref{thm:one-step}.
\begin{proof}[Proof of Proposition~\ref{prop:augmented-energy}]
The process
\[
\rvec{m}_t
=
\mathbb E_\Xi
\left[
\rvec{s}
\mid
\mathcal G_t
\right],
\qquad
t=0,\ldots,T,
\]
is a square-integrable
\((\mathcal G_t)\)-martingale.  Indeed, for
\(t=1,\ldots,T\),
\begin{equation}\label{eq:posterior-mean-martingale}
\begin{aligned}
\mathbb E_\Xi
\left[
\rvec{m}_t
\mid
\mathcal G_{t-1}
\right]
=
\mathbb E_\Xi
\left[
\mathbb E_\Xi[
\rvec{s}\mid\mathcal G_t]
\mid
\mathcal G_{t-1}
\right]
=
\mathbb E_\Xi
\left[
\rvec{s}
\mid
\mathcal G_{t-1}
\right]
=
\rvec{m}_{t-1}.
\end{aligned}
\end{equation}

The truncated Gaussian law
\(\gamma_d^{\mathrm{tr}}\)
is centrally symmetric, so $\mathbb E_\Xi\rvec{b}_i
=
\mathbf 0$ for every $i$. 
Consequently,
\[
\rvec{m}_0
=
\mathbb E_\Xi\rvec{s}
=
\mathbf 0.
\]
Since martingale increments are orthogonal in \(L^2\),
\begin{equation}
\mathbb E_\Xi
\left\|
\rvec{m}_T
\right\|_2^2
=
\sum_{t=1}^T
\mathbb E_\Xi
\left\|
\rvec{m}_t-\rvec{m}_{t-1}
\right\|_2^2.
\label{eq:posterior-martingale-orthogonality}
\end{equation}
Fix $t\in\{1,\ldots,T\}$.

For every $\tau_{t-1}
\in
\mathsf X_{t-1}^{\mathrm{post}}$,  
Lemma~\ref{lem:regular-product-posterior}
gives the conditional product law
\[
\mathbb P_{\tau_{t-1}}
:=
\bigotimes_{i=1}^k
\mu_{i,t-1}
\left(
\tau_{t-1},
\cdot
\right),
\]
and every factor is \(1\)-strongly log-concave relative to an affine
support. Since $\mathbb P_{\mathsf T_{t-1}}
\left(
\mathsf X_{t-1}^{\mathrm{post}}
\right)
=
1$,  
the resulting conditional increment estimate holds
\(\mathbb P_\Xi\)-almost surely. 
Let $\mathbf q_t(\tau_{t-1})
:=
\chi_t(\tau_{t-1})$ 
be the deterministic query corresponding to this past transcript, and set $a_t(\tau_{t-1})
:=
\left\|
\mathbf q_t(\tau_{t-1})
\right\|_2$.
Under
\(\mathbb P_{\tau_{t-1}}\),
let $\rvec{u}_1^{\,\tau_{t-1}},
\ldots,
\rvec{u}_k^{\,\tau_{t-1}}$ 
denote the coordinate random vectors, and define $\rvec{s}^{\,\tau_{t-1}}
:=
\sum_{i=1}^k
\rvec{u}_i^{\,\tau_{t-1}}$.

If $a_t(\tau_{t-1})=0$, 
then all projected values are zero.  Hence the newly revealed observation
is deterministically $(\rsc I_t,\rsc Y_t)
=
(1,0)$ 
under the conditional law.  It carries no information about the hidden
blocks, and therefore
\begin{equation}\label{eq:zero-query-increment}
\mathbb E_\Xi
\left[
\left\|
\rvec{m}_t-\rvec{m}_{t-1}
\right\|_2^2
\;\middle|\;
\mathsf T_{t-1}=\tau_{t-1}
\right]
=
0.
\end{equation}

Suppose now that $a_t(\tau_{t-1})>0$,  
and define the unit direction
\[
\mathbf v_t(\tau_{t-1})
:=
\frac{
\mathbf q_t(\tau_{t-1})
}{
a_t(\tau_{t-1})
}
\in
\mathbb S^{d-1}.
\]
Under
\(\mathbb P_{\tau_{t-1}}\), define
\[
\rsc{R}_{i,t}^{\,\tau_{t-1}}
:=
\left\langle
\rvec{u}_i^{\,\tau_{t-1}},
\mathbf v_t(\tau_{t-1})
\right\rangle,\quad  \rsc{Y}_t^{\,\tau_{t-1}}
:=
\max_{1\le i\le k}
\rsc{R}_{i,t}^{\,\tau_{t-1}},\quad\text{and}\quad
\rsc{I}_t^{\,\tau_{t-1}}
:=
\min
\argmax_{1\le i\le k}
\rsc{R}_{i,t}^{\,\tau_{t-1}}.
\]
The actual unscaled maximum observation conditioned on 
\(\mathsf T_{t-1}=\tau_{t-1}\)
is $a_t(\tau_{t-1})
\rsc{Y}_t^{\,\tau_{t-1}}$, 
and positive rescaling does not change the active index.  Since
\(a_t(\tau_{t-1})>0\) is deterministic under the conditional law,
\[
\sigma\left(
\rsc{I}_t^{\,\tau_{t-1}},
a_t(\tau_{t-1})
\rsc{Y}_t^{\,\tau_{t-1}}
\right)
=
\sigma\left(
\rsc{I}_t^{\,\tau_{t-1}},
\rsc{Y}_t^{\,\tau_{t-1}}
\right).
\] 
Because the query-coordinate is already determined by the past
transcript, the only new information at time \(t\) is this active-index
and maximum-value pair.  The posterior-kernel disintegration therefore
gives
\begin{equation}
\begin{aligned}
&
\mathbb E_\Xi
\left[
\left\|
\rvec{m}_t-\rvec{m}_{t-1}
\right\|_2^2
\;\middle|\;
\mathsf T_{t-1}=\tau_{t-1}
\right]
\\
=
&\mathbb E_{\tau_{t-1}}
\left\|
\mathbb E_{\tau_{t-1}}
\left[
\rvec{s}^{\,\tau_{t-1}}
\mid
\rsc{I}_t^{\,\tau_{t-1}},
\rsc{Y}_t^{\,\tau_{t-1}}
\right]
-
\mathbb E_{\tau_{t-1}}
\rvec{s}^{\,\tau_{t-1}}
\right\|_2^2.
\end{aligned}
\label{eq:conditional-increment-identification}
\end{equation}

The conditional block laws satisfy all assumptions of
Theorem~\ref{thm:one-step}.  Applying that theorem in the direction
\(\mathbf v_t(\tau_{t-1})\) yields
\begin{equation}\label{eq:conditional-increment-bound-transcript}
\mathbb E_\Xi
\left[
\left\|
\rvec{m}_t-\rvec{m}_{t-1}
\right\|_2^2
\;\middle|\;
\mathsf T_{t-1}=\tau_{t-1}
\right]
\le
C\log(ek).
\end{equation}

Combining
\eqref{eq:zero-query-increment}
and
\eqref{eq:conditional-increment-bound-transcript},
we obtain, for
\(\mathbb P_{\mathsf T_{t-1}}\)-almost every
\(\tau_{t-1}\),
\[
\mathbb E_\Xi
\left[
\left\|
\rvec{m}_t-\rvec{m}_{t-1}
\right\|_2^2
\;\middle|\;
\mathsf T_{t-1}=\tau_{t-1}
\right]
\le
C\log(ek).
\]
Equivalently,
\begin{equation}
\mathbb E_\Xi
\left[
\left\|
\rvec{m}_t-\rvec{m}_{t-1}
\right\|_2^2
\mid
\mathcal G_{t-1}
\right]
\le
C\log(ek)
\qquad
\mathbb P_\Xi\text{-almost surely}.
\label{eq:conditional-increment-bound}
\end{equation}

Taking expectations in
\eqref{eq:conditional-increment-bound}
and summing over \(t\), equation
\eqref{eq:posterior-martingale-orthogonality}
gives
\[
\mathbb E_\Xi
\left\|
\rvec{m}_T
\right\|_2^2
\le
CT\log(ek).
\]
Thus
\eqref{eq:augmented-energy-unscaled}
holds with $C_{\mathrm{en}}
:=
C$.

Finally,
\[
\rvec{m}^{\mathrm{sc}}_T
=
\mathbb E_\Xi
\left[
\rvec{s}^{\mathrm{sc}}
\mid
\mathcal G_T
\right]
=
\frac{
\rvec{m}_T
}{
\sqrt d
},
\]
and hence
\[
\mathbb E_\Xi
\left\|
\rvec{m}^{\mathrm{sc}}_T
\right\|_2^2
=
\frac1d
\mathbb E_\Xi
\left\|
\rvec{m}_T
\right\|_2^2
\le
C_{\mathrm{en}}
\frac{T\log(ek)}{d}.
\]
\end{proof}
\paragraph{Value-only consequence.}

The actual oracle reveals only the scalar value history.  Recall that
\[
\mathcal F_T
=
\sigma\left(
\rsc{Y}^{\mathrm{sc}}_1,
\ldots,
\rsc{Y}^{\mathrm{sc}}_T
\right)
=
\sigma\left(
\rsc{Y}_1,
\ldots,
\rsc{Y}_T
\right),
\]
and $\mathcal F_T
\subseteq
\mathcal G_T$.  
By the tower property,
\[
\mathbb E_\Xi
\left[
\rvec{s}
\mid
\mathcal F_T
\right]
=
\mathbb E_\Xi
\left[
\rvec{m}_T
\mid
\mathcal F_T
\right].
\]
Conditional Jensen's inequality and
Proposition~\ref{prop:augmented-energy}
therefore give
\[
\begin{aligned}
\mathbb E_\Xi
\left\|
\mathbb E_\Xi
\left[
\rvec{s}
\mid
\mathcal F_T
\right]
\right\|_2^2
=
\mathbb E_\Xi
\left\|
\mathbb E_\Xi
\left[
\rvec{m}_T
\mid
\mathcal F_T
\right]
\right\|_2^2
\le
\mathbb E_\Xi
\left\|
\rvec{m}_T
\right\|_2^2
\le
C_{\mathrm{en}}
T\log(ek).
\end{aligned}
\]
Consequently,
\[
\mathbb E_\Xi
\left\|
\mathbb E_\Xi
\left[
\rvec{s}^{\mathrm{sc}}
\mid
\mathcal F_T
\right]
\right\|_2^2
\le
C_{\mathrm{en}}
\frac{T\log(ek)}{d}.
\]

Thus the same posterior-mean energy bound holds for the value history
actually observed by the algorithm.  The lower-bound proof nevertheless
continues to condition on \(\mathcal G_T\), because the conditional
residual analysis relies on the product posterior structure available
after the active indices are revealed.

The posterior-mean estimate does not by itself control the centered
posterior residual in a direction selected from the transcript.  Define $\rvec{r}_T
:=
\rvec{s}^{\mathrm{sc}}
-
\rvec{m}^{\mathrm{sc}}_T$. 
Conditionally on \(\mathcal G_T\), this residual is a sum of independent
centered posterior blocks.  The following lemma gives the required
subgaussian tail bound.
\begin{lemma}[Conditional residual tail]
\label{lem:conditional-residual-tail}
Define $\rvec{r}_T
:=
\rvec{s}^{\mathrm{sc}}
-
\mathbb E_\Xi
\left[
\rvec{s}^{\mathrm{sc}}
\mid
\mathcal G_T
\right]
=
\rvec{s}^{\mathrm{sc}}
-
\rvec{m}^{\mathrm{sc}}_T$. 
Let $\rvec{v}_T$ 
be any
\(\mathcal G_T\)-measurable random vector taking values in
\(B_2^d\).  Then, for every \(u>0\),
\begin{equation}
\mathbb P_\Xi
\left(
\left\langle
\rvec{r}_T,
\rvec{v}_T
\right\rangle
\le
-u
\;\middle|\;
\mathcal G_T
\right)
\le
\exp\left(
-\frac{du^2}{2k}
\right)
\qquad
\mathbb P_\Xi\text{-almost surely}.
\label{eq:conditional-residual-tail}
\end{equation}
Consequently,
\begin{equation}
\mathbb P_\Xi
\left(
\left\langle
\rvec{r}_T,
\rvec{v}_T
\right\rangle
\le
-u
\right)
\le
\exp\left(
-\frac{du^2}{2k}
\right).
\label{eq:unconditional-residual-tail}
\end{equation}
\end{lemma}
\begin{proof}[Proof of Lemma~\ref{lem:conditional-residual-tail}]

By Lemma~\ref{lem:regular-product-posterior}, there exist probability
kernels
\[
\mu_{i,T}:
\mathsf X_T
\times
\mathcal B(\mathbb R^d)
\to
[0,1],
\qquad
i=1,\ldots,k,
\]
such that
\[
K_T
\left(
\tau,
\mathrm  d\mathbf b_1,\cdots ,\mathrm  d\mathbf b_k
\right)
:=
\bigotimes_{i=1}^k
\mu_{i,T}(\tau,\mathrm  d\mathbf b_i)
\]
is a regular conditional distribution of $(\rvec b_1,\ldots,\rvec b_k)$ 
given \(\mathsf T_T\).

Moreover, there exists a Borel set $\mathsf X_T^{\mathrm{post}}
\subseteq
\mathsf X_T$ 
such that $\mathbb P_{\mathsf T_T}
\left(
\mathsf X_T^{\mathrm{post}}
\right)
=
1$ 
and, for every $\tau\in
\mathsf X_T^{\mathrm{post}}$ 
and every \(i\in[k]\), the measure $\mu_{i,T}(\tau,\cdot)$ 
is \(1\)-strongly log-concave relative to an affine support.  
For
\(\tau\in\mathsf X_T\), define
\[
\overline{\mathbf b}_i(\tau)
:=
\int_{\mathbb R^d}
\mathbf b\,
\mathbf 1_{\{
\|\mathbf b\|_2\le2\sqrt d
\}}
\mu_{i,T}(\tau,\mathrm d\mathbf b).
\]
The integrand is bounded. Hence parameterized integration shows that $\tau
\longmapsto
\overline{\mathbf b}_i(\tau)$ 
is Borel measurable on all of
\(\mathsf X_T\).
For $\tau\in\mathsf X_T^{\mathrm{post}},$ 
the posterior factor is supported on $\{
\|\mathbf b\|_2\le2\sqrt d
\}$,  
so the preceding integral is its actual posterior mean.
Consequently,
\begin{equation}\label{eq:posterior-block-mean-kernel}
\overline{\mathbf b}_i(\mathsf T_T)
=
\mathbb E_\Xi
\left[
\rvec b_i
\mid
\mathcal G_T
\right]
\qquad
\mathbb P_\Xi\text{-almost surely}.
\end{equation}
Since $\mathcal G_T
=
\sigma(\mathsf T_T)$ 
and \(\rvec v_T\) is
\(\mathcal G_T\)-measurable with values in \(B_2^d\), the Doob--Dynkin
lemma gives a Borel map $\mathbf v:
\mathsf X_T
\to
B_2^d$ 
such that
\begin{equation}
\rvec v_T
=
\mathbf v(\mathsf T_T)
\qquad
\mathbb P_\Xi\text{-almost surely}.
\label{eq:direction-doob-dynkin}
\end{equation}
In particular,
\[
\|\mathbf v(\tau)\|_2
\le
1
\qquad
\text{for every }\tau\in\mathsf X_T.
\]
Fix
\(\tau\in\mathsf X_T^{\mathrm{post}}\)
and \(i\in[k]\).
Then
\(\mu_{i,T}(\tau,\cdot)\)
is \(1\)-strongly log-concave relative to its affine support, and
\(\overline{\mathbf b}_i(\tau)\)
is its mean.  Hence
Lemma~\ref{lem:slc-subgaussian-affine} gives, for every
\(a\in\mathbb R\),
\begin{equation}
\begin{aligned}
\int_{\mathbb R^d}
\exp\left(
a
\left\langle
\mathbf b
-
\overline{\mathbf b}_i(\tau),
\mathbf v(\tau)
\right\rangle
\right)
\mu_{i,T}(\tau,\mathrm d\mathbf b)
\le
\exp\left(
\frac{a^2}{2}
\|\mathbf v(\tau)\|_2^2
\right)
\le
\exp\left(
\frac{a^2}{2}
\right).
\end{aligned}
\label{eq:conditional-block-mgf}
\end{equation}
If the affine support is zero-dimensional, the centered block is
identically zero, and the same inequality remains valid. 
By linearity of conditional expectation,
\[
\begin{aligned}
\rvec r_T
=
\rvec s^{\mathrm{sc}}
-
\mathbb E_\Xi
\left[
\rvec s^{\mathrm{sc}}
\mid
\mathcal G_T
\right]
=
\frac1{\sqrt d}
\sum_{i=1}^k
\left(
\rvec b_i
-
\mathbb E_\Xi[
\rvec b_i\mid\mathcal G_T]
\right).
\end{aligned}
\]
Therefore,
\begin{equation}
\left\langle
\rvec r_T,
\rvec v_T
\right\rangle
=
\frac1{\sqrt d}
\sum_{i=1}^k
\left\langle
\rvec b_i
-
\mathbb E_\Xi[
\rvec b_i\mid\mathcal G_T],
\rvec v_T
\right\rangle
\qquad
\mathbb P_\Xi\text{-almost surely}.
\label{eq:residual-block-decomposition}
\end{equation}
Fix $\lambda\in\mathbb R$. 
Define, for
\(\tau\in\mathsf X_T\),
\[
\begin{aligned}
M_\lambda(\tau)
:=
\int_{(\mathbb R^d)^k}
\exp\Biggl(
\frac{\lambda}{\sqrt d}
\sum_{i=1}^k
\left\langle
\mathbf b_i
-
\overline{\mathbf b}_i(\tau),
\mathbf v(\tau)
\right\rangle
\Biggr)
K_T
\left(
\tau,
\mathrm  d\mathbf b_1,\cdots, \mathrm  d\mathbf b_k
\right).
\end{aligned}
\]
Since
\(\mathbb P_{\mathsf T_T}
(\mathsf X_T^{\mathrm{post}})=1\),
it suffices to bound
\(M_\lambda(\tau)\)
for
\(\tau\in\mathsf X_T^{\mathrm{post}}\). 
Since \(K_T\) is the product kernel, Tonelli's theorem gives, for every
\(\tau\in\mathsf X_T^{\mathrm{post}}\),
\[
\begin{aligned}
M_\lambda(\tau)
&=
\prod_{i=1}^k
\int_{\mathbb R^d}
\exp\left(
\frac{\lambda}{\sqrt d}
\left\langle
\mathbf b
-
\overline{\mathbf b}_i(\tau),
\mathbf v(\tau)
\right\rangle
\right)
\mu_{i,T}(\tau,\mathrm  d\mathbf b).
\end{aligned}
\]
Applying
\eqref{eq:conditional-block-mgf}
with $a
=
\frac{\lambda}{\sqrt d}$ 
to each factor yields
\[
\begin{aligned}
M_\lambda(\tau)
\le
\prod_{i=1}^k
\exp\left(
\frac{\lambda^2}{2d}
\right)
=
\exp\left(
\frac{\lambda^2k}{2d}
\right).
\end{aligned}
\]
Using
\eqref{eq:posterior-block-mean-kernel},
\eqref{eq:direction-doob-dynkin},
\eqref{eq:residual-block-decomposition},
and the regular conditional distribution identity for \(K_T\), we obtain
\begin{equation}
\begin{aligned}
\mathbb E_\Xi
\left[
\exp\left(
\lambda
\left\langle
\rvec r_T,
\rvec v_T
\right\rangle
\right)
\middle|
\mathcal G_T
\right]
=
M_\lambda(\mathsf T_T)
\le
\exp\left(
\frac{\lambda^2k}{2d}
\right)
\qquad
\mathbb P_\Xi\text{-almost surely}.
\end{aligned}
\label{eq:conditional-residual-mgf}
\end{equation}
Let \(u>0\).  For every \(\lambda>0\), conditional Markov's inequality,
applied to $\exp\left(
-\lambda
\left\langle
\rvec r_T,
\rvec v_T
\right\rangle
\right)$,  
gives
\[
\begin{aligned}
&
\mathbb P_\Xi
\left(
\left\langle
\rvec r_T,
\rvec v_T
\right\rangle
\le
-u
\;\middle|\;
\mathcal G_T
\right)
\\
&\qquad=
\mathbb P_\Xi
\left(
\exp\left(
-\lambda
\left\langle
\rvec r_T,
\rvec v_T
\right\rangle
\right)
\ge
e^{\lambda u}
\;\middle|\;
\mathcal G_T
\right)
\\
&\qquad\le
e^{-\lambda u}
\mathbb E_\Xi
\left[
\exp\left(
-\lambda
\left\langle
\rvec r_T,
\rvec v_T
\right\rangle
\right)
\middle|
\mathcal G_T
\right]
\\
&\qquad\le
\exp\left(
-\lambda u
+
\frac{\lambda^2k}{2d}
\right),
\end{aligned}
\]
where the final inequality follows from
\eqref{eq:conditional-residual-mgf}
with \(-\lambda\) in place of \(\lambda\). 
The exponent is minimized at $\lambda
=
\frac{du}{k}$.
Substituting this value gives
\[
\mathbb P_\Xi
\left(
\left\langle
\rvec r_T,
\rvec v_T
\right\rangle
\le
-u
\;\middle|\;
\mathcal G_T
\right)
\le
\exp\left(
-\frac{du^2}{2k}
\right)
\qquad
\mathbb P_\Xi\text{-almost surely}.
\]
This proves
\eqref{eq:conditional-residual-tail}.
Finally, taking expectations and using the tower property,
\[
\begin{aligned}
\mathbb P_\Xi
\left(
\left\langle
\rvec r_T,
\rvec v_T
\right\rangle
\le
-u
\right)=
\mathbb E_\Xi
\left[
\mathbb P_\Xi
\left(
\left\langle
\rvec r_T,
\rvec v_T
\right\rangle
\le
-u
\;\middle|\;
\mathcal G_T
\right)
\right]
\le
\exp\left(
-\frac{du^2}{2k}
\right).
\end{aligned}
\]
This proves
\eqref{eq:unconditional-residual-tail}.
\end{proof}

\subsection{Proof of Theorem~\ref{thm:one-step}}
\label{sec:one-step}

We prove Theorem~\ref{thm:one-step} in two stages.  First, we establish
the estimate for smooth full-dimensional block laws whose residual
potentials have finite global Hessian bounds.  These upper Hessian bounds
are used only to justify differentiation of Gaussian-fiber marginals;
all quantitative estimates are uniform in their numerical values.  This
gives Proposition~\ref{prop:one-step-smooth}.

Second, we pass to general \(1\)-strongly log-concave laws on affine
supports by Gaussian smoothing and deterministic priority perturbations.
The Gaussian perturbation produces smooth full-dimensional laws with
finite, though possibly large, residual Hessian bounds.  Since the smooth
estimate is uniform in those bounds, posterior-energy lower
semicontinuity permits passage to the limit.

Throughout this section, smooth refers to the probability laws of the
hidden blocks, not to the optimization objectives.  In particular, the
hard objectives $f_\Xi(\mathbf x)
=
\max_{1\le i\le k}
\left\langle
\rvec{a}_i,
\mathbf x
\right\rangle$ 
remain nonsmooth. 
We write
\[
\phi(t)
:=
\frac1{\sqrt{2\pi}}e^{-t^2/2},
\qquad
\Phi(t)
:=
\int_{-\infty}^t\phi(s)\,\mathrm  ds
\]
for the standard Gaussian density and distribution function.

\subsubsection{Smooth Full-Dimensional Case}

We first prove the one-step estimate for smooth full-dimensional block
laws with bounded residual Hessians.  The upper Hessian bounds enter only
through
Lemmas~\ref{lem:smooth-marginal-regularity}
and~\ref{lem:gaussian-fiber-domination}; the final constant is independent
of their numerical values.

\begin{proposition}[One-step selection, smooth bounded-Hessian case]
\label{prop:one-step-smooth}
Let $\rvec{u}_1,\ldots,\rvec{u}_k$ 
be independent random vectors with positive \(C^2\) densities
\[
p_i(\mathbf u)
=
\mathcal Z_i^{-1}
\exp\left(
-\frac12\|\mathbf u\|_2^2
-
Q_i(\mathbf u)
\right),
\qquad
\mathbf u\in\mathbb R^d,
\]
where
\(Q_i\in C^2(\mathbb R^d)\)
is convex and, for some finite \(L_i\),
\[
\mathbf O_d
\preceq
\nabla^2Q_i(\mathbf u)
\preceq
L_i\mathbf I_d
\qquad
\text{for every }\mathbf u\in\mathbb R^d.
\]
Fix $\mathbf v\in\mathbb S^{d-1}$ 
and define
\[
\rsc{R}_i
:=
\langle\rvec{u}_i,\mathbf v\rangle,
\quad
\rsc{Y}
:=
\max_{1\le i\le k}\rsc{R}_i,\quad \rsc{I}
:=
\iota((\rsc{R}_i)_{i=1}^k),
\quad\text{and} \quad
\rvec{s}
:=
\sum_{i=1}^k\rvec{u}_i.
\]
Then
\[
\mathbb E
\left\|
\mathbb E[
\rvec{s}\mid\rsc{I},\rsc{Y}]
-
\mathbb E\rvec{s}
\right\|_2^2
\le
C\log(ek),
\]
where \(C>0\) is universal and independent of
\(L_1,\ldots,L_k\).
\end{proposition}
\begin{lemma}[One-dimensional max-selection estimate]
\label{lem:hazard-detailed}
Let $\rsc{R}_1,\ldots,\rsc{R}_k$ 
be independent real-valued random variables whose distribution functions $F_i(y)
:=
\mathbb P(\rsc{R}_i\le y)$ 
are continuous and strictly increasing.  Define
\[
\rsc{Y}
:=
\max_{1\le i\le k}\rsc{R}_i,
\quad
\rsc{I}
:=
\iota((\rsc{R}_i)_{i=1}^k),\quad\text{and} \quad z_i(y)
:=
\Phi^{-1}(F_i(y)). 
\]
Then
\begin{equation}
\mathbb E
\left[
\left(
1+
|z_{\rsc{I}}(\rsc{Y})|
+
\sum_{\ell\ne\rsc{I}}
\frac{
\phi(z_\ell(\rsc{Y}))
}{
\Phi(z_\ell(\rsc{Y}))
}
\right)^2
\right]
\le
C\log(ek).
\label{eq:hazard-detailed-conclusion}
\end{equation}
\end{lemma}

\begin{proof}[Proof of Lemma~\ref{lem:hazard-detailed}]
Since all summands are nonnegative,
\[
\sum_{\ell\ne\rsc{I}}
\frac{
\phi(z_\ell(\rsc{Y}))
}{
\Phi(z_\ell(\rsc{Y}))
}
\le
\sum_{\ell=1}^k
\frac{
\phi(z_\ell(\rsc{Y}))
}{
\Phi(z_\ell(\rsc{Y}))
}.
\]
Define $\rsc{U}_\ell
:=
F_\ell(\rsc{Y})$ for $\ell\in[k]$,  
and
\[
\Lambda(u)
:=
\frac{
\phi(\Phi^{-1}(u))
}{
u
},
\qquad
u\in(0,1).
\]
Since every \(F_\ell\) is strictly increasing, $0
<
\rsc{U}_\ell
<
1$  almost surely.  
Moreover,
\[
\frac{
\phi(z_\ell(\rsc{Y}))
}{
\Phi(z_\ell(\rsc{Y}))
}
=
\Lambda(\rsc{U}_\ell).
\]
The distribution function of $\rsc{Y}
=
\max_{1\le i\le k}\rsc{R}_i$
is
\[
H(y)
:=
\mathbb P(\rsc{Y}\le y)
=
\prod_{\ell=1}^k
F_\ell(y),
\]
by independence.  Since \(H\) is continuous,
Lemma~\ref{lem:PIT} gives $H(\rsc{Y})
\sim
\operatorname{Unif}(0,1)$. 
Consequently,
\[
\rsc{L}
:=
-\log H(\rsc{Y})
=
\sum_{\ell=1}^k
-\log\rsc{U}_\ell
\sim
\operatorname{Exp}(1).
\]

Define the random index sets
\[
\mathcal I_-
:=
\left\{
\ell\in[k]:
\rsc{U}_\ell\le\frac12
\right\},
\qquad
\mathcal I_+
:=
\left\{
\ell\in[k]:
\rsc{U}_\ell>\frac12
\right\}.
\]

For $\ell\in\mathcal I_-$, 
the first estimate in
Lemma~\ref{lem:normal-hazard} gives
\[
\Lambda(\rsc{U}_\ell)
\le
C
\sqrt{
\log\frac1{\rsc{U}_\ell}
}.
\]
Since $\log\frac1{\rsc{U}_\ell}
\ge
\log2$,  
we have
\[
\sqrt{
\log\frac1{\rsc{U}_\ell}
}
\le
C
\log\frac1{\rsc{U}_\ell}.
\]
Therefore
\begin{equation}
\begin{aligned}
\sum_{\ell\in\mathcal I_-}
\Lambda(\rsc{U}_\ell)
\le
C
\sum_{\ell\in\mathcal I_-}
\log\frac1{\rsc{U}_\ell}
\le
C\rsc{L}.
\end{aligned}
\label{eq:hazard-small-u}
\end{equation}

For
\(\ell\in\mathcal I_+\), define $\rsc{D}_\ell
:=
1-\rsc{U}_\ell
\in
\left(
0,\frac12
\right)$,  
and set $\rsc{\Delta}
:=
\sum_{\ell\in\mathcal I_+}
\rsc{D}_\ell$.  
The second estimate in
Lemma~\ref{lem:normal-hazard}, applied with $\delta
=
\rsc{D}_\ell$, 
gives
\[
\Lambda(\rsc{U}_\ell)
\le
C
\rsc{D}_\ell
\sqrt{
\log\frac e{\rsc{D}_\ell}
}.
\]

If $\rsc{\Delta}=0,$ 
then
\(\mathcal I_+=\varnothing\)
and the corresponding sum is zero.  Suppose henceforth that $\rsc{\Delta}>0$. 
By Cauchy--Schwarz,
\begin{equation}\label{eq:hazard-large-u-cauchy1}
\begin{aligned}
\sum_{\ell\in\mathcal I_+}
\rsc{D}_\ell
\sqrt{
\log\frac e{\rsc{D}_\ell}
}
\le
\left(
\sum_{\ell\in\mathcal I_+}
\rsc{D}_\ell
\right)^{1/2}
\left(
\sum_{\ell\in\mathcal I_+}
\rsc{D}_\ell
\log\frac e{\rsc{D}_\ell}
\right)^{1/2}.
\end{aligned}
\end{equation}

Define $p_\ell
:=
\frac{
\rsc{D}_\ell
}{
\rsc{\Delta}
}$ for  $\ell\in\mathcal I_+$. 
Then $\sum_{\ell\in\mathcal I_+}
p_\ell
=
1.$
Moreover,
\[
\begin{aligned}
\sum_{\ell\in\mathcal I_+}
\rsc{D}_\ell
\log\frac e{\rsc{D}_\ell}=
\rsc{\Delta}
\left[
1
+
\log\frac1{\rsc{\Delta}}
+
\sum_{\ell\in\mathcal I_+}
p_\ell
\log\frac1{p_\ell}
\right].
\end{aligned}
\]
By Lemma~\ref{lem:entropy},
\[
\sum_{\ell\in\mathcal I_+}
p_\ell
\log\frac1{p_\ell}
\le
\log|\mathcal I_+|
\le
\log k.
\]
Hence
\[
\sum_{\ell\in\mathcal I_+}
\rsc{D}_\ell
\log\frac e{\rsc{D}_\ell}
\le
\rsc{\Delta}
\log\frac{ek}{\rsc{\Delta}}.
\]
Combining this estimate with
\eqref{eq:hazard-large-u-cauchy1}
gives
\[
\sum_{\ell\in\mathcal I_+}
\rsc{D}_\ell
\sqrt{
\log\frac e{\rsc{D}_\ell}
}
\le
\rsc{\Delta}
\sqrt{
\log\frac{ek}{\rsc{\Delta}}
}.
\]
Therefore
\begin{equation}
\sum_{\ell\in\mathcal I_+}
\Lambda(\rsc{U}_\ell)
\le
C
\rsc{\Delta}
\sqrt{
\log\frac{ek}{\rsc{\Delta}}
}.
\label{eq:hazard-large-u}
\end{equation}
By Lemma~\ref{lem:one-minus-log},
\[
\rsc{D}_\ell
=
1-\rsc{U}_\ell
\le
-\log\rsc{U}_\ell.
\]
Thus
\[
\rsc{\Delta}
\le
\sum_{\ell\in\mathcal I_+}
-\log\rsc{U}_\ell
\le
\rsc{L}.
\]
Lemma~\ref{lem:delta-L}, applied pointwise with $\Delta
=
\rsc{\Delta}$ and $L
=
\rsc{L}$,  
therefore gives
\[
\rsc{\Delta}
\sqrt{
\log\frac{ek}{\rsc{\Delta}}
}
\le
C
(1+\rsc{L})
\sqrt{\log(ek)}.
\]
Together with
\eqref{eq:hazard-large-u},
\[
\sum_{\ell\in\mathcal I_+}
\Lambda(\rsc{U}_\ell)
\le
C
(1+\rsc{L})
\sqrt{\log(ek)}.
\]
Combining this with
\eqref{eq:hazard-small-u},
we obtain
\[
\sum_{\ell=1}^k
\Lambda(\rsc{U}_\ell)
\le
C
(1+\rsc{L})
\sqrt{\log(ek)}.
\]
Since $\rsc{L}
\sim
\operatorname{Exp}(1)$, 
we have $\mathbb E(1+\rsc{L})^2
<
\infty$. 
Consequently,
\begin{equation}
\mathbb E
\left[
\left(
\sum_{\ell=1}^k
\frac{
\phi(z_\ell(\rsc{Y}))
}{
\Phi(z_\ell(\rsc{Y}))
}
\right)^2
\right]
\le
C\log(ek).
\label{eq:hazard-sum-second-moment}
\end{equation}

It remains to control $z_{\rsc{I}}(\rsc{Y})$. 
For every \(r\ge0\),
\[
\begin{aligned}
\mathbb P
\left(
z_{\rsc{I}}(\rsc{Y})
\ge
r
\right)
&=
\sum_{i=1}^k
\mathbb P
\left(
\rsc{I}=i,\,
z_i(\rsc{Y})\ge r
\right)
\\
&=
\sum_{i=1}^k
\mathbb P
\left(
\rsc{I}=i,\,
z_i(\rsc{R}_i)\ge r
\right)
\\
&\le
\sum_{i=1}^k
\mathbb P
\left(
z_i(\rsc{R}_i)\ge r
\right).
\end{aligned}
\]
By Lemma~\ref{lem:PIT}, $F_i(\rsc{R}_i)
\sim
\operatorname{Unif}(0,1)$,  
and hence
\[
z_i(\rsc{R}_i)
=
\Phi^{-1}
\left(
F_i(\rsc{R}_i)
\right)
\sim
N(0,1).
\]
Therefore
\[
\mathbb P
\left(
z_{\rsc{I}}(\rsc{Y})
\ge r
\right)
\le
k\Phi(-r).
\]
The same argument gives
\[
\mathbb P
\left(
z_{\rsc{I}}(\rsc{Y})
\le-r
\right)
\le
k\Phi(-r).
\]
Thus
\[
\mathbb P
\left(
\left|
z_{\rsc{I}}(\rsc{Y})
\right|
\ge r
\right)
\le
\min\left\{
1,
2k\Phi(-r)
\right\}.
\]
Lemma~\ref{lem:gaussian-max-moment} now yields
\begin{equation}
\mathbb E
\left[
z_{\rsc{I}}(\rsc{Y})^2
\right]
\le
C\log(ek).
\label{eq:selected-quantile-second-moment}
\end{equation}

Finally, using
\[
(a+b+c)^2
\le
3a^2+3b^2+3c^2,
\]
together with
\eqref{eq:hazard-sum-second-moment}
and
\eqref{eq:selected-quantile-second-moment},
proves
\eqref{eq:hazard-detailed-conclusion}.
\end{proof}

\begin{proof}[Proof of Proposition~\ref{prop:one-step-smooth}]
For every \(i\in[k]\), define the deterministic mean vector $\overline{\mathbf u}_i
:=
\mathbb E\rvec{u}_i$. 
Since $\rvec{s}
=
\sum_{i=1}^k
\rvec{u}_i$,  
we have $\mathbb E\rvec{s}
=
\sum_{i=1}^k
\overline{\mathbf u}_i$.   
Recall that $\rsc{R}_i
=
\left\langle
\rvec{u}_i,
\mathbf v
\right\rangle$.  
The scalar random variable
\(\rsc{R}_i\)
has density
\[
\varphi_i(y)
:=
\int_{\mathbf v^\perp}
p_i(y\mathbf v+\mathbf r)
\,\mathrm  d\sigma_{\mathbf v^\perp}(\mathbf r),
\qquad
y\in\mathbb R,
\]
and distribution function $F_i(y)
:=
\mathbb P(
\rsc{R}_i\le y)$.  
The positivity of \(p_i\) implies $\varphi_i(y)>0$ for every  $y\in\mathbb R$.   
Moreover, the bounded-Hessian assumption, together with
Lemmas~\ref{lem:smooth-marginal-regularity}
and~\ref{lem:gaussian-fiber-domination}, implies that
\(\varphi_i\) is continuous.  Consequently, $0
<
F_i(y)
<
1$ for every  $y\in\mathbb R$.  
Define $z_i(y)
:=
\Phi^{-1}(F_i(y))$.  
For every \(y\in\mathbb R\), define the canonical winner-slice
regression vector by
\begin{equation}
\mathbf b_i^{=}(y)
:=
\frac{
\displaystyle
\int_{\mathbf v^\perp}
\left(
y\mathbf v
+
\mathbf r
-
\overline{\mathbf u}_i
\right)
p_i(y\mathbf v+\mathbf r)
\,\mathrm  d\sigma_{\mathbf v^\perp}(\mathbf r)
}{
\displaystyle
\varphi_i(y)
}.
\label{eq:smooth-winner-regression-definition}
\end{equation}
This is the slice-density version of $\mathbb E
\left[
\rvec{u}_i-\overline{\mathbf u}_i
\mid
\rsc{R}_i=y
\right]$. 
Unlike an abstract conditional expectation at a point, the right-hand
side of
\eqref{eq:smooth-winner-regression-definition}
is canonically defined for every \(y\in\mathbb R\). 
Define also the loser-tail regression vector by
\begin{equation}
\mathbf b_i^{\le}(y)
:=
\frac{
\displaystyle
\mathbb E
\left[
\left(
\rvec{u}_i-\overline{\mathbf u}_i
\right)
\mathbf 1_{\{\rsc{R}_i\le y\}}
\right]
}{
F_i(y)
}.
\label{eq:smooth-loser-regression-definition}
\end{equation}
Equivalently, $\mathbf b_i^{\le}(y)
=
\mathbb E
\left[
\rvec{u}_i-\overline{\mathbf u}_i
\mid
\rsc{R}_i\le y
\right]$.  
Applying
Theorem~\ref{thm:GSR}
to
\(\rvec{u}_i\)
in the direction \(\mathbf v\) gives
\begin{equation}
\left\|
\mathbf b_i^{=}(y)
\right\|_2
\le
C
\left(
1+
|z_i(y)|
\right),
\qquad
y\in\mathbb R.
\label{eq:smooth-winner-regression}
\end{equation}
Similarly,
Lemma~\ref{lem:loser-tail}
gives
\begin{equation}
\left\|
\mathbf b_i^{\le}(y)
\right\|_2
\le
C
\frac{
\phi(z_i(y))
}{
\Phi(z_i(y))
},
\qquad
y\in\mathbb R.
\label{eq:smooth-loser-regression}
\end{equation} 
The random variables $\rsc{R}_1,\ldots,\rsc{R}_k$ 
are independent and have continuous densities.  Hence ties occur with
probability zero, and
\[
\rsc{I}
=
\argmax_{1\le i\le k}
\rsc{R}_i
\qquad
\text{almost surely}.
\]
For \(i\in[k]\), the joint law of
\((\rsc{I},\rsc{Y})\)
on
\(\{i\}\times\mathbb R\)
has density
\begin{equation}
q_i(y)
:=
\varphi_i(y)
\prod_{\ell\ne i}
F_\ell(y).
\label{eq:smooth-IY-density}
\end{equation}
Indeed, block \(i\) attains the maximum at level \(y\) precisely when
\(\rsc{R}_i\in \mathrm dy\) and all other projections are at most \(y\). 
For every Borel set
\(E\subseteq\mathbb R^d\), define the canonical winner-slice kernel
\[
\nu_{i,y}^{=}(E)
:=
\frac{
\displaystyle
\int_{\mathbf v^\perp}
\mathbf 1_E(y\mathbf v+\mathbf r)
p_i(y\mathbf v+\mathbf r)
\,\mathrm d\sigma_{\mathbf v^\perp}(\mathbf r)
}{
\displaystyle
\varphi_i(y)
}.
\]
For every
\(\ell\ne i\), define the loser-tail kernel
\[
\nu_{\ell,y}^{\le}(E)
:=
\frac{
\mathbb P
\left(
\rvec{u}_\ell\in E,\,
\rsc{R}_\ell\le y
\right)
}{
F_\ell(y)
}.
\]
Both are probability kernels in \(y\).

We next identify the conditional law of the blocks given the selected
index and maximum value.  Let $E
=
E_1\times\cdots\times E_k$ 
be a product rectangle in
\((\mathbb R^d)^k\), and let
\(A\subseteq\mathbb R\)
be Borel.  Independence gives
\[
\begin{aligned}
&
\mathbb P
\left(
\rvec{u}_j\in E_j
\text{ for all }j,\,
\rsc{I}=i,\,
\rsc{Y}\in A
\right)
\\
&\quad=
\int_A
\nu_{i,y}^{=}(E_i)
\prod_{\ell\ne i}
\nu_{\ell,y}^{\le}(E_\ell)
\,
\varphi_i(y)
\prod_{\ell\ne i}
F_\ell(y)
\,\mathrm dy
\\
&\quad=
\int_A
\left[
\nu_{i,y}^{=}
\otimes
\bigotimes_{\ell\ne i}
\nu_{\ell,y}^{\le}
\right](E)
q_i(y)
\,\mathrm dy.
\end{aligned}
\]
A monotone-class argument extends this identity from product rectangles
to every Borel set $E\subseteq(\mathbb R^d)^k$.  
Consequently,
\begin{equation}
\Law
\left(
\rvec{u}_1,\ldots,\rvec{u}_k
\mid
\rsc{I}=i,\rsc{Y}=y
\right)
=
\nu_{i,y}^{=}
\otimes
\bigotimes_{\ell\ne i}
\nu_{\ell,y}^{\le}
\label{eq:smooth-selection-product-law}
\end{equation}
for
\(q_i(y)\,\mathrm dy\)-almost every \(y\), and hence for
\(\mathbb P_{(\rsc{I},\rsc{Y})}\)-almost every \((i,y)\).

By the definitions of the regression vectors,
\[
\int_{\mathbb R^d}
\left(
\mathbf u-\overline{\mathbf u}_i
\right)
\nu_{i,y}^{=}(\mathrm d\mathbf u)
=
\mathbf b_i^{=}(y),
\]
and, for every
\(\ell\ne i\),
\[
\int_{\mathbb R^d}
\left(
\mathbf u-\overline{\mathbf u}_\ell
\right)
\nu_{\ell,y}^{\le}(\mathrm d\mathbf u)
=
\mathbf b_\ell^{\le}(y).
\]
Therefore
\eqref{eq:smooth-selection-product-law}
implies
\begin{equation}
\mathbb E
\left[
\rvec{s}
\mid
\rsc{I}=i,\rsc{Y}=y
\right]
-
\mathbb E\rvec{s}
=
\mathbf b_i^{=}(y)
+
\sum_{\ell\ne i}
\mathbf b_\ell^{\le}(y)
\label{eq:smooth-conditional-sum-regression}
\end{equation}
for
\(\mathbb P_{(\rsc{I},\rsc{Y})}\)-almost every \((i,y)\).

Taking norms in
\eqref{eq:smooth-conditional-sum-regression}
and using
\eqref{eq:smooth-winner-regression}
and
\eqref{eq:smooth-loser-regression},
we obtain
\[
\begin{aligned}
&
\left\|
\mathbb E
\left[
\rvec{s}
\mid
\rsc{I},\rsc{Y}
\right]
-
\mathbb E\rvec{s}
\right\|_2
\\
&\qquad\le
\left\|
\mathbf b_{\rsc{I}}^{=}(\rsc{Y})
\right\|_2
+
\sum_{\ell\ne\rsc{I}}
\left\|
\mathbf b_\ell^{\le}(\rsc{Y})
\right\|_2
\\
&\qquad\le
C
\left(
1
+
\left|
z_{\rsc{I}}(\rsc{Y})
\right|
+
\sum_{\ell\ne\rsc{I}}
\frac{
\phi(z_\ell(\rsc{Y}))
}{
\Phi(z_\ell(\rsc{Y}))
}
\right)
\qquad
\text{almost surely}.
\end{aligned}
\]
Squaring, taking expectations, and applying
Lemma~\ref{lem:hazard-detailed}
gives
\[
\mathbb E
\left\|
\mathbb E
\left[
\rvec{s}
\mid
\rsc{I},\rsc{Y}
\right]
-
\mathbb E\rvec{s}
\right\|_2^2
\le
C\log(ek).
\]
The constant \(C\) is universal and does not depend on
\(L_1,\ldots,L_k\).  This proves the proposition.
\end{proof}

\subsubsection{Closure to Affine Supports and Nonsmooth Posterior Laws}

We now complete the proof of
Theorem~\ref{thm:one-step}
by passing from the smooth bounded-Hessian result to general
\(1\)-strongly log-concave laws on affine supports.  Gaussian
perturbations produce smooth full-dimensional laws, while deterministic
priority perturbations resolve ties according to the fixed
smallest-index selector.

\begin{proof}[Completion of the proof of Theorem~\ref{thm:one-step}]
Retain the notation 
from Theorem~\ref{thm:one-step}.  Let $\mu_i
:=
\Law(\rvec u_i)$ for $i\in[k]$.  
The strong log-concavity assumption implies that every
\(\rvec u_i\) has finite moments of all orders; in particular, all
\(L^2\)-quantities below are finite. 
Let $\breve{\boldsymbol\zeta}_1,\ldots,
\breve{\boldsymbol\zeta}_k
\overset{\mathrm{i.i.d.}}{\sim}
N(\mathbf 0_d,\mathbf I_d)$ 
be independent of $\rvec u_1,\ldots,\rvec u_k$.  
Set $\kappa_i
:=
-i$ for  $i\in[k].$ 
For
\(\eta,\rho>0\), define
\[
\rvec u_i^{(\eta,\rho)}
:=
\rvec u_i
+
\eta\breve{\boldsymbol\zeta}_i
+
\rho\kappa_i\mathbf v,
\]
\[
\rsc R_i^{(\eta,\rho)}
:=
\left\langle
\rvec u_i^{(\eta,\rho)},
\mathbf v
\right\rangle,\  \rsc Y^{(\eta,\rho)}
:=
\max_{1\le i\le k}
\rsc R_i^{(\eta,\rho)},\  \rsc I^{(\eta,\rho)}
:=
\iota\left(
(\rsc R_i^{(\eta,\rho)})_{i=1}^k
\right),\  \rvec s^{(\eta,\rho)}
:=
\sum_{i=1}^k
\rvec u_i^{(\eta,\rho)}.
\]
For every \(i\in[k]\), the random vector
\(\rvec u_i^{(\eta,\rho)}\)
has density
\begin{equation}
p_{i,\eta,\rho}(\mathbf y)
:=
\int_{H_i}
\frac{1}{(2\pi\eta^2)^{d/2}}
\exp\left(
-\frac{
\|\mathbf y-\rho\kappa_i\mathbf v-\mathbf u\|_2^2
}{
2\eta^2
}
\right)
\mu_i(\mathrm d\mathbf u),
\qquad
\mathbf y\in\mathbb R^d.
\label{eq:closure-smoothed-density}
\end{equation}
Indeed, conditionally on
\(\rvec u_i=\mathbf u\),
the perturbed vector has law
\[
N\left(
\mathbf u+\rho\kappa_i\mathbf v,
\eta^2\mathbf I_d
\right),
\]
and
\eqref{eq:closure-smoothed-density}
is obtained by integrating the conditional Gaussian density with respect
to \(\mu_i\).

Since the Gaussian kernel is strictly positive and smooth,
\[
p_{i,\eta,\rho}
\in
C^\infty(\mathbb R^d),
\qquad
p_{i,\eta,\rho}(\mathbf y)>0.
\]
Define the total negative log-density
\[
\mathcal V_{i,\eta,\rho}(\mathbf y)
:=
-\log
p_{i,\eta,\rho}(\mathbf y).
\]
The potential used in
Lemma~\ref{lem:gaussian-smoothing-slc}
may differ from
\(\mathcal V_{i,\eta,\rho}\)
by an additive constant only, so the Hessian bounds from that lemma apply
unchanged.  Moreover, the deterministic translation
\(\rho\kappa_i\mathbf v\)
does not change Hessians.  Hence
\begin{equation}
\alpha_\eta\mathbf I_d
\preceq
\nabla^2
\mathcal V_{i,\eta,\rho}(\mathbf y)
\preceq
\frac1{\eta^2}\mathbf I_d,
\qquad
\alpha_\eta
:=
\frac1{1+\eta^2}.
\label{eq:closure-total-hessian}
\end{equation}

Define the normalized random vectors
\[
\rvec u_{i,\mathrm{nor}}^{(\eta,\rho)}
:=
\sqrt{\alpha_\eta}
\,\rvec u_i^{(\eta,\rho)}.
\]
Their densities satisfy the change-of-variables formula
\[
p_{i,\eta,\rho}^{\mathrm{nor}}(\mathbf w)
=
\alpha_\eta^{-d/2}
p_{i,\eta,\rho}
\left(
\frac{\mathbf w}{\sqrt{\alpha_\eta}}
\right).
\]
Define
\begin{equation}
Q_{i,\eta,\rho}(\mathbf w)
:=
\mathcal V_{i,\eta,\rho}
\left(
\frac{\mathbf w}{\sqrt{\alpha_\eta}}
\right)
-
\frac12\|\mathbf w\|_2^2.
\label{eq:closure-normalized-residual-potential}
\end{equation}
Then
\begin{equation}
p_{i,\eta,\rho}^{\mathrm{nor}}(\mathbf w)
=
\left(
\mathcal Z_{i,\eta,\rho}^{\mathrm{nor}}
\right)^{-1}
\exp\left(
-\frac12\|\mathbf w\|_2^2
-
Q_{i,\eta,\rho}(\mathbf w)
\right),
\label{eq:closure-normalized-density}
\end{equation}
where one may take
\[
\mathcal Z_{i,\eta,\rho}^{\mathrm{nor}}
=
\alpha_\eta^{d/2}.
\]
The numerical value of this normalizer is immaterial.

By the chain rule,
\[
\nabla^2Q_{i,\eta,\rho}(\mathbf w)
=
\frac1{\alpha_\eta}
\nabla^2
\mathcal V_{i,\eta,\rho}
\left(
\frac{\mathbf w}{\sqrt{\alpha_\eta}}
\right)
-
\mathbf I_d.
\]
The lower Hessian bound in
\eqref{eq:closure-total-hessian}
gives
\[
\nabla^2Q_{i,\eta,\rho}(\mathbf w)
\succeq
\mathbf O_d.
\]
The upper bound gives
\[
\begin{aligned}
\nabla^2Q_{i,\eta,\rho}(\mathbf w)
\preceq
\left(
\frac1{\alpha_\eta\eta^2}
-
1
\right)
\mathbf I_d
=
\left(
\frac{1+\eta^2}{\eta^2}
-
1
\right)
\mathbf I_d
=
\frac1{\eta^2}\mathbf I_d.
\end{aligned}
\]
Thus
\begin{equation}
\mathbf O_d
\preceq
\nabla^2Q_{i,\eta,\rho}(\mathbf w)
\preceq
\frac1{\eta^2}\mathbf I_d.
\label{eq:closure-normalized-residual-hessian}
\end{equation}
Consequently, the independent normalized blocks
\[
\rvec u_{1,\mathrm{nor}}^{(\eta,\rho)},
\ldots,
\rvec u_{k,\mathrm{nor}}^{(\eta,\rho)}
\]
satisfy all assumptions of
Proposition~\ref{prop:one-step-smooth}.

Define
\[
\rsc R_{i,\mathrm{nor}}^{(\eta,\rho)}
:=
\left\langle
\rvec u_{i,\mathrm{nor}}^{(\eta,\rho)},
\mathbf v
\right\rangle
=
\sqrt{\alpha_\eta}
\,\rsc R_i^{(\eta,\rho)},
\]
\[
\rsc Y_{\mathrm{nor}}^{(\eta,\rho)}
:=
\max_i
\rsc R_{i,\mathrm{nor}}^{(\eta,\rho)}
=
\sqrt{\alpha_\eta}
\,\rsc Y^{(\eta,\rho)},
\]
and
\[
\rvec s_{\mathrm{nor}}^{(\eta,\rho)}
:=
\sum_{i=1}^k
\rvec u_{i,\mathrm{nor}}^{(\eta,\rho)}
=
\sqrt{\alpha_\eta}
\,\rvec s^{(\eta,\rho)}.
\]
Positive scaling does not change the maximizing index, so
\[
\iota\left(
(\rsc R_{i,\mathrm{nor}}^{(\eta,\rho)})_{i=1}^k
\right)
=
\rsc I^{(\eta,\rho)}.
\]
Moreover,
\[
\sigma\left(
\rsc I^{(\eta,\rho)},
\rsc Y_{\mathrm{nor}}^{(\eta,\rho)}
\right)
=
\sigma\left(
\rsc I^{(\eta,\rho)},
\rsc Y^{(\eta,\rho)}
\right),
\]
because
\(\alpha_\eta>0\)
is deterministic.

Proposition~\ref{prop:one-step-smooth}
therefore gives
\[
\mathbb E
\left\|
\mathbb E
\left[
\rvec s_{\mathrm{nor}}^{(\eta,\rho)}
\mid
\rsc I^{(\eta,\rho)},
\rsc Y^{(\eta,\rho)}
\right]
-
\mathbb E
\rvec s_{\mathrm{nor}}^{(\eta,\rho)}
\right\|_2^2
\le
C\log(ek).
\]
Since $\rvec s_{\mathrm{nor}}^{(\eta,\rho)}
=
\sqrt{\alpha_\eta}
\,\rvec s^{(\eta,\rho)}$,  
we obtain
\begin{equation}
\mathbb E
\left\|
\mathbb E
\left[
\rvec s^{(\eta,\rho)}
\mid
\rsc I^{(\eta,\rho)},
\rsc Y^{(\eta,\rho)}
\right]
-
\mathbb E
\rvec s^{(\eta,\rho)}
\right\|_2^2
\le
\frac{C}{\alpha_\eta}
\log(ek).
\label{eq:closure-smooth-energy}
\end{equation}
For $0<\eta\le1$, 
we have $\alpha_\eta
=
\frac1{1+\eta^2}
\ge
\frac12$.  
Hence
\begin{equation}
\mathbb E
\left\|
\mathbb E
\left[
\rvec s^{(\eta,\rho)}
\mid
\rsc I^{(\eta,\rho)},
\rsc Y^{(\eta,\rho)}
\right]
-
\mathbb E
\rvec s^{(\eta,\rho)}
\right\|_2^2
\le
2C\log(ek).
\label{eq:closure-uniform-energy}
\end{equation}

Choose deterministic sequences
\[
\rho_n\downarrow0,
\qquad
\eta_n\downarrow0,
\qquad
\frac{\eta_n}{\rho_n}\longrightarrow0,
\qquad
\eta_n\le1.
\]
For example, one may take $\rho_n
=
n^{-1}$ and $\eta_n
=
n^{-2}$.

We first verify convergence of the hidden sum.  By construction,
\[
\rvec s^{(\eta_n,\rho_n)}
=
\rvec s
+
\eta_n
\sum_{i=1}^k
\breve{\boldsymbol\zeta}_i
+
\rho_n
\left(
\sum_{i=1}^k
\kappa_i
\right)
\mathbf v.
\]
Therefore
\begin{equation}
\rvec s^{(\eta_n,\rho_n)}
\longrightarrow
\rvec s
\qquad
\text{in }L^2.
\label{eq:closure-s-L2}
\end{equation}
For every \(i\in[k]\),
\[
\rsc R_i^{(\eta_n,\rho_n)}
=
\rsc R_i
+
\eta_n
\left\langle
\breve{\boldsymbol\zeta}_i,
\mathbf v
\right\rangle
+
\rho_n\kappa_i,
\]
and hence
\[
\rsc R_i^{(\eta_n,\rho_n)}
\longrightarrow
\rsc R_i
\qquad
\text{almost surely}.
\]
Since the maximum map on \(\mathbb R^k\) is continuous,
\begin{equation}
\rsc Y^{(\eta_n,\rho_n)}
\longrightarrow
\rsc Y
\qquad
\text{almost surely}.
\label{eq:closure-Y-as}
\end{equation}

We next prove convergence of the selected index.  Fix an outcome in the
probability-one event on which all preceding scalar convergences hold.
Set
\[
m
:=
\max_{1\le i\le k}
\rsc R_i,
\qquad
i_\star
:=
\min
\left\{
i:
\rsc R_i=m
\right\}.
\]

For every \(j\) satisfying
\(\rsc R_j<m\),
the gap $m-\rsc R_j$ 
is strictly positive.  Since there are finitely many indices and all
perturbations converge to zero, every such \(j\) remains strictly below
\(i_\star\) for all sufficiently large \(n\).

Now suppose that $j>i_\star$ and $\rsc R_j=m$.  
Then
\[
\begin{aligned}
&
\rsc R_{i_\star}^{(\eta_n,\rho_n)}
-
\rsc R_j^{(\eta_n,\rho_n)}
\\
&\quad=
\rho_n
\left(
\kappa_{i_\star}-\kappa_j
\right)
+
\eta_n
\left\langle
\breve{\boldsymbol\zeta}_{i_\star}
-
\breve{\boldsymbol\zeta}_j,
\mathbf v
\right\rangle
\\
&\quad=
\rho_n
(j-i_\star)
+
\eta_n
\left\langle
\breve{\boldsymbol\zeta}_{i_\star}
-
\breve{\boldsymbol\zeta}_j,
\mathbf v
\right\rangle.
\end{aligned}
\]
Dividing by \(\rho_n>0\) gives
\[
j-i_\star
+
\frac{\eta_n}{\rho_n}
\left\langle
\breve{\boldsymbol\zeta}_{i_\star}
-
\breve{\boldsymbol\zeta}_j,
\mathbf v
\right\rangle
\longrightarrow
j-i_\star
>
0.
\]
Thus the preceding difference is positive for all sufficiently large
\(n\).  Since no tied index is smaller than \(i_\star\), the perturbed
selector is eventually equal to \(i_\star\).  Therefore
\begin{equation}
\rsc I^{(\eta_n,\rho_n)}
\longrightarrow
\rsc I
\qquad
\text{almost surely}.
\label{eq:closure-I-as}
\end{equation}

Define the observation random elements
\[
\mathsf W_n
:=
\left(
\rsc I^{(\eta_n,\rho_n)},
\rsc Y^{(\eta_n,\rho_n)}
\right),
\qquad
\mathsf W
:=
(\rsc I,\rsc Y).
\]
They take values in the Polish space $[k]\times\mathbb R$,  
where \([k]\) has the discrete topology.  Equations
\eqref{eq:closure-Y-as}
and
\eqref{eq:closure-I-as}
give
\[
\mathsf W_n
\longrightarrow
\mathsf W
\qquad
\text{almost surely}.
\]

Applying
Lemma~\ref{lem:posterior-lsc}
with $\rvec u_n
=
\rvec s^{(\eta_n,\rho_n)}$,  $\rvec u
=
\rvec s$ 
and observations $\mathsf W_n$, $\mathsf W$,  
gives
\[
\begin{aligned}
&
\mathbb E
\left\|
\mathbb E
\left[
\rvec s
\mid
\rsc I,\rsc Y
\right]
-
\mathbb E\rvec s
\right\|_2^2
\\
&\quad\le
\liminf_{n\to\infty}
\mathbb E
\left\|
\mathbb E
\left[
\rvec s^{(\eta_n,\rho_n)}
\mid
\rsc I^{(\eta_n,\rho_n)},
\rsc Y^{(\eta_n,\rho_n)}
\right]
-
\mathbb E
\rvec s^{(\eta_n,\rho_n)}
\right\|_2^2.
\end{aligned}
\]
Using
\eqref{eq:closure-uniform-energy},
we conclude that
\[
\mathbb E
\left\|
\mathbb E
\left[
\rvec s
\mid
\rsc I,\rsc Y
\right]
-
\mathbb E\rvec s
\right\|_2^2
\le
2C\log(ek).
\]
Absorbing the numerical factor \(2\) into the universal constant proves
Theorem~\ref{thm:one-step}.
\end{proof}

\subsubsection{Winner-Slice Regression for Smooth Laws}
\label{sec:GSR}

The winner-slice estimate below is used in the smooth full-dimensional
proof of the one-step selection theorem.  The finite upper Hessian bound
is needed only to justify differentiability of Gaussian-fiber marginals;
the resulting estimate is uniform in its numerical value. 

\begin{theorem}[High-dimensional slice regression, smooth full-dimensional case]
\label{thm:GSR}
Let \(\rvec g\) have a positive \(C^2\) density
\[
p(\mathbf g)
=
\mathcal Z^{-1}
\exp\left(
-\frac12\|\mathbf g\|_2^2
-
Q(\mathbf g)
\right),
\qquad
\mathbf g\in\mathbb R^d,
\]
where \(Q\in C^2(\mathbb R^d)\) is convex and, for some finite \(L_Q\),
\[
\mathbf O_d
\preceq
\nabla^2Q(\mathbf g)
\preceq
L_Q\mathbf I_d
\qquad
\text{for every }\mathbf g\in\mathbb R^d.
\]
Fix $\mathbf v\in\mathbb S^{d-1}$, 
and define $\rsc R
:=
\langle\rvec g,\mathbf v\rangle$ and $F_R(y)
:=
\mathbb P(\rsc R\le y)$.  
Define the marginal density
\[
p_R(y)
:=
\int_{\mathbf v^\perp}
p(y\mathbf v+\mathbf r)
\,\sigma_{\mathbf v^\perp}(\mathrm d\mathbf r),
\]
and the canonical slice barycenter
\[
\boldsymbol\beta(y)
:=
\frac{
\displaystyle
\int_{\mathbf v^\perp}
(y\mathbf v+\mathbf r)
p(y\mathbf v+\mathbf r)
\,\sigma_{\mathbf v^\perp}(\mathrm d\mathbf r)
}{
p_R(y)
}.
\]
Let $\overline{\mathbf g}
:=
\mathbb E\rvec g$.  
Then there exists a universal constant \(C>0\), independent of \(L_Q\),
such that
\[
\left\|
\boldsymbol\beta(y)
-
\overline{\mathbf g}
\right\|_2
\le
C
\left(
1+
\left|
\Phi^{-1}(F_R(y))
\right|
\right)
\]
for every \(y\in\mathbb R\).
\end{theorem}

\begin{theorem}[Two-dimensional slice estimate, smooth full-dimensional case]
\label{thm:2d}
Let $(\rsc{Y},\rsc{W})$ 
have a positive \(C^2\) density
\[
p_{Y,W}(y,w)
=
\mathcal Z_Q^{-1}
\exp\left(
-\frac12(y^2+w^2)
-
Q(y,w)
\right),
\qquad
(y,w)\in\mathbb R^2,
\]
where $Q\in C^2(\mathbb R^2)$ 
is convex and, for some finite \(L_Q\),
\[
\mathbf O_2
\preceq
\nabla^2Q(y,w)
\preceq
L_Q\mathbf I_2
\qquad
\text{for every }(y,w)\in\mathbb R^2.
\]

Define $p_Y(y)
:=
\int_{\mathbb R}
p_{Y,W}(y,w)\,\mathrm dw$,  $F_Y(y)
:=
\mathbb P(\rsc{Y}\le y)$ 
and define the canonical conditional mean
\[
m(y)
:=
\frac{
\displaystyle
\int_{\mathbb R}
w
\exp\left(
-\frac12w^2-Q(y,w)
\right)
\,\mathrm dw
}{
\displaystyle
\int_{\mathbb R}
\exp\left(
-\frac12w^2-Q(y,w)
\right)
\,\mathrm dw
}.
\]
Then $m\in C^1(\mathbb R)$,  
and there exists a universal constant \(C>0\), independent of \(L_Q\),
such that
\begin{equation}
\bigl(
y-\mathbb E\rsc{Y}
\bigr)^2
+
\bigl(
m(y)-\mathbb E\rsc{W}
\bigr)^2
\le
C\left(
1+
\log
\frac1{
\min\{
F_Y(y),
1-F_Y(y)
\}
}
\right)
\label{eq:two-dimensional-slice-bound}
\end{equation}
for every \(y\in\mathbb R\).
\end{theorem}
\begin{proof}[Proof of Theorem~\ref{thm:GSR}]
Set $\boldsymbol\Delta(y)
:=
\boldsymbol\beta(y)
-
\overline{\mathbf g}$.  
We first control the component parallel to \(\mathbf v\). 
Choose an isometric linear map
\[
\mathcal J_{\mathbf v}:
\mathbb R^{d-1}
\to
\mathbf v^\perp,
\]
with the usual zero-dimensional interpretation when \(d=1\), and define
\[
\mathcal E_1:
\mathbb R
\to
\mathbb R^d,
\qquad
\mathcal E_1r
=
r\mathbf v.
\]
Then the density of \(\rsc R\) can be written as
\[
p_R(r)
\propto
e^{-r^2/2}
A_1(r),
\]
where
\[
A_1(r)
:=
\int_{\mathbb R^{d-1}}
\exp\left(
-\frac12\|\mathbf z\|_2^2
-
Q\left(
\mathcal E_1r
+
\mathcal J_{\mathbf v}\mathbf z
\right)
\right)
\,\mathrm d\mathbf z.
\]
Under the preceding zero dimensional convention, when \(d=1\),
\[
A_1(r)
=
\exp\left(
-
Q(\mathcal E_1r)
\right).
\]
Thus the invocation of
Lemma~\ref{lem:gaussian-fiber-domination}
also covers \(d=1\). For $d>1$, 
by Lemma~\ref{lem:gaussian-fiber-domination},
\[
\widetilde Q_1(r)
:=
-\log A_1(r)
\]
belongs to \(C^2(\mathbb R)\), is convex, and satisfies
\[
0
\le
\widetilde Q_1''(r)
\le
L_Q.
\]
Consequently,
\[
p_R(r)
\propto
\exp\left(
-\frac12r^2
-
\widetilde Q_1(r)
\right).
\]
Let $\rsc{\zeta}_0
\sim
N(0,1)$ 
be independent of \(\rsc R\).  The pair $(\rsc R,\rsc{\zeta}_0)$ 
has density proportional to
\[
\exp\left(
-\frac12(r^2+z^2)
-
\widetilde Q_1(r)
\right).
\]
Its residual potential $(r,z)
\longmapsto
\widetilde Q_1(r)$ 
is convex, \(C^2\), and satisfies
\[
\mathbf O_2
\preceq
\nabla^2
\bigl[
\widetilde Q_1(r)
\bigr]
\preceq
L_Q\mathbf I_2.
\]
Theorem~\ref{thm:2d} therefore gives
\[
\left|
y-\mathbb E\rsc R
\right|^2
\le
C
\left(
1+
\log
\frac1{
\min\{
F_R(y),
1-F_R(y)
\}
}
\right).
\]

Since $\left\langle
\boldsymbol\beta(y),
\mathbf v
\right\rangle
=
y$ 
and $\left\langle
\overline{\mathbf g},
\mathbf v
\right\rangle
=
\mathbb E\rsc R$,  
we obtain
\[
\left|
\left\langle
\boldsymbol\Delta(y),
\mathbf v
\right\rangle
\right|^2
\le
C
\left(
1+
\log
\frac1{
\min\{
F_R(y),
1-F_R(y)
\}
}
\right).
\]
By Lemma~\ref{lem:quantile-log},
\begin{equation}
\left|
\left\langle
\boldsymbol\Delta(y),
\mathbf v
\right\rangle
\right|
\le
C
\left(
1+
\left|
\Phi^{-1}(F_R(y))
\right|
\right).
\label{eq:GSR-parallel-component}
\end{equation}

 If \(d=1\), this proves the theorem.  Assume henceforth that \(d\ge2\). 
Fix a unit vector $\mathbf h
\in
\mathbf v^\perp$,  
and define $\rsc Y_{\mathbf h}
:=
\langle\rvec g,\mathbf v\rangle$ and $\rsc W_{\mathbf h}
:=
\langle\rvec g,\mathbf h\rangle$. 
Choose an isometric map $\mathcal J_{\mathbf v,\mathbf h}:
\mathbb R^{d-2}
\to
\{\mathbf v,\mathbf h\}^\perp$, 
and define
\[
\mathcal E_{\mathbf h}:
\mathbb R^2
\to
\mathbb R^d,
\qquad
\mathcal E_{\mathbf h}(y,w)
=
y\mathbf v+w\mathbf h.
\]
Since \(\mathbf v\) and \(\mathbf h\) are orthonormal, $\mathcal E_{\mathbf h}^\ast
\mathcal E_{\mathbf h}
=
\mathbf I_2$. 
The density of
\((\rsc Y_{\mathbf h},\rsc W_{\mathbf h})\)
has the form
\[
p_{\mathbf h}(y,w)
\propto
e^{-(y^2+w^2)/2}
A_{\mathbf h}(y,w),
\]
where
\[
A_{\mathbf h}(y,w)
:=
\int_{\mathbb R^{d-2}}
\exp\left(
-\frac12\|\mathbf z\|_2^2
-
Q\left(
\mathcal E_{\mathbf h}(y,w)
+
\mathcal J_{\mathbf v,\mathbf h}\mathbf z
\right)
\right)
\,\mathrm d\mathbf z.
\]
Let $\widetilde Q_{\mathbf h}(y,w)
:=
-\log A_{\mathbf h}(y,w)$.  
When \(d=2\), the integral is over
\(\mathbb R^0\), and hence
\[
A_{\mathbf h}(y,w)
=
\exp\left(
-
Q\left(
\mathcal E_{\mathbf h}(y,w)
\right)
\right).
\]
Therefore $\widetilde Q_{\mathbf h}(y,w)
=
Q\left(
\mathcal E_{\mathbf h}(y,w)
\right)$,  
and the stated convexity and Hessian bounds follow directly. For $d>2$, 
by Lemma~\ref{lem:gaussian-fiber-domination}, $\widetilde Q_{\mathbf h}(y,w)$ 
belongs to \(C^2(\mathbb R^2)\), is convex, and satisfies
\[
\begin{aligned}
\mathbf O_2
\preceq
\nabla^2
\widetilde Q_{\mathbf h}(y,w)
\preceq
L_Q
\mathcal E_{\mathbf h}^\ast
\mathcal E_{\mathbf h}
=
L_Q\mathbf I_2.
\end{aligned}
\]
Thus
\[
p_{\mathbf h}(y,w)
=
\mathcal Z_{\mathbf h}^{-1}
\exp\left(
-\frac12(y^2+w^2)
-
\widetilde Q_{\mathbf h}(y,w)
\right)
\]
satisfies the assumptions of Theorem~\ref{thm:2d}.

The canonical conditional mean of
\(\rsc W_{\mathbf h}\) given
\(\rsc Y_{\mathbf h}=y\) is $\left\langle
\boldsymbol\beta(y),
\mathbf h
\right\rangle$,  
while $\mathbb E\rsc W_{\mathbf h}
=
\left\langle
\overline{\mathbf g},
\mathbf h
\right\rangle$. 
Applying Theorem~\ref{thm:2d} gives
\[
\left|
\left\langle
\boldsymbol\Delta(y),
\mathbf h
\right\rangle
\right|^2
\le
C
\left(
1+
\log
\frac1{
\min\{
F_R(y),
1-F_R(y)
\}
}
\right).
\]
By Lemma~\ref{lem:quantile-log},
\[
\left|
\left\langle
\boldsymbol\Delta(y),
\mathbf h
\right\rangle
\right|
\le
C
\left(
1+
\left|
\Phi^{-1}(F_R(y))
\right|
\right).
\]
Taking the supremum over all unit
\(\mathbf h\in\mathbf v^\perp\), we obtain
\begin{equation}
\left\|
P_{\mathbf v^\perp}
\boldsymbol\Delta(y)
\right\|_2
\le
C
\left(
1+
\left|
\Phi^{-1}(F_R(y))
\right|
\right).
\label{eq:GSR-orthogonal-component}
\end{equation}

Finally,
\[
\left\|
\boldsymbol\Delta(y)
\right\|_2^2
=
\left|
\left\langle
\boldsymbol\Delta(y),
\mathbf v
\right\rangle
\right|^2
+
\left\|
P_{\mathbf v^\perp}
\boldsymbol\Delta(y)
\right\|_2^2.
\]
Combining
\eqref{eq:GSR-parallel-component}
and
\eqref{eq:GSR-orthogonal-component}
proves the theorem.
\end{proof}

\subsubsection{Loser-Tail Regression}
\label{sec:loser-tail}

The winner-slice estimate controls the conditional barycenter of the
block that attains the maximum.  For each losing block, the relevant
conditioning event is instead a lower-tail event of the form
\[
\{\rsc R\le y\}.
\]
Unlike the equality event
\(\{\rsc R=y\}\),
this event has positive probability and is handled by ordinary event
conditioning.  The following estimate is stated in the smooth
full-dimensional setting used in
Proposition~\ref{prop:one-step-smooth}.  Its proof uses only the
quadratic transport inequality implied by \(1\)-strong log-concavity and
does not use an upper Hessian bound.

\begin{lemma}[Lower-tail barycenter bound]
\label{lem:loser-tail}
Let \(\rvec g\) have a positive \(C^2\) density
\[
p(\mathbf g)
=
\mathcal Z^{-1}
\exp\left(
-\frac12\|\mathbf g\|_2^2
-
Q(\mathbf g)
\right),
\qquad
\mathbf g\in\mathbb R^d,
\]
where $Q\in C^2(\mathbb R^d)$ 
is convex.  Fix $\mathbf v\in\mathbb S^{d-1}$, $\rsc R
:=
\langle\rvec g,\mathbf v\rangle$,  
and define $F_R(y)
:=
\mathbb P(\rsc R\le y)$.  
Then $0<F_R(y)<1$ for every   $y\in\mathbb R$. 
Set $z(y)
:=
\Phi^{-1}(F_R(y))$.  
There exists a universal constant \(C>0\) such that
\begin{equation}
\left\|
\mathbb E
\left[
\rvec g-\mathbb E\rvec g
\mid
\rsc R\le y
\right]
\right\|_2
\le
C
\frac{
\phi(z(y))
}{
\Phi(z(y))
},
\qquad
y\in\mathbb R.
\label{eq:loser-tail-barycenter}
\end{equation}
\end{lemma}

\begin{proof}[Proof of Lemma~\ref{lem:loser-tail}]
The total potential
\[
V(\mathbf g)
:=
\frac12\|\mathbf g\|_2^2
+
Q(\mathbf g)
\]
satisfies
\[
\nabla^2V(\mathbf g)
=
\mathbf I_d
+
\nabla^2Q(\mathbf g)
\succeq
\mathbf I_d.
\]
Thus the law $\mu
:=
\Law(\rvec g)$ 
is \(1\)-strongly log-concave.

Since \(p\) is strictly positive on all of \(\mathbb R^d\), both open
halfspaces
\[
\left\{
\mathbf g:
\langle\mathbf g,\mathbf v\rangle<y
\right\}
\quad\text{and}\quad
\left\{
\mathbf g:
\langle\mathbf g,\mathbf v\rangle>y
\right\}
\]
have positive \(\mu\)-measure for every finite \(y\).  Consequently, $0<F_R(y)<1$. 

Fix $y\in\mathbb R$,  
and define
\[
A_y
:=
\{\rsc R\le y\},
\qquad
p_y
:=
\mathbb P(A_y)
=
F_R(y),
\qquad
q_y
:=
1-p_y.
\]
Then $0<p_y,q_y<1$.  
Let $\nu_A
:=
\Law(\rvec g\mid A_y).$ 
For every Borel set
\(E\subseteq\mathbb R^d\),
\[
\nu_A(E)
=
\frac{
\mu(E\cap A_y)
}{
p_y
}.
\]
Hence
\[
\frac{
\mathrm d\nu_A
}{
\mathrm d\mu
}(\mathbf g)
=
\frac{
\mathbf 1_{A_y}(\mathbf g)
}{
p_y
}
\qquad
\mu\text{-almost everywhere}.
\]
Since \(\nu_A\) is supported on \(A_y\),
\begin{equation}
\begin{aligned}
\KL(
\nu_A\|\mu)
=
\int_{\mathbb R^d}
\log
\left(
\frac{
\mathrm d\nu_A
}{
\mathrm d\mu
}
\right)
\nu_A(\mathrm d\mathbf g)
=
\log\frac1{p_y}.
\end{aligned}
\label{eq:loser-tail-KL}
\end{equation}
The quadratic transport inequality for \(1\)-strongly log-concave
measures gives
\begin{equation}
W_2^2(
\nu_A,\mu)
\le
2
\KL(
\nu_A\|\mu).
\label{eq:loser-tail-T2}
\end{equation}
Let
\(\Gamma\)
be any coupling of \(\nu_A\) and \(\mu\).  Then
\[
\begin{aligned}
\left\|
\int_{\mathbb R^d}
\mathbf x\,
\nu_A(\mathrm d\mathbf x)
-
\int_{\mathbb R^d}
\mathbf y\,
\mu(\mathrm d\mathbf y)
\right\|_2
&=
\left\|
\int_{\mathbb R^d\times\mathbb R^d}
(\mathbf x-\mathbf y)
\,
\Gamma(\mathrm d\mathbf x,\mathrm d\mathbf y)
\right\|_2
\\
&\le
\int_{\mathbb R^d\times\mathbb R^d}
\|\mathbf x-\mathbf y\|_2
\,
\Gamma(\mathrm d\mathbf x,\mathrm d\mathbf y)
\\
&\le
\left[
\int_{\mathbb R^d\times\mathbb R^d}
\|\mathbf x-\mathbf y\|_2^2
\,
\Gamma(\mathrm d\mathbf x,\mathrm d\mathbf y)
\right]^{1/2}.
\end{aligned}
\]
Taking the infimum over all couplings gives
\[
\left\|
\mathbb E[
\rvec g\mid A_y]
-
\mathbb E\rvec g
\right\|_2
\le
W_2(
\nu_A,\mu).
\]
Combining this with
\eqref{eq:loser-tail-KL}
and
\eqref{eq:loser-tail-T2},
we obtain
\begin{equation}
\left\|
\mathbb E[
\rvec g\mid A_y]
-
\mathbb E\rvec g
\right\|_2
\le
\sqrt{
2\log\frac1{p_y}
}.
\label{eq:loser-tail-direct}
\end{equation}

This estimate is useful when \(p_y\) is not close to one.  To obtain the
complementary estimate, define $\nu_B
:=
\Law(
\rvec g\mid A_y^c)$ 
and set
\[
\overline{\mathbf g}_A
:=
\mathbb E[
\rvec g\mid A_y],
\qquad
\overline{\mathbf g}_B
:=
\mathbb E[
\rvec g\mid A_y^c],
\qquad
\overline{\mathbf g}
:=
\mathbb E\rvec g.
\]
Since
\[
\overline{\mathbf g}
=
p_y\overline{\mathbf g}_A
+
q_y\overline{\mathbf g}_B,
\]
we have
\[
\overline{\mathbf g}_A
-
\overline{\mathbf g}
=
-\frac{q_y}{p_y}
\left(
\overline{\mathbf g}_B
-
\overline{\mathbf g}
\right).
\]
Applying the preceding entropy--transport argument to the event
\(A_y^c\), whose probability is \(q_y\), gives
\[
\left\|
\overline{\mathbf g}_B
-
\overline{\mathbf g}
\right\|_2
\le
\sqrt{
2\log\frac1{q_y}
}.
\]
Therefore
\begin{equation}
\left\|
\overline{\mathbf g}_A
-
\overline{\mathbf g}
\right\|_2
\le
\frac{q_y}{p_y}
\sqrt{
2\log\frac1{q_y}
}.
\label{eq:loser-tail-complement}
\end{equation}

Let $z
:=
z(y)
=
\Phi^{-1}(p_y)$.  
Then $p_y
=
\Phi(z)$ and  $q_y
=
\Phi(-z)$. 
Combining
\eqref{eq:loser-tail-direct}
and
\eqref{eq:loser-tail-complement},
we obtain
\[
\begin{aligned}
\left\|
\mathbb E
\left[
\rvec g-\mathbb E\rvec g
\mid
\rsc R\le y
\right]
\right\|_2
\quad\le
\min\left\{
\sqrt{
2\log
\frac1{\Phi(z)}
},
\,
\frac{
\Phi(-z)
}{
\Phi(z)
}
\sqrt{
2\log
\frac1{\Phi(-z)}
}
\right\}.
\end{aligned}
\]
Lemma~\ref{lem:tail-to-hazard}
therefore yields \eqref{eq:loser-tail-barycenter}:
\[
\left\|
\mathbb E
\left[
\rvec g-\mathbb E\rvec g
\mid
\rsc R\le y
\right]
\right\|_2
\le
C
\frac{
\phi(z)
}{
\Phi(z)
}.
\]
\end{proof}

\subsubsection{Proof of the Two-Dimensional Slice Estimate}
\label{sec:2dproof}
The proof of Theorem~\ref{thm:2d} reduces to controlling two
one-dimensional displacement functions.  The following no-spike lemma
converts a subgaussian tail bound for \(h(\rsc{Y})\), together with the
pointwise derivative control
\[
h'(y)^2\le U''(y),
\]
into a pointwise bound expressed in terms of the lower and upper quantiles
of \(\rsc{Y}\).  We will apply the lemma first to $h_0(y)=y-\mathbb E\rsc{Y}$ 
and then to $h(y)=m(y)-\mathbb E\rsc{W}$.

\begin{lemma}[No-spike lemma, smooth form]
\label{lem:nospike}
Let \(\rsc{Y}\) have a positive density
\[
p_Y(y)
=
\mathcal Z_U^{-1}
e^{-U(y)},
\qquad
\mathcal Z_U
:=
\int_{\mathbb R}
e^{-U(t)}
\,\mathrm dt
<
\infty,
\]
where $U\in C^2(\mathbb R)$ 
is convex.  Define $F_Y(y)
:=
\mathbb P(\rsc{Y}\le y)$,  
and let $h:\mathbb R\to\mathbb R$ 
be continuously differentiable. 
Assume that, for some constants
\(c_0>0\) and \(C_0<\infty\),
\begin{equation}
\mathbb P
\left(
|h(\rsc{Y})|
\ge
r
\right)
\le
C_0e^{-c_0r^2},
\qquad
r\ge0,
\label{eq:no-spike-tail-assumption}
\end{equation}
and
\begin{equation}
h'(y)^2
\le
U''(y),
\qquad
y\in\mathbb R.
\label{eq:no-spike-derivative-assumption}
\end{equation}
Then there exists
\(C=C(c_0,C_0)\)
such that
\begin{equation}
|h(y)|^2
\le
C\left(
1+
\log
\frac1{
\min\{
F_Y(y),
1-F_Y(y)
\}
}
\right)
\label{eq:no-spike-conclusion}
\end{equation}
for every \(y\in\mathbb R\).
\end{lemma}

\begin{proof}[Proof of Lemma~\ref{lem:nospike}]
Let $\mu_Y
:=
\Law(\rsc{Y})$ 
denote the law of \(\rsc{Y}\).  Fix an arbitrary point $y_0\in\mathbb R$, 
and set $R
:=
|h(y_0)|$.  
If \(R=0\), there is nothing to prove.

Define
\[
\widetilde h
:=
\begin{cases}
h,
&
h(y_0)>0,
\\
-h,
&
h(y_0)<0.
\end{cases}
\]
Then $\widetilde h(y_0)=R$,  $\widetilde h'(y)^2=h'(y)^2$,  
and $|\widetilde h(\rsc{Y})|
=
|h(\rsc{Y})|$.  
Thus
\(\widetilde h\)
satisfies the same tail and derivative assumptions as \(h\).  It is
therefore enough to prove the claim under the normalization  $h(y_0)=R>0$.  
Choose a constant $R_0\ge2$ 
sufficiently large, depending only on \(c_0\) and \(C_0\).  If $0<R\le R_0$,  
then
\[
1+
\log
\frac1{
\min\{
F_Y(y_0),
1-F_Y(y_0)
\}
}
\ge1,
\]
and hence
\[
R^2
\le
R_0^2
\left(
1+
\log
\frac1{
\min\{
F_Y(y_0),
1-F_Y(y_0)
\}
}
\right).
\]
Thus the conclusion follows after enlarging the constant.  We may
therefore assume throughout the remainder of the proof that $R>R_0$.  

Define the open superlevel set
\[
A_R
:=
\left\{
y\in\mathbb R:
h(y)>\frac R2
\right\}.
\]
Since \(h\) is continuous,
\(A_R\) is open and contains \(y_0\).  By
\eqref{eq:no-spike-tail-assumption},
\[
\begin{aligned}
\mu_Y(A_R)
\le
\mathbb P
\left(
|h(\rsc{Y})|
\ge
\frac R2
\right)
\le
C_0
\exp\left(
-\frac{c_0R^2}{4}
\right).
\end{aligned}
\]
Consequently, there exist constants
\(C_1,c_1>0\),
depending only on \(C_0,c_0\), such that
\begin{equation}
\mu_Y(A_R)
\le
C_1e^{-c_1R^2}.
\label{eq:no-spike-superlevel-mass}
\end{equation}

Let $J=(a,b)$ 
be the connected component of \(A_R\) containing \(y_0\), where either
endpoint is allowed to be infinite. 
Suppose first that $b=+\infty.$ 
Then $[y_0,\infty)
\subseteq
J,$
and hence
\[
1-F_Y(y_0)
\le
\mu_Y(J)
\le
C_1e^{-c_1R^2}.
\]
Similarly, if $a=-\infty,$ 
then $(-\infty,y_0]
\subseteq
J$ 
and therefore
\[
F_Y(y_0)
\le
\mu_Y(J)
\le
C_1e^{-c_1R^2}.
\]

It remains to consider the case
\[
-\infty<a<y_0<b<\infty.
\]
By continuity of \(h\) and maximality of the connected component \(J\),
\[
h(a)
=
h(b)
=
\frac R2.
\]
Set $\alpha
:=
y_0-a$ and $\beta
:=
b-y_0.$ 
Then $\alpha,\beta>0$. 
Since $h(y_0)-h(a)
=
\frac R2$, 
 Cauchy--Schwarz gives
\[
\begin{aligned}
\frac R2
=
\left|
\int_a^{y_0}
h'(t)\,\mathrm dt
\right|
\le
\sqrt{\alpha}
\left(
\int_a^{y_0}
h'(t)^2\,\mathrm dt
\right)^{1/2}.
\end{aligned}
\]
Hence
\begin{equation}
\int_a^{y_0}
h'(t)^2\,\mathrm dt
\ge
\frac{R^2}{4\alpha}.
\label{eq:no-spike-left-energy}
\end{equation}
Similarly,
\begin{equation}
\int_{y_0}^b
h'(t)^2\,\mathrm dt
\ge
\frac{R^2}{4\beta}.
\label{eq:no-spike-right-energy}
\end{equation}

Using
\eqref{eq:no-spike-derivative-assumption},
we obtain
\begin{equation}
\begin{aligned}
U'(y_0)-U'(a)
=
\int_a^{y_0}
U''(t)\,\mathrm dt
\ge
\int_a^{y_0}
h'(t)^2\,\mathrm dt
\ge
\frac{R^2}{4\alpha},
\end{aligned}
\label{eq:no-spike-left-curvature}
\end{equation}
and
\begin{equation}
\begin{aligned}
U'(b)-U'(y_0)
=
\int_{y_0}^b
U''(t)\,\mathrm dt
\ge
\int_{y_0}^b
h'(t)^2\,\mathrm dt
\ge
\frac{R^2}{4\beta}.
\end{aligned}
\label{eq:no-spike-right-curvature}
\end{equation}

We now distinguish two cases according to the sign of \(U'(y_0)\).

\paragraph{Case 1: \(U'(y_0)\ge0\).}

By
\eqref{eq:no-spike-right-curvature},
\[
U'(b)
\ge
\frac{R^2}{4\beta}
>
0.
\]
Since \(U'\) is nondecreasing and \(U'(y_0)\ge0\), the function \(U\)
is nondecreasing on \([y_0,b]\).  Therefore
\[
p_Y(t)
\ge
p_Y(b),
\qquad
t\in[y_0,b],
\]
and hence
\begin{equation}
\mu_Y([y_0,b])
\ge
\beta p_Y(b).
\label{eq:no-spike-right-interval-mass}
\end{equation}
For \(t\ge b\), convexity gives
\[
U(t)
\ge
U(b)+U'(b)(t-b).
\]
Consequently,
\[
\begin{aligned}
1-F_Y(b)
&=
\int_b^\infty
p_Y(t)\,\mathrm dt
\\
&\le
p_Y(b)
\int_0^\infty
e^{-U'(b)s}\,\mathrm ds
\\
&=
\frac{p_Y(b)}{U'(b)}
\\
&\le
\frac{4\beta p_Y(b)}{R^2}
\\
&\le
\frac4{R^2}
\mu_Y([y_0,b]),
\end{aligned}
\]
where the last inequality uses
\eqref{eq:no-spike-right-interval-mass}.
Therefore
\[
\begin{aligned}
1-F_Y(y_0)
&=
\mu_Y([y_0,b])
+
1-F_Y(b)
\\
&\le
\left(
1+\frac4{R^2}
\right)
\mu_Y([y_0,b])
\\
&\le
2\mu_Y(J),
\end{aligned}
\]
because \(R\ge R_0\ge2\).  By
\eqref{eq:no-spike-superlevel-mass},
\begin{equation}
1-F_Y(y_0)
\le
2C_1e^{-c_1R^2}.
\label{eq:no-spike-right-quantile}
\end{equation}

\paragraph{Case 2: \(U'(y_0)<0\).}

By
\eqref{eq:no-spike-left-curvature},
\[
\begin{aligned}
-U'(a)
=
-U'(y_0)
+
U'(y_0)-U'(a)
\ge
\frac{R^2}{4\alpha}.
\end{aligned}
\]
Since \(U'\) is nondecreasing and \(U'(y_0)<0\), the function \(U\) is
nonincreasing on \([a,y_0]\).  Therefore
\[
p_Y(t)
\ge
p_Y(a),
\qquad
t\in[a,y_0],
\]
and hence
\begin{equation}
\mu_Y([a,y_0])
\ge
\alpha p_Y(a).
\label{eq:no-spike-left-interval-mass}
\end{equation}
For \(t\le a\), convexity gives
\[
\begin{aligned}
U(t)
\ge
U(a)+U'(a)(t-a)
=
U(a)+(-U'(a))(a-t).
\end{aligned}
\]
Consequently,
\[
\begin{aligned}
F_Y(a)
&=
\int_{-\infty}^a
p_Y(t)\,\mathrm dt
\\
&\le
p_Y(a)
\int_0^\infty
e^{-(-U'(a))s}\,\mathrm ds
\\
&=
\frac{p_Y(a)}{-U'(a)}
\\
&\le
\frac{4\alpha p_Y(a)}{R^2}
\\
&\le
\frac4{R^2}
\mu_Y([a,y_0]),
\end{aligned}
\]
where the last inequality uses
\eqref{eq:no-spike-left-interval-mass}.
Therefore
\[
\begin{aligned}
F_Y(y_0)
&=
F_Y(a)
+
\mu_Y([a,y_0])
\\
&\le
\left(
1+\frac4{R^2}
\right)
\mu_Y([a,y_0])
\\
&\le
2\mu_Y(J).
\end{aligned}
\]
Using
\eqref{eq:no-spike-superlevel-mass},
\begin{equation}
F_Y(y_0)
\le
2C_1e^{-c_1R^2}.
\label{eq:no-spike-left-quantile}
\end{equation}

The unbounded cases and
\eqref{eq:no-spike-right-quantile}--\eqref{eq:no-spike-left-quantile}
show that, in every case,
\[
\min\{
F_Y(y_0),
1-F_Y(y_0)
\}
\le
C_2e^{-c_2R^2}
\]
for constants \(C_2,c_2>0\) depending only on \(C_0,c_0\).  Taking
logarithms gives
\[
c_2R^2
\le
\log C_2
+
\log
\frac1{
\min\{
F_Y(y_0),
1-F_Y(y_0)
\}
}.
\]
After enlarging the constant,
\[
R^2
\le
C
\left(
1+
\log
\frac1{
\min\{
F_Y(y_0),
1-F_Y(y_0)
\}
}
\right).
\]
Since $R
=
|h(y_0)|$ 
and \(y_0\in\mathbb R\) was arbitrary, this proves
\eqref{eq:no-spike-conclusion}.
\end{proof}

\begin{proof}[Proof of Theorem~\ref{thm:2d}]
By assumption,
\[
Q\in C^2(\mathbb R^2),
\qquad
\mathbf O_2
\preceq
\nabla^2Q(y,w)
\preceq
L_Q\mathbf I_2
\]
for every
\((y,w)\in\mathbb R^2\).

For $(s,y)\in\mathbb R^2,$ 
define
\[
A(s,y)
:=
\int_{\mathbb R}
\exp\left(
-\frac12(w-s)^2
-
Q(y,w)
\right)
\,\mathrm dw,
\]
and
\[
\Psi(s,y)
:=
\log A(s,y).
\]
Also set $A_0(y)
:=
A(0,y)$. 

The function
\[
G_2((s,y),w)
:=
\frac12(w-s)^2
+
Q(y,w)
\]
satisfies the domination condition of
Lemma~\ref{lem:smooth-marginal-regularity}
by
Lemma~\ref{lem:gaussian-fiber-domination}.
Consequently,
\[
A\in C^2(\mathbb R^2),
\qquad
\Psi\in C^2(\mathbb R^2).
\]
In particular, all differentiations of \(A\) and \(\Psi\) below are
justified.

Moreover, the function
\[
(s,y,w)
\longmapsto
-\frac12(w-s)^2
-
Q(y,w)
\]
is jointly concave.  By Prékopa's theorem, $(s,y)
\longmapsto
A(s,y)$ 
is log-concave.  Hence \(\Psi\) is concave:
\begin{equation}
\nabla^2\Psi(s,y)
\preceq
\mathbf O_2.
\label{eq:two-d-Psi-concavity}
\end{equation}

\paragraph{Marginal potential of \(\rsc{Y}\).}

The marginal density of \(\rsc{Y}\) satisfies
\[
\begin{aligned}
p_Y(y)
=
\int_{\mathbb R}
p_{Y,W}(y,w)
\,\mathrm dw
\propto
e^{-y^2/2}
A_0(y)
=
\exp\left(
-\frac12y^2
+
\Psi(0,y)
\right).
\end{aligned}
\]
Writing
\[
p_Y(y)
=
\mathcal Z_Y^{-1}
e^{-U(y)},
\]
we have
\begin{equation}
U(y)
=
\frac12y^2
-
\Psi(0,y)
+
\mathrm{const}.
\label{eq:two-d-marginal-potential}
\end{equation}
Therefore
\begin{equation}
U''(y)
=
1
-
\Psi_{yy}(0,y)
\ge
1,
\label{eq:two-d-marginal-curvature}
\end{equation}
because \(\Psi\) is concave.

Thus the marginal law of \(\rsc{Y}\) is
\(1\)-strongly log-concave.

\paragraph{Displacement in the first coordinate.}

Define
\[
h_0(y)
:=
y-\mathbb E\rsc{Y}.
\]
By
Lemma~\ref{lem:slc-subgaussian-affine},
applied to the one-dimensional law of \(\rsc{Y}\),
\[
\mathbb E
\exp\left(
\lambda h_0(\rsc{Y})
\right)
\le
e^{\lambda^2/2},
\qquad
\lambda\in\mathbb R.
\]
Consequently,
\[
\mathbb P
\left(
|h_0(\rsc{Y})|
\ge
r
\right)
\le
2e^{-r^2/2},
\qquad
r\ge0.
\]
Furthermore,
\[
h_0'(y)^2
=
1
\le
U''(y)
\]
by
\eqref{eq:two-d-marginal-curvature}.
Applying
Lemma~\ref{lem:nospike}
to \(h_0\) gives
\begin{equation}
\left|
y-\mathbb E\rsc{Y}
\right|^2
\le
C\left(
1+
\log
\frac1{
\min\{
F_Y(y),
1-F_Y(y)
\}
}
\right).
\label{eq:two-d-first-coordinate-bound}
\end{equation}

\paragraph{Subgaussianity of the conditional-mean displacement.}

Define $h(y)
:=
m(y)-\mathbb E\rsc{W}$.   
We first prove that $h(\rsc{Y})$ 
is subgaussian.

The marginal density of \(\rsc{W}\) can be written as
\[
p_W(w)
\propto
e^{-w^2/2}
B(w),
\]
where
\[
B(w)
:=
\int_{\mathbb R}
\exp\left(
-\frac12y^2
-
Q(y,w)
\right)
\,\mathrm dy.
\]
The function
\[
(y,w)
\longmapsto
-\frac12y^2
-
Q(y,w)
\]
is jointly concave.  Prékopa's theorem therefore implies that \(B\) is
log-concave.  Hence the marginal law of \(\rsc{W}\) is
\(1\)-strongly log-concave.

By
Lemma~\ref{lem:slc-subgaussian-affine},
\begin{equation}
\mathbb E
\exp\left(
\lambda
\left(
\rsc{W}
-
\mathbb E\rsc{W}
\right)
\right)
\le
e^{\lambda^2/2},
\qquad
\lambda\in\mathbb R.
\label{eq:two-d-W-marginal-mgf}
\end{equation}

The density-ratio definition of \(m\) gives a canonical version of the
conditional mean, and therefore
\[
m(\rsc{Y})
=
\mathbb E
\left[
\rsc{W}
\mid
\rsc{Y}
\right]
\qquad
\text{almost surely}.
\]
Consequently,
\[
h(\rsc{Y})
=
\mathbb E
\left[
\rsc{W}
-
\mathbb E\rsc{W}
\mid
\rsc{Y}
\right].
\]
Conditional Jensen's inequality and
\eqref{eq:two-d-W-marginal-mgf}
give
\[
\begin{aligned}
\mathbb E
e^{\lambda h(\rsc{Y})}
&=
\mathbb E
\exp\left(
\lambda
\mathbb E
\left[
\rsc{W}
-
\mathbb E\rsc{W}
\mid
\rsc{Y}
\right]
\right)
\\
&\le
\mathbb E
\mathbb E
\left[
\exp\left(
\lambda
\left(
\rsc{W}
-
\mathbb E\rsc{W}
\right)
\right)
\middle|
\rsc{Y}
\right]
\\
&=
\mathbb E
\exp\left(
\lambda
\left(
\rsc{W}
-
\mathbb E\rsc{W}
\right)
\right)
\\
&\le
e^{\lambda^2/2}.
\end{aligned}
\]
Hence
\begin{equation}
\mathbb P
\left(
|h(\rsc{Y})|
\ge
r
\right)
\le
2e^{-r^2/2},
\qquad
r\ge0.
\label{eq:two-d-conditional-mean-tail}
\end{equation}

\paragraph{Derivative control for the conditional mean.}

For every
\((s,y)\in\mathbb R^2\),
define the probability measure
\[
\nu_{s,y}(\mathrm dw)
:=
A(s,y)^{-1}
\exp\left(
-\frac12(w-s)^2
-
Q(y,w)
\right)
\,\mathrm dw.
\]
Define
\[
\overline w(s,y)
:=
\int_{\mathbb R}
w
\,\nu_{s,y}(\mathrm dw)
\]
and
\[
\operatorname{Var}_{s,y}(w)
:=
\int_{\mathbb R}
\left(
w-\overline w(s,y)
\right)^2
\nu_{s,y}(\mathrm dw).
\]

Differentiation with respect to \(s\) gives
\begin{equation}
\Psi_s(s,y)
=
\overline w(s,y)
-
s.
\label{eq:two-d-Psi-s}
\end{equation}
At \(s=0\),
\[
\nu_{0,y}(\mathrm dw)
=
A_0(y)^{-1}
\exp\left(
-\frac12w^2
-
Q(y,w)
\right)
\,\mathrm dw,
\]
which is the canonical conditional law of
\(\rsc{W}\) given
\(\rsc{Y}=y\).
Therefore
\begin{equation}
\Psi_s(0,y)
=
m(y).
\label{eq:two-d-m-as-Psi-s}
\end{equation}
Since
\(\Psi\in C^2(\mathbb R^2)\),
this proves that $m\in C^1(\mathbb R)$,  
with
\begin{equation}
m'(y)
=
\Psi_{sy}(0,y).
\label{eq:two-d-m-prime}
\end{equation}
A second differentiation with respect to \(s\) gives
\begin{equation}
\Psi_{ss}(s,y)
=
\operatorname{Var}_{s,y}(w)
-
1.
\label{eq:two-d-Psi-ss}
\end{equation}
By
\eqref{eq:two-d-Psi-concavity},
\[
-\nabla^2\Psi(0,y)
\succeq
\mathbf O_2.
\]
The determinant condition for this
\(2\times2\) positive semidefinite matrix gives
\begin{equation}
\Psi_{sy}(0,y)^2
\le
\left[
-\Psi_{ss}(0,y)
\right]
\left[
-\Psi_{yy}(0,y)
\right].
\label{eq:two-d-Hessian-determinant}
\end{equation}
By
\eqref{eq:two-d-Psi-ss},
\[
-\Psi_{ss}(0,y)
=
1
-
\operatorname{Var}_{0,y}(w).
\]
Concavity of \(\Psi\) implies
\[
-\Psi_{ss}(0,y)
\ge
0,
\]
while nonnegativity of variance gives
\[
-\Psi_{ss}(0,y)
\le
1.
\]
Thus
\begin{equation}
0
\le
-\Psi_{ss}(0,y)
\le
1.
\label{eq:two-d-Psi-ss-range}
\end{equation}
Combining
\eqref{eq:two-d-m-prime},
\eqref{eq:two-d-Hessian-determinant},
and
\eqref{eq:two-d-Psi-ss-range},
we obtain
\[
\begin{aligned}
m'(y)^2
=
\Psi_{sy}(0,y)^2
\le
\left[
-\Psi_{ss}(0,y)
\right]
\left[
-\Psi_{yy}(0,y)
\right]
\le
-\Psi_{yy}(0,y).
\end{aligned}
\]
On the other hand,
\eqref{eq:two-d-marginal-potential}
gives
\[
U''(y)
=
1
-
\Psi_{yy}(0,y)
\ge
-\Psi_{yy}(0,y).
\]
Therefore
\begin{equation}
h'(y)^2
=
m'(y)^2
\le
U''(y).
\label{eq:two-d-h-prime-bound}
\end{equation}

Applying
Lemma~\ref{lem:nospike}
to \(h\), using
\eqref{eq:two-d-conditional-mean-tail}
and
\eqref{eq:two-d-h-prime-bound},
gives
\begin{equation}
\left|
m(y)-\mathbb E\rsc{W}
\right|^2
\le
C\left(
1+
\log
\frac1{
\min\{
F_Y(y),
1-F_Y(y)
\}
}
\right).
\label{eq:two-d-conditional-mean-bound}
\end{equation}

Finally, adding
\eqref{eq:two-d-first-coordinate-bound}
and
\eqref{eq:two-d-conditional-mean-bound},
and enlarging the universal constant \(C\), gives
\[
\bigl(
y-\mathbb E\rsc{Y}
\bigr)^2
+
\bigl(
m(y)-\mathbb E\rsc{W}
\bigr)^2
\le
C\left(
1+
\log
\frac1{
\min\{
F_Y(y),
1-F_Y(y)
\}
}
\right).
\]
This is exactly
\eqref{eq:two-dimensional-slice-bound},
and the proof is complete.
\end{proof}

\begin{figure}[t]
\centering

\begin{minipage}[t]{0.48\linewidth}
\centering
\vspace{0pt}
\includegraphics[
    width=\linewidth,
    height=0.235\textheight,
    keepaspectratio
]{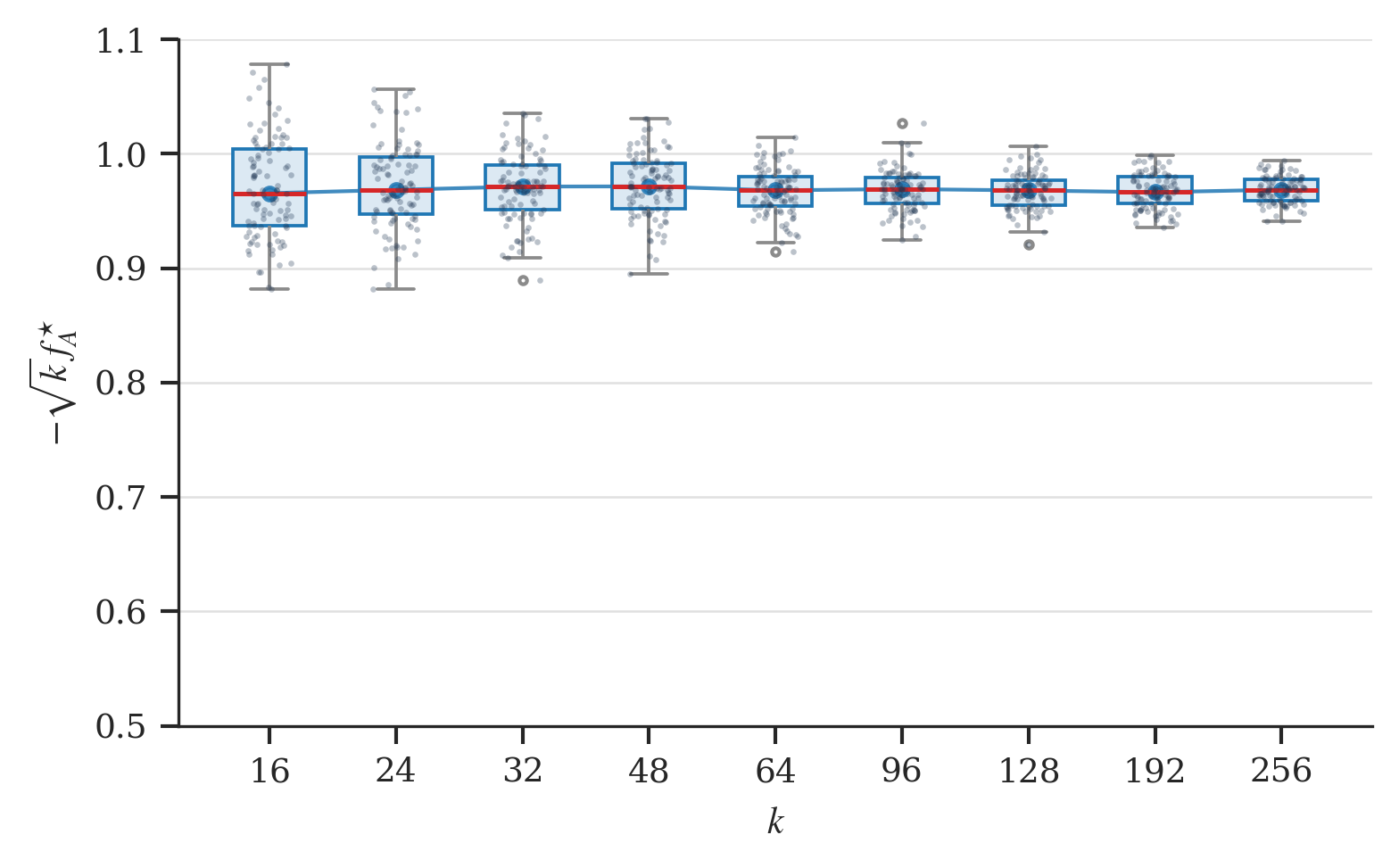}

\smallskip
\textbf{(a)} Optimum scale.
\end{minipage}
\hfill
\begin{minipage}[t]{0.48\linewidth}
\centering
\vspace{0pt}
\includegraphics[
    width=\linewidth,
    height=0.235\textheight,
    keepaspectratio
]{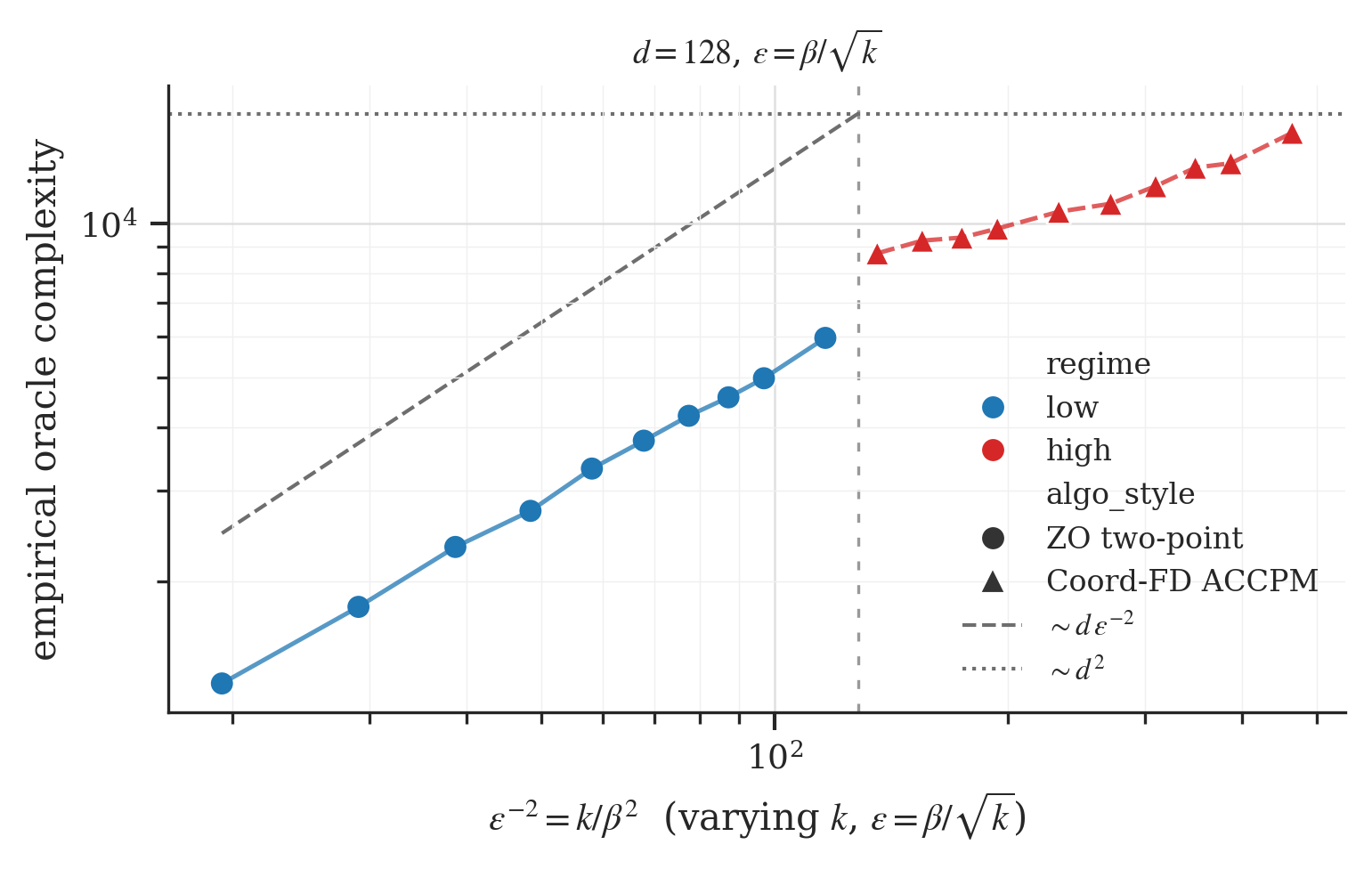}

\smallskip
\textbf{(b)} Accuracy transition.
\end{minipage}

\vspace{3mm}

\begin{minipage}[t]{0.48\linewidth}
\centering
\vspace{0pt}
\includegraphics[
    width=\linewidth,
    height=0.285\textheight,
    keepaspectratio
]{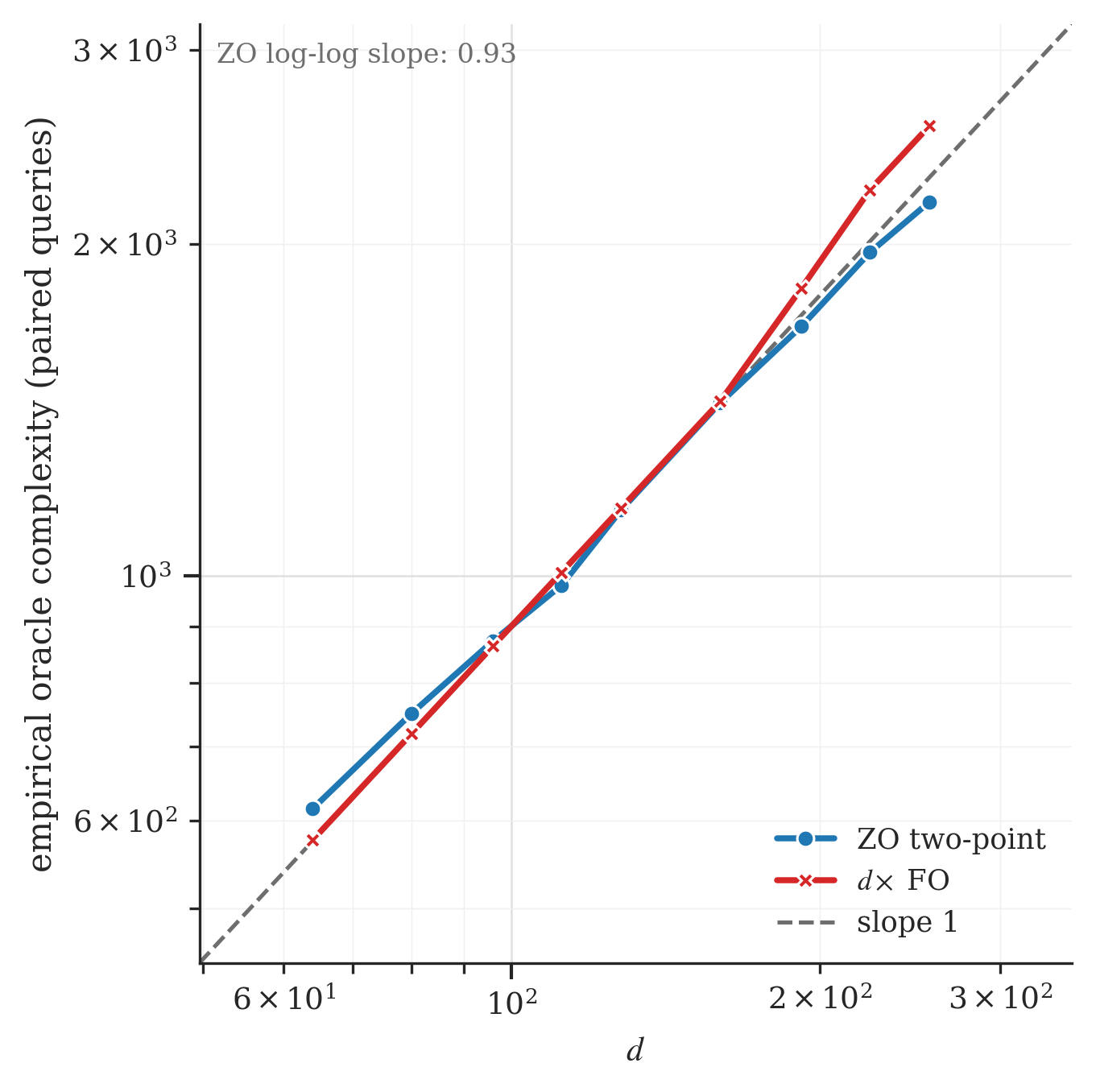}

\smallskip
\textbf{(c)} Dimension factor.
\end{minipage}
\hfill
\begin{minipage}[t]{0.48\linewidth}
\vspace{0pt}
\small
\refstepcounter{figure}
\textbf{Figure~\thefigure:}
Numerical illustrations on the random support function hard family.
\textbf{(a)} The normalized optimum scale
\(-\sqrt{k}f_{\rmat{A}}^\star\) remains of constant order over the tested values of
\(k\), supporting the scale \(f_{\rmat{A}}^\star\asymp-k^{-1/2}\).
\textbf{(b)} With \(d\) fixed and \(\epsilon^{-2}\) increasing, the
observed complexity initially follows the \(d\epsilon^{-2}\) scale. Around
the \(d^2\) scale, the coordinate finite-difference analytic-center
cutting-plane implementation, denoted Coord-FD ACCPM, shows a slower growing
nearly saturated behavior, consistent with the high-accuracy
\(\widetilde O(d^2)\) scale.
\textbf{(c)} In the regime \(\epsilon^{-2}\lesssim d\), the evaluation complexity of the two-point zeroth-order method is close
to \(d\) times the first-order subgradient baseline, illustrating the
dimension factor predicted by the theory.
\label{fig:numerical-illustrations}
\end{minipage}
\end{figure}

\section{Numerical Illustrations}
\label{sec:numerical}

The lower bounds proved in this paper are worst case oracle lower bounds
and are not established by experiments. The purpose of the following
numerical illustrations is instead to show that the random support function
family used in the proof exhibits the predicted geometry and that standard
exact value methods display the corresponding scaling behavior.

For each pair \((d,k)\), we generate independent truncated Gaussian blocks
\[
        \rvec{b}_i
        \sim
        N(\mathbf 0_d,\mathbf I_d)
        \mid
        \left\{
        \|\rvec{b}_i\|_2
        \le
        2\sqrt d
        \right\},
        \qquad
        \rvec{a}_i
        =
        \rvec{b}_i/\sqrt d,
\]
and define
\[
        f_{\rmat{A}}(\mathbf x)
        =
        \max_{1\le i\le k}
        \left\langle
        \rvec{a}_i,
        \mathbf x
        \right\rangle,
        \qquad
        \mathbf x\in B_2^d.
\]
The optimum value
\(f_{\rmat{A}}^\star\)
is computed from the dual quadratic
program
\[
        f_{\rmat{A}}^\star
        =
        -\sqrt{
        \min_{\mathbf p\in\Delta_k}
        \left\|
        \rmat{A}^{\top}\mathbf p
        \right\|_2^2
        }.
\]
All complexities are counted in scalar exact function value evaluations.

We report three numerical illustrations. First, we plot
\(-\sqrt{k}f_{\rmat{A}}^\star\) over random instances. This quantity remains of
constant order, numerically close to one, confirming the scale
\(f_{\rmat{A}}^\star\asymp-k^{-1/2}\)
used in the lower bound construction. 
Second, fixing \(d\), we vary \(\epsilon^{-2}\) across the transition
near \(d\). For \(\epsilon^{-2}\lesssim d\), the two-point zeroth-order
method follows the \(d\epsilon^{-2}\) scaling. For
\(\epsilon^{-2}\gtrsim d\), we also run a coordinate finite difference
analytic center cutting plane method, denoted Coord-FD ACCPM. This method
is used
to illustrate the high-accuracy behavior suggested by the
\(\widetilde O(d^2)\) evaluation oracle upper bound theory. The resulting
curve flattens after the \(d^2\) scale, consistent with the saturation
predicted by the full scale theory.
Third, in the regime
\(\epsilon^{-2}\lesssim d\), we compare the empirical oracle complexity of
a two-point zeroth-order method with \(d\) times the query complexity of a
projected first-order subgradient baseline. The two curves are close,
illustrating the expected linear dimension penalty of scalar value access.

\appendix

\section{Auxiliary Analytic and Probabilistic Estimates}

This appendix collects the analytic and probabilistic estimates used in the proof.

\subsection{Probability Integral Transform}

\begin{lemma}[Probability integral transform]
\label{lem:PIT}
Let \(\rsc Y\) have continuous CDF \(H\). Then $H(\rsc Y)\sim\Unif(0,1)$.  
Consequently, $-\log H(\rsc Y)\sim\Exp(1)$.
\end{lemma}

\begin{proof}[Proof of Lemma~\ref{lem:PIT}]
Fix \(u\in(0,1)\), and define the upper quantile
\[
q(u)
:=
\sup
\left\{
y\in\mathbb R:
H(y)\le u
\right\}.
\]
Because $\lim_{y\to-\infty}H(y)=0$ and $\lim_{y\to+\infty}H(y)=1$ 
the set in the preceding display is nonempty and bounded above, so
\(q(u)\in\mathbb R\).
Choose a sequence
\(y_n\uparrow q(u)\)
such that
\(H(y_n)\le u\).
By continuity,
\[
H(q(u))
=
\lim_{n\to\infty}H(y_n)
\le u.
\]
On the other hand, for every \(n\ge1\), $q(u)+\frac1n>q(u)$, 
and hence, by the definition of \(q(u)\), $H\left(q(u)+\frac1n\right)>u$. 
Continuity again gives
\[
H(q(u))
=
\lim_{n\to\infty}
H\left(q(u)+\frac1n\right)
\ge u.
\]
Therefore $H(q(u))=u$.  
Since \(H\) is nondecreasing,
\[
H(y)\le u
\quad\Longleftrightarrow\quad
y\le q(u).
\]
Consequently,
\[
\begin{aligned}
\mathbb P
\left(
H(\rsc Y)\le u
\right)
=
\mathbb P
\left(
\rsc Y\le q(u)
\right)
=
H(q(u))
=
u.
\end{aligned}
\]
Thus $H(\rsc Y)
\sim
\operatorname{Unif}(0,1)$. 

Now set $\rsc U:=H(\rsc Y).$ 
For every \(t\ge0\),
\[
\begin{aligned}
\mathbb P
\left(
-\log\rsc U\le t
\right)
=
\mathbb P
\left(
\rsc U\ge e^{-t}
\right)=
1-e^{-t}.
\end{aligned}
\]
Hence $-\log H(\rsc Y)
\sim
\operatorname{Exp}(1)$. 
\end{proof}
\subsection{Gaussian Tail and Inverse Mills Estimates}

\begin{lemma}[Two sided Mills ratio bound]
\label{lem:mills}
There exist absolute constants \(0<c<C<\infty\) such that, for all \(r\ge0\),
\begin{equation}
c\frac{\phi(r)}{1+r}
        \le
        \Phi(-r)
        \le
        C\frac{\phi(r)}{1+r}.
        \label{eq:A-1}
\end{equation}
\end{lemma}

\begin{proof}[Proof of Lemma~\ref{lem:mills}]
For the lower bound, set $h=\frac1{1+r}$. 
Then
\[
        \Phi(-r)
        =
        \int_r^\infty\phi(t)\,\mathrm dt
        \ge
        \int_r^{r+h}\phi(t)\,\mathrm dt.
\]
Since \(\phi\) is decreasing on \([0,\infty)\),
\[
        \int_r^{r+h}\phi(t)\,\mathrm dt
        \ge
        h\phi(r+h).
\]
Moreover,
\[
        \frac{\phi(r+h)}{\phi(r)}
        =
        \exp\left(-rh-\frac{h^2}{2}\right).
\]
Because $rh=\frac r{1+r}\le1$ and $h^2\le1$, 
we get $\phi(r+h)\ge e^{-3/2}\phi(r)$. 
Thus
\[
        \Phi(-r)\ge e^{-3/2}\frac{\phi(r)}{1+r}.
\]

For the upper bound, first suppose \(r\ge1\). Then
\[
        \Phi(-r)
        =
        \int_r^\infty \phi(t)\,\mathrm dt
        \le
        \frac1r\int_r^\infty t\phi(t)\,\mathrm dt.
\]
Since \(\phi'(t)=-t\phi(t)\), $\int_r^\infty t\phi(t)\,\mathrm dt=\phi(r)$. 
Thus
\[
        \Phi(-r)\le \frac{\phi(r)}r\le 2\frac{\phi(r)}{1+r}.
\]
If \(0\le r\le1\), then \(\Phi(-r)\le1/2\), while $\frac{\phi(r)}{1+r}\ge \frac{\phi(1)}2$. 
Hence
\[
        \Phi(-r)\le C\frac{\phi(r)}{1+r}
\]
also for \(0\le r\le1\). This proves the upper bound.
\end{proof}

\begin{lemma}[Gaussian Chernoff tail]
\label{lem:gaussian-chernoff}
For \(\breve{\zeta}\sim N(0,1)\) and \(r\ge0\),
\begin{equation}
\Phi(-r)=\Pp(\breve{\zeta} \le -r)\le e^{-r^2/2}.
        \label{eq:A-2}
\end{equation}
\end{lemma}

\begin{proof}[Proof of Lemma~\ref{lem:gaussian-chernoff}]
For any \(\lambda>0\),
\[
        \Pp(\breve{\zeta}\le -r)
        =
        \Pp(e^{-\lambda \breve{\zeta}}\ge e^{\lambda r})
        \le
        e^{-\lambda r}\E e^{-\lambda \breve{\zeta}}
        =
        e^{-\lambda r+\lambda^2/2}.
\]
Choosing \(\lambda=r\) gives the claim.
\end{proof}

\begin{lemma}[Gaussian quantile density ratio estimates]
\label{lem:normal-hazard}
Define $\Lambda(u)=\frac{\phi(\Phi^{-1}u)}{u}$ and $u\in(0,1)$.
Then there is an absolute constant \(C>0\) such that
\begin{equation}
\Lambda(u)\le C\sqrt{-\log u},
        \qquad 0<u\le\frac{1}{2},
        \label{eq:A-3}
\end{equation}
and, writing \(\delta=1-u\),
\begin{equation}
\Lambda(u)\le C\delta\sqrt{\log(e/\delta)},
        \qquad \frac{1}{2}<u<1.
        \label{eq:A-4}
\end{equation}
\end{lemma}

\begin{proof}[Proof of Lemma~\ref{lem:normal-hazard}]
First suppose \(0<u\le1/2\). Write $u=\Phi(-r)$, $r\ge0$. 
Then 
\[\Lambda(u)=\frac{\phi(r)}{\Phi(-r)}\].  
By the lower bound in Lemma~\ref{lem:mills}, $\Phi(-r)\ge c\frac{\phi(r)}{1+r}$,  
hence 
\[ \Lambda(u)\le C(1+r) \].
By Lemma~\ref{lem:gaussian-chernoff}, $u=\Phi(-r)\le e^{-r^2/2}$,  
so $-\log u\ge r^2/2$. 
Since \(u\le1/2\), also \(-\log u\ge\log2\). Therefore $1+r\le C\sqrt{-\log u}$. 
This proves \eqref{eq:A-3}.

Now suppose \(1/2<u<1\). Let $\delta=1-u$.
Write $u=\Phi(r)$, $r>0$.
Then
\[
        \delta=\Phi(-r),
        \qquad
        \Lambda(u)=\frac{\phi(r)}{\Phi(r)}\le2\phi(r).
\]
By the lower bound in Lemma~\ref{lem:mills}, $\delta=\Phi(-r)\ge c\frac{\phi(r)}{1+r}$, 
so $ \phi(r)\le C\delta(1+r)$.  
By Lemma~\ref{lem:gaussian-chernoff},
\[
        \delta=\Phi(-r)\le e^{-r^2/2},
\]
so
\[
        r\le C\sqrt{\log(e/\delta)}.
\]
Thus $ \Lambda(u)\le C\delta(1+r)
        \le C\delta\sqrt{\log(e/\delta)}.$
This proves \eqref{eq:A-4}.
\end{proof}

\begin{lemma}[Tail to density-ratio comparison]
\label{lem:tail-to-hazard}
For every \(z\in\R\),
\[
        \min\left\{
        \sqrt{2\log\frac1{\Phi(z)}},
        \frac{\Phi(-z)}{\Phi(z)}
        \sqrt{2\log\frac1{\Phi(-z)}}
        \right\}
        \le
        C\frac{\phi(z)}{\Phi(z)}.
\]
\end{lemma}

\begin{proof}[Proof of Lemma~\ref{lem:tail-to-hazard}]
First suppose \(z\le0\). Write \(z=-r\), \(r\ge0\). Then $\Phi(z)=\Phi(-r)$ and  $\phi(z)=\phi(r)$.  
By the lower bound in Lemma~\ref{lem:mills}, $\Phi(-r)\ge c\frac{\phi(r)}{1+r}$. 
Therefore, $\frac1{\Phi(-r)}
        \le
        C\frac{1+r}{\phi(r)}$.  
Taking logarithms,
\[
        \log\frac1{\Phi(-r)}
        \le
        C+\log(1+r)+\log\frac1{\phi(r)}.
\]
Since $\phi(r)=\frac1{\sqrt{2\pi}}e^{-r^2/2}$,  
we have $ \log\frac1{\phi(r)}
        =
        \frac{r^2}{2}+\frac{1}{2}\log(2\pi)$. 
Moreover,
\[
        \log(1+r)\le r\le\frac{1+r^2}{2},
        \qquad r\ge0.
\]
Hence
\begin{equation}
\sqrt{2\log\frac1{\Phi(-r)}}\le C(1+r).
        \label{eq:A-5}
\end{equation}
Now use the upper bound in Lemma~\ref{lem:mills}: $\Phi(-r)\le C\frac{\phi(r)}{1+r}$.  
Rearranging,
\begin{equation}
\frac{\phi(r)}{\Phi(-r)}\ge c(1+r).
        \label{eq:A-6}
\end{equation}
Combining \eqref{eq:A-5} and \eqref{eq:A-6},
\[
        \sqrt{2\log\frac1{\Phi(z)}}
        =
        \sqrt{2\log\frac1{\Phi(-r)}}
        \le
        C\frac{\phi(r)}{\Phi(-r)}
        =
        C\frac{\phi(z)}{\Phi(z)}.
\]

Now suppose \(z>0\). Write \(r=z>0\). We prove
\begin{equation}
\Phi(-r)\sqrt{2\log\frac1{\Phi(-r)}}\le C\phi(r).
        \label{eq:A-7}
\end{equation}
If \(0<r\le1\), then \(\Phi(-r)\le1/2\),
\[
        \log\frac1{\Phi(-r)}\le \log\frac1{\Phi(-1)}=O(1),
\]
and \(\phi(r)\ge\phi(1)>0\). Thus \eqref{eq:A-7} holds.

If \(r\ge1\), then the upper bound in Lemma~\ref{lem:mills} gives
\[
        \Phi(-r)\le C\frac{\phi(r)}{1+r}\le C\frac{\phi(r)}r.
\]
From the first part of the proof, $\sqrt{2\log\frac1{\Phi(-r)}}\le Cr$.  
Therefore
\[
        \Phi(-r)\sqrt{2\log\frac1{\Phi(-r)}}
        \le
        C\frac{\phi(r)}r\cdot r
        =
        C\phi(r).
\]
Dividing by \(\Phi(z)\) yields
\[
        \frac{\Phi(-z)}{\Phi(z)}
        \sqrt{2\log\frac1{\Phi(-z)}}
        \le
        C\frac{\phi(z)}{\Phi(z)}.
\]
This proves the lemma.
\end{proof}

\begin{lemma}[Quantile-log comparison]
\label{lem:quantile-log}
For every \(u\in(0,1)\),
\[
        1+\log\frac1{\min\{u,1-u\}}
        \le
        C(1+|\Phi^{-1}(u)|^2).
\]
Consequently, $  \left(
        1+\log\frac1{\min\{u,1-u\}}
        \right)^{1/2}
        \le
        C(1+|\Phi^{-1}(u)|)$.
\end{lemma}

\begin{proof}[Proof of Lemma~\ref{lem:quantile-log}]
Let \(z=\Phi^{-1}(u)\). If \(z\ge0\), then \(\min\{u,1-u\}=1-u=\Phi(-z)\). The lower bound in Lemma~\ref{lem:mills} gives $  \Phi(-z)\ge c\frac{\phi(z)}{1+z}$. 
Thus
\[
        \log\frac1{\Phi(-z)}
        \le
        C+\log(1+z)+\log\frac1{\phi(z)}
        \le
        C(1+z^2).
\]
The case \(z<0\) is identical with \(r=-z\).
\end{proof}

\subsection{Elementary Entropy and Summation Bounds}

\begin{lemma}
\label{lem:one-minus-log}
For every \(u\in(0,1]\), $1-u\le -\log u.$
\end{lemma}

\begin{proof}[Proof of Lemma~\ref{lem:one-minus-log}]
Let $ g(u)=-\log u-(1-u).$ 
Then
\[
        g'(u)=1-\frac1u\le0
        \qquad (0<u\le1),
\]
and \(g(1)=0\). Hence \(g(u)\ge0\) for \(0<u\le1\).
\end{proof}

\begin{lemma}[Entropy bound]
\label{lem:entropy}
Let \(m\ge1\).
If \(p_1,\ldots,p_m\ge0\) and
\(\sum_{j=1}^m p_j=1\), then
\[
\sum_{j=1}^m
p_j\log\frac1{p_j}
\le
\log m,
\]
with the convention
\(0\log(1/0)=0\).
\end{lemma}

\begin{proof}[Proof of Lemma~\ref{lem:entropy}]
Let $q_j
:=
\frac1m$ and $J
:=
\left\{
j\in[m]:
p_j>0
\right\}$.  
Since
\(\sum_jp_j=1\),
the set \(J\) is nonempty.  Define $D(p\|q)
:=
\sum_{j\in J}
p_j
\log\frac{p_j}{q_j}$. 
By concavity of the logarithm,
\[
\begin{aligned}
-D(p\|q)=
\sum_{j\in J}
p_j
\log\frac{q_j}{p_j}
\le
\log\left(
\sum_{j\in J}
p_j
\frac{q_j}{p_j}
\right)
=
\log\left(
\sum_{j\in J}
q_j
\right)
\le
0.
\end{aligned}
\]
Hence
\(D(p\|q)\ge0\). 
Since
\(q_j=1/m\),
\[
\begin{aligned}
D(p\|q)
=
\sum_{j\in J}
p_j\log p_j
+
\log m
=
\log m
-
\sum_{j=1}^m
p_j\log\frac1{p_j},
\end{aligned}
\]
where the zero terms are interpreted according to the stated
convention.  Rearranging proves the claim.
\end{proof}

\begin{lemma}
\label{lem:delta-L}
For \(0\le\Delta\le k\) and \(L\ge\Delta\),
\[
        \Delta\sqrt{\log(ek/\Delta)}
        \le
        C(1+L)\sqrt{\log(ek)}.
\]
The expression is interpreted as \(0\) when \(\Delta=0\).
\end{lemma}

\begin{proof}[Proof of Lemma~\ref{lem:delta-L}]
Assume \(\Delta>0\). If \(\Delta\le1\), define
\[
        f(t)=t\sqrt{\log(ek/t)},
        \qquad 0<t\le1.
\]
Let $\ell(t)=\log(ek/t)$. 
Then
\[
        f'(t)=\sqrt{\ell(t)}-\frac1{2\sqrt{\ell(t)}}.
\]
Since \(t\le1\), \(\ell(t)\ge\log(ek)\ge1\), and hence \(f'(t)\ge0\). Therefore
\[
        \Delta\sqrt{\log(ek/\Delta)}
        \le
        \sqrt{\log(ek)}
        \le
        (1+L)\sqrt{\log(ek)}.
\]

If \(\Delta>1\), then $\log(ek/\Delta)\le\log(ek)$, 
and since \(\Delta\le L\),
\[
        \Delta\sqrt{\log(ek/\Delta)}
        \le
        L\sqrt{\log(ek)}
        \le
        (1+L)\sqrt{\log(ek)}.
\]
\end{proof}

\begin{lemma}[Second moment from Gaussian union tail]
\label{lem:gaussian-max-moment}
Suppose a real random variable \(\breve{\zeta}\) satisfies, for all \(r\ge0\), $ \Pp(|\breve{\zeta}|\ge r)\le \min\{1,2k\Phi(-r)\}$.  
Then $ \E \breve{\zeta}^2\le C\log(ek)$.  
\end{lemma}

\begin{proof}[Proof of Lemma~\ref{lem:gaussian-max-moment}]
By the tail integral formula,
\[
        \E \breve{\zeta}^2
        =
        \int_0^\infty 2r\,\Pp(|\breve{\zeta}|\ge r)\,\mathrm dr.
\]
Let $R=\sqrt{2\log(2ek)}$.  
Then
\[
        \int_0^R2r\,\mathrm dr=R^2=O(\log(ek)).
\]
For \(r\ge R\), Lemma~\ref{lem:gaussian-chernoff} gives \(\Phi(-r)\le e^{-r^2/2}\), so
\[
\begin{aligned}
        \int_R^\infty2r\,\Pp(|\breve{\zeta}|\ge r)\,\mathrm dr
        \le
        \int_R^\infty4kr e^{-r^2/2}\,\mathrm dr
        =
        4k e^{-R^2/2}
        \le C.
\end{aligned}
\]
Thus \(\E \breve{\zeta}^2\le C\log(ek)\).
\end{proof}

\subsection{Strong Log-Concavity, Smoothing, and Posterior Energy}

\begin{lemma}[Smoothness of dominated Gaussian-convex marginals]
\label{lem:smooth-marginal-regularity}
Let \(p\ge1\) and \(q\ge0\). Let $G:
\mathbb R^p\times\mathbb R^q
\to
\mathbb R$ 
be convex and \(C^2\).  Assume that, for every compact set $\mathcal K
\subseteq
\mathbb R^p$ 
and every pair of directions $\mathbf a,\mathbf b
\in
\mathbb R^p$, 
there exists $M_{\mathcal K,\mathbf a,\mathbf b}
\in
L^1(\mathbb R^q)$
such that, for all
\(\mathbf t\in\mathcal K\)
and
\(\mathbf r\in\mathbb R^q\),
\begin{equation}
\begin{aligned}
&
e^{-G(\mathbf t,\mathbf r)}
+
\left|
\partial_{\mathbf a}
G(\mathbf t,\mathbf r)
\right|
e^{-G(\mathbf t,\mathbf r)}
\\
&\quad+
\left(
\left|
\partial_{\mathbf a}G(\mathbf t,\mathbf r)
\,
\partial_{\mathbf b}G(\mathbf t,\mathbf r)
\right|
+
\left|
\partial_{\mathbf a\mathbf b}
G(\mathbf t,\mathbf r)
\right|
\right)
e^{-G(\mathbf t,\mathbf r)}
\\
&\le
M_{\mathcal K,\mathbf a,\mathbf b}(\mathbf r).
\end{aligned}
\label{eq:MD-condition}
\end{equation}
Define $A(\mathbf t)
:=
\int_{\mathbb R^q}
e^{-G(\mathbf t,\mathbf r)}
\,\mathrm d\mathbf r$. 
Then \(A\) is positive and belongs to
\(C^2(\mathbb R^p)\).  Moreover, \(A\) is log-concave, so $\Psi(\mathbf t)
:=
\log A(\mathbf t)$ 
belongs to \(C^2(\mathbb R^p)\) and is concave.  In particular,
\[\nabla^2\Psi(\mathbf t)
\preceq
\mathbf O_p.\]
\end{lemma}

\begin{proof}[Proof of Lemma~\ref{lem:smooth-marginal-regularity}]
Taking
\(\mathcal K=\{\mathbf t\}\)
in
\eqref{eq:MD-condition},
the first term on the left-hand side gives an integrable majorant for
\(\mathbf r\mapsto e^{-G(\mathbf t,\mathbf r)}\).
Hence $A(\mathbf t)<\infty$. 
Since the integrand is strictly positive, $A(\mathbf t)>0.$

If \(q=0\), then, under the zero-dimensional convention, we have  $A(\mathbf t)
=
e^{-G(\mathbf t,0)}$ and $\Psi(\mathbf t)
=
-G(\mathbf t,0)$.  
Hence
\(A\in C^2(\mathbb R^p)\).
Moreover,
\(\mathbf t\mapsto G(\mathbf t,0)\)
is convex, so
\(\Psi\) is concave and
\[
\nabla^2\Psi(\mathbf t)
\preceq
\mathbf O_p.
\]
Thus the conclusion is immediate when \(q=0\).

Assume henceforth that \(q\ge1\).   
The function $(\mathbf t,\mathbf r)
\to
e^{-G(\mathbf t,\mathbf r)}$ 
is jointly log-concave because \(G\) is convex.  Hence Prékopa's theorem
implies that \(A\) is log-concave.  Therefore $\Psi
=
\log A$ 
is concave. 
It remains to justify the \(C^2\)-regularity.  Set
\[
F(\mathbf t,\mathbf r)
:=
e^{-G(\mathbf t,\mathbf r)}.
\]
For every direction
\(\mathbf a\in\mathbb R^p\),
\[
\partial_{\mathbf a}
F(\mathbf t,\mathbf r)
=
-
\partial_{\mathbf a}
G(\mathbf t,\mathbf r)
e^{-G(\mathbf t,\mathbf r)}.
\]
For directions
\(\mathbf a,\mathbf b\in\mathbb R^p\),
\[
\partial_{\mathbf a\mathbf b}
F(\mathbf t,\mathbf r)
=
\left[
\partial_{\mathbf a}G(\mathbf t,\mathbf r)
\partial_{\mathbf b}G(\mathbf t,\mathbf r)
-
\partial_{\mathbf a\mathbf b}G(\mathbf t,\mathbf r)
\right]
e^{-G(\mathbf t,\mathbf r)}.
\]
The domination condition
\eqref{eq:MD-condition}
permits differentiation under the integral sign.  Thus
\[
\partial_{\mathbf a}A(\mathbf t)
=
-
\int_{\mathbb R^q}
\partial_{\mathbf a}G(\mathbf t,\mathbf r)
e^{-G(\mathbf t,\mathbf r)}
\,\mathrm d\mathbf r,
\]
and
\[
\begin{aligned}
\partial_{\mathbf a\mathbf b}A(\mathbf t)=
\int_{\mathbb R^q}
\left[
\partial_{\mathbf a}G(\mathbf t,\mathbf r)
\partial_{\mathbf b}G(\mathbf t,\mathbf r)
-
\partial_{\mathbf a\mathbf b}G(\mathbf t,\mathbf r)
\right]
\cdot
e^{-G(\mathbf t,\mathbf r)}
\,\mathrm d\mathbf r.
\end{aligned}
\]
The same domination, together with dominated convergence, implies that
these derivatives are continuous.  Hence $A\in C^2(\mathbb R^p)$.  
Since \(A>0\), $\Psi
=
\log A$ 
also belongs to \(C^2(\mathbb R^p)\), and
\[
\partial_{\mathbf a\mathbf b}\Psi(\mathbf t)
=
\frac{
\partial_{\mathbf a\mathbf b}A(\mathbf t)
}{
A(\mathbf t)
}
-
\frac{
\partial_{\mathbf a}A(\mathbf t)
\partial_{\mathbf b}A(\mathbf t)
}{
A(\mathbf t)^2
}.
\]
Finally, since \(\Psi\) is \(C^2\) and concave,
\[
\nabla^2\Psi(\mathbf t)
\preceq
\mathbf O_p.
\]
\end{proof}

\begin{lemma}[Domination for Gaussian-fiber marginals]
\label{lem:gaussian-fiber-domination}
Let $Q\in C^2(\mathbb R^m)$ 
be convex, and suppose that, for some finite \(L\),
\[
\mathbf O_m
\preceq
\nabla^2Q(\mathbf z)
\preceq
L\mathbf I_m,
\qquad
\mathbf z\in\mathbb R^m.
\] 
First, let \(m,p\ge1\) and \(q\ge0\). Let $\mathbf t\in\mathbb R^p$, $\mathbf r\in\mathbb R^q$ 
and let $\mathcal A:
\mathbb R^p
\to
\mathbb R^m$, $\mathcal B:
\mathbb R^q
\to
\mathbb R^m$ 
be linear maps.  Define
\[
G_1(\mathbf t,\mathbf r)
:=
\frac12\|\mathbf r\|^2
+
Q\left(
\mathcal A\mathbf t
+
\mathcal B\mathbf r
\right).
\]

Second, in the case \(m=2\), let $\mathbf t
=
(s,y)
\in
\mathbb R^2$ for  $w\in\mathbb R$, 
and define
\[
G_2((s,y),w)
:=
\frac12(w-s)^2
+
Q(y,w).
\]

Both \(G_1\) and \(G_2\) satisfy the domination condition
\eqref{eq:MD-condition}
on every compact parameter set.

Moreover, if $A_1(\mathbf t)
:=
\int_{\mathbb R^q}
e^{-G_1(\mathbf t,\mathbf r)}
\,\mathrm d\mathbf r$ 
and $\widetilde Q_1(\mathbf t)
:=
-\log A_1(\mathbf t)$,  
then $\widetilde Q_1
\in
C^2(\mathbb R^p)$ 
is convex and
\begin{equation}
\mathbf O_p
\preceq
\nabla^2\widetilde Q_1(\mathbf t)
\preceq
L\mathcal A^\ast\mathcal A
\preceq
L\|\mathcal A\|_{\mathrm{op}}^2
\mathbf I_p.
\label{eq:marginal-residual-hessian}
\end{equation}
\end{lemma}

\begin{proof}[Proof of Lemma~\ref{lem:gaussian-fiber-domination}]
If \(q=0\), then
\(\mathcal B\) is the unique zero-dimensional linear map and $G_1(\mathbf t,0)
=
Q(\mathcal A\mathbf t)$. 
Hence
\[
A_1(\mathbf t)
=
e^{-Q(\mathcal A\mathbf t)},
\qquad
\widetilde Q_1(\mathbf t)
=
Q(\mathcal A\mathbf t).
\]
The domination condition is immediate, because on every compact
\(\mathcal K\subseteq\mathbb R^p\)
all relevant derivatives are bounded and
\(L^1(\mathbb R^0)\)
consists simply of finite functions at the unique point. 
Moreover, $\nabla^2\widetilde Q_1(\mathbf t)
=
\mathcal A^\ast
\nabla^2Q(\mathcal A\mathbf t)
\mathcal A$,  
and therefore $\mathbf O_p
\preceq
\nabla^2\widetilde Q_1(\mathbf t)
\preceq
L\mathcal A^\ast\mathcal A$.  
Thus all conclusions hold when \(q=0\). 

Assume henceforth that \(q\ge1\). 
We first consider \(G_1\). 
Fix a compact set $\mathcal K
\subseteq
\mathbb R^p$  
and directions $\mathbf a,\mathbf b
\in
\mathbb R^p$.  
Set $R_{\mathcal K}
:=
\sup_{\mathbf t\in\mathcal K}
\|\mathcal A\mathbf t\|
<
\infty$. 
Since
\[
\mathbf O_m
\preceq
\nabla^2Q
\preceq
L\mathbf I_m,
\]
the gradient of \(Q\) is globally \(L\)-Lipschitz.  Hence
\[
\|\nabla Q(\mathbf z)\|
\le
\|\nabla Q(\mathbf 0_m)\|
+
L\|\mathbf z\|,
\qquad
\mathbf z\in\mathbb R^m.
\]
For $\mathbf z
=
\mathcal A\mathbf t
+
\mathcal B\mathbf r$, 
with
\(\mathbf t\in\mathcal K\),
this gives
\begin{equation}\label{eq:gaussian-fiber-gradient-growth}
\begin{aligned}
\left\|
\nabla Q(
\mathcal A\mathbf t+\mathcal B\mathbf r)
\right\|
&\le
\|\nabla Q(\mathbf 0_m)\|
+
L
\left\|
\mathcal A\mathbf t+\mathcal B\mathbf r
\right\|
\\
&\le
\|\nabla Q(\mathbf 0_m)\|
+
LR_{\mathcal K}
+
L\|\mathcal B\|_{\mathrm{op}}
\|\mathbf r\|
\\
&\le
C_{\mathcal K}
\left(
1+\|\mathbf r\|
\right).
\end{aligned}
\end{equation}
By convexity of \(Q\),
\[
Q(\mathbf z)
\ge
Q(\mathbf 0_m)
+
\left\langle
\nabla Q(\mathbf 0_m),
\mathbf z
\right\rangle.
\]
Therefore
\[
\begin{aligned}
G_1(\mathbf t,\mathbf r)
&\ge
\frac12\|\mathbf r\|^2
+
Q(\mathbf 0_m)
+
\left\langle
\nabla Q(\mathbf 0_m),
\mathcal A\mathbf t+\mathcal B\mathbf r
\right\rangle
\\
&\ge
\frac12\|\mathbf r\|^2
-
C_{\mathcal K}
-
\left\|
\mathcal B^\ast\nabla Q(\mathbf 0_m)
\right\|
\|\mathbf r\|.
\end{aligned}
\]
Using
\[
c\|\mathbf r\|
\le
\frac14\|\mathbf r\|^2
+
c^2,
\]
we obtain
\[
G_1(\mathbf t,\mathbf r)
\ge
\frac14\|\mathbf r\|^2
-
C_{\mathcal K}'.
\]
Consequently,
\begin{equation}
e^{-G_1(\mathbf t,\mathbf r)}
\le
C_{\mathcal K}''
e^{-\|\mathbf r\|^2/4},
\qquad
\mathbf t\in\mathcal K.
\label{eq:G1-basic-envelope}
\end{equation}

The directional parameter derivatives are
\[
\partial_{\mathbf a}
G_1(\mathbf t,\mathbf r)
=
\left\langle
\nabla Q(
\mathcal A\mathbf t+\mathcal B\mathbf r),
\mathcal A\mathbf a
\right\rangle,
\]
and
\[
\partial_{\mathbf a\mathbf b}
G_1(\mathbf t,\mathbf r)
=
\left\langle
\mathcal A\mathbf a,
\nabla^2Q(
\mathcal A\mathbf t+\mathcal B\mathbf r)
\mathcal A\mathbf b
\right\rangle.
\]
By
\eqref{eq:gaussian-fiber-gradient-growth},
\[
\left|
\partial_{\mathbf a}
G_1(\mathbf t,\mathbf r)
\right|
\le
C_{\mathcal K,\mathbf a}
\left(
1+\|\mathbf r\|
\right),
\]
while
\[
\left|
\partial_{\mathbf a\mathbf b}
G_1(\mathbf t,\mathbf r)
\right|
\le
L
\|\mathcal A\mathbf a\|
\|\mathcal A\mathbf b\|.
\]
Together with
\eqref{eq:G1-basic-envelope},
this gives
\[
\begin{aligned}
&
e^{-G_1}
+
\left|
\partial_{\mathbf a}G_1
\right|
e^{-G_1}
\\
&\quad+
\left(
\left|
\partial_{\mathbf a}G_1
\partial_{\mathbf b}G_1
\right|
+
\left|
\partial_{\mathbf a\mathbf b}G_1
\right|
\right)
e^{-G_1}
\\
&\le
C_{\mathcal K,\mathbf a,\mathbf b}
\left(
1+\|\mathbf r\|^2
\right)
e^{-\|\mathbf r\|^2/4}.
\end{aligned}
\]
The right-hand side is integrable over \(\mathbb R^q\), proving the
domination condition for \(G_1\).

We now consider \(G_2\).  Fix a compact set $\mathcal K
\subseteq
\mathbb R^2$,  
and write
\[
\mathbf a
=
(a_s,a_y),
\qquad
\mathbf b
=
(b_s,b_y).
\]
Using
\[
\frac12(w-s)^2
\ge
\frac14w^2
-
\frac12s^2
\]
and absorbing the linear term in \(w\), we obtain
\[
G_2((s,y),w)
\ge
\frac18w^2
-
C_{\mathcal K}'.
\]
Thus
\[
e^{-G_2((s,y),w)}
\le
C_{\mathcal K}''
e^{-w^2/8}.
\]

The parameter derivatives are
\[
\partial_{\mathbf a}
G_2((s,y),w)
=
-a_s(w-s)
+
a_y
\partial_yQ(y,w),
\]
and
\[
\partial_{\mathbf a\mathbf b}
G_2((s,y),w)
=
a_sb_s
+
a_yb_y
\partial_{yy}Q(y,w).
\]
The global Hessian bound gives
\[
|\partial_yQ(y,w)|
\le
C_{\mathcal K}(1+|w|),
\]
and
\[
|\partial_{yy}Q(y,w)|
\le
L.
\]
Hence all terms in
\eqref{eq:MD-condition}
are bounded by
\[
C_{\mathcal K,\mathbf a,\mathbf b}
(1+w^2)e^{-w^2/8},
\]
which is integrable.  This proves the domination assertion for \(G_2\).

It remains to establish
\eqref{eq:marginal-residual-hessian}.
Since \(G_1\) is jointly convex, $(\mathbf t,\mathbf r)
\to
e^{-G_1(\mathbf t,\mathbf r)}$ 
is jointly log-concave.  Prékopa's theorem implies that \(A_1\) is
log-concave, and therefore $\widetilde Q_1
=
-\log A_1$ 
is convex.

The domination assertion and
Lemma~\ref{lem:smooth-marginal-regularity}
give $A_1,\widetilde Q_1
\in
C^2(\mathbb R^p).$ 
Since \(\widetilde Q_1\) is convex,
\[
\nabla^2\widetilde Q_1(\mathbf t)
\succeq
\mathbf O_p.
\] 
For each \(\mathbf t\), define
\[
\nu_{\mathbf t}(\mathrm d\mathbf r)
:=
A_1(\mathbf t)^{-1}
e^{-G_1(\mathbf t,\mathbf r)}
\,\mathrm d\mathbf r.
\]
Differentiating the log-partition function gives
\[
\nabla\widetilde Q_1(\mathbf t)
=
\mathbb E_{\nu_{\mathbf t}}
\left[
\nabla_{\mathbf t}
G_1(\mathbf t,\mathbf r)
\right],
\]
and
\begin{equation}
\begin{aligned}
\nabla^2\widetilde Q_1(\mathbf t)
&=
\mathbb E_{\nu_{\mathbf t}}
\left[
\nabla_{\mathbf t}^2
G_1(\mathbf t,\mathbf r)
\right]
\\
&\quad-
\operatorname{Cov}_{\nu_{\mathbf t}}
\left(
\nabla_{\mathbf t}
G_1(\mathbf t,\mathbf r)
\right).
\end{aligned}
\label{eq:log-partition-hessian}
\end{equation}
Here
\[
\nabla_{\mathbf t}
G_1(\mathbf t,\mathbf r)
=
\mathcal A^\ast
\nabla Q(
\mathcal A\mathbf t+\mathcal B\mathbf r),
\]
and
\[
\nabla_{\mathbf t}^2
G_1(\mathbf t,\mathbf r)
=
\mathcal A^\ast
\nabla^2Q(
\mathcal A\mathbf t+\mathcal B\mathbf r)
\mathcal A.
\]
A covariance matrix is positive semidefinite.  Hence
\[
\begin{aligned}
\nabla^2\widetilde Q_1(\mathbf t)
&\preceq
\mathbb E_{\nu_{\mathbf t}}
\left[
\mathcal A^\ast
\nabla^2Q(
\mathcal A\mathbf t+\mathcal B\mathbf r)
\mathcal A
\right]
\\
&\preceq
L\mathcal A^\ast\mathcal A
\\
&\preceq
L\|\mathcal A\|_{\mathrm{op}}^2
\mathbf I_p.
\end{aligned}
\]
This completes the proof.
\end{proof}

\begin{lemma}[Exponential moments from Brascamp--Lieb]
\label{lem:bl-herbst}
Let \(\mu\) have a positive \(C^2\) density on \(\mathbb R^r\) of the
form
\[
\mu(\mathrm d\mathbf u)
=
\mathcal Z^{-1}
e^{-V(\mathbf u)}
\,\mathrm d\mathbf u,
\qquad
\nabla^2V(\mathbf u)
\succeq
\alpha\mathbf I_r
\]
for some \(\alpha>0\).  Let $\rvec u
\sim
\mu.$
Then, for every
\(\mathbf w\in\mathbb R^r\)
and every
\(\lambda\in\mathbb R\),
\[
\mathbb E_\mu
\exp\left(
\lambda
\left\langle
\rvec u-\mathbb E_\mu\rvec u,
\mathbf w
\right\rangle
\right)
\le
\exp\left(
\frac{
\lambda^2\|\mathbf w\|^2
}{
2\alpha
}
\right).
\]
\end{lemma}

\begin{proof}[Proof of Lemma~\ref{lem:bl-herbst}]
Define
\[
\psi(\lambda)
:=
\log
\mathbb E_\mu
\exp\left(
\lambda
\left\langle
\rvec u-\mathbb E_\mu\rvec u,
\mathbf w
\right\rangle
\right).
\]
Strong convexity implies that the exponential moment is finite for every
\(\lambda\in\mathbb R\).

Define the tilted probability measure
\[
\mu_\lambda(\mathrm d\mathbf u)
:=
\frac{
e^{\lambda\langle\mathbf u,\mathbf w\rangle}
}{
\mathbb E_\mu
e^{\lambda\langle\rvec u,\mathbf w\rangle}
}
\mu(\mathrm d\mathbf u).
\]
Its potential is
\[
V_\lambda(\mathbf u)
=
V(\mathbf u)
-
\lambda
\langle\mathbf u,\mathbf w\rangle
+
\mathrm{const},
\]
and hence
\[
\nabla^2V_\lambda(\mathbf u)
=
\nabla^2V(\mathbf u)
\succeq
\alpha\mathbf I_r.
\]

Differentiating the log-partition function gives
\[
\psi''(\lambda)
=
\operatorname{Var}_{\mu_\lambda}
\left(
\left\langle
\rvec u,\mathbf w
\right\rangle
\right).
\]
Brascamp--Lieb, applied under \(\mu_\lambda\), yields
\[
\psi''(\lambda)
\le
\frac{\|\mathbf w\|^2}{\alpha}.
\]
Since
\[
\psi(0)=\psi'(0)=0,
\]
for \(\lambda\ge0\),
\[
\begin{aligned}
\psi(\lambda)
&=
\int_0^\lambda
(\lambda-s)
\psi''(s)
\,\mathrm ds
\\
&\le
\frac{
\lambda^2\|\mathbf w\|^2
}{
2\alpha
}.
\end{aligned}
\]
The case \(\lambda<0\) follows by replacing
\(\mathbf w\) with \(-\mathbf w\).
\end{proof}

\begin{lemma}[Gaussian smoothing and Hessian bounds]
\label{lem:gaussian-smoothing-slc}
Let $H
\subseteq
\mathbb R^d$
be a nonempty affine subspace, and let \(\rvec u\) have law
\[
\mu(\mathrm d\mathbf x)
=
\mathcal Z_H^{-1}
e^{-\mathcal V_H(\mathbf x)}
\,\sigma_H(\mathrm d\mathbf x),
\]
where
\[
\mathcal V_H:
H
\to
\mathbb R\cup\{+\infty\}
\]
is closed and proper, and
\[
\mathbf x
\longmapsto
\mathcal V_H(\mathbf x)
-
\frac12\|\mathbf x\|^2
\]
is convex on \(H\). 
Let $\breve{\boldsymbol\zeta}
\sim
N(\mathbf 0_d,\mathbf I_d)$
be independent of \(\rvec u\), and, for \(\eta>0\), set $\rvec u_\eta
:=
\rvec u
+
\eta
\breve{\boldsymbol\zeta}$.  
Define
\begin{equation}
\mathcal V_\eta(\mathbf y)
:=
-\log
\int_H
\exp\left(
-\mathcal V_H(\mathbf x)
-
\frac{
\|\mathbf y-\mathbf x\|^2
}{
2\eta^2
}
\right)
\,\sigma_H(\mathrm d\mathbf x).
\label{eq:smoothing-potential-explicit}
\end{equation}
Then \(\rvec u_\eta\) has a positive \(C^\infty\) density $p_\eta(\mathbf y)
=
\mathcal Z_\eta^{-1}
e^{-\mathcal V_\eta(\mathbf y)}$ 
on \(\mathbb R^d\), where $\mathcal Z_\eta
=
\mathcal Z_H
(2\pi\eta^2)^{d/2}$. 
Moreover,
\begin{equation}
\frac1{1+\eta^2}
\mathbf I_d
\preceq
\nabla^2\mathcal V_\eta(\mathbf y)
\preceq
\frac1{\eta^2}
\mathbf I_d,
\qquad
\mathbf y\in\mathbb R^d.
\label{eq:smoothing-two-sided-hessian}
\end{equation}

Set $\alpha_\eta
:=
\frac1{1+\eta^2}$, $\rvec u_{\eta,\mathrm{nor}}
:=
\sqrt{\alpha_\eta}\,
\rvec u_\eta$, 
and define
\begin{equation}
Q_\eta(\mathbf w)
:=
\mathcal V_\eta
\left(
\frac{\mathbf w}{\sqrt{\alpha_\eta}}
\right)
-
\frac12\|\mathbf w\|^2.
\label{eq:normalized-residual-potential}
\end{equation}
Then \(\rvec u_{\eta,\mathrm{nor}}\) has density
\[
p_{\eta,\mathrm{nor}}(\mathbf w)
=
\mathcal Z_{\eta,\mathrm{nor}}^{-1}
\exp\left(
-\frac12\|\mathbf w\|^2
-
Q_\eta(\mathbf w)
\right),
\]
where $\mathcal Z_{\eta,\mathrm{nor}}
=
\alpha_\eta^{d/2}
\mathcal Z_\eta$,  
and \(Q_\eta\in C^\infty(\mathbb R^d)\) is convex with
\begin{equation}
\mathbf O_d
\preceq
\nabla^2Q_\eta(\mathbf w)
\preceq
\frac1{\eta^2}
\mathbf I_d.
\label{eq:normalized-residual-hessian}
\end{equation}
\end{lemma}

\begin{proof}[Proof of Lemma~\ref{lem:gaussian-smoothing-slc}]
If
\(\dim H=0\),
write
\(H=\{\mathbf x_0\}\).
Then
\(\rvec u=\mathbf x_0\)
almost surely and $\rvec u_\eta
\sim
N(
\mathbf x_0,
\eta^2\mathbf I_d
)$. 
In this case,
\[
\mathcal V_\eta(\mathbf y)
=
\frac{
\|\mathbf y-\mathbf x_0\|_2^2
}{
2\eta^2
}
+
\mathrm{const},
\]
so $\nabla^2\mathcal V_\eta(\mathbf y)
=
\frac1{\eta^2}\mathbf I_d$.  
Since $\frac1{\eta^2}
\ge
\frac1{1+\eta^2}$,  
the two sided Hessian bound follows.  The normalized residual potential
also satisfies
\[
\nabla^2Q_\eta
=
\frac1{\eta^2}\mathbf I_d.
\]
Thus all conclusions hold in the zero-dimensional case.

Assume henceforth that
\(\dim H\ge1\). 
By convolution with the
\(N(\mathbf 0_d,\eta^2\mathbf I_d)\)
density, the law of \(\rvec u_\eta\) has density
\[
\begin{aligned}
p_\eta(\mathbf y)
&=
\int_H
\frac1{(2\pi\eta^2)^{d/2}}
\exp\left(
-\frac{
\|\mathbf y-\mathbf x\|^2
}{
2\eta^2
}
\right)
\mu(\mathrm d\mathbf x)
\\
&=
\frac1{
\mathcal Z_H
(2\pi\eta^2)^{d/2}
}
\int_H
\exp\left(
-\mathcal V_H(\mathbf x)
-
\frac{
\|\mathbf y-\mathbf x\|^2
}{
2\eta^2
}
\right)
\,\sigma_H(\mathrm d\mathbf x).
\end{aligned}
\]
Thus $p_\eta(\mathbf y)
=
\mathcal Z_\eta^{-1}
e^{-\mathcal V_\eta(\mathbf y)}$.  
The Gaussian kernel is strictly positive, so $p_\eta(\mathbf y)>0$ 
for every \(\mathbf y\).  Every derivative of the Gaussian kernel is a
polynomial times the same Gaussian kernel.  Since \(\mu\) is finite,
differentiation under the integral sign is valid to every order.  Hence
\[
p_\eta,\mathcal V_\eta
\in
C^\infty(\mathbb R^d).
\]

Define
\[
\mathcal W_H(\mathbf x)
:=
\mathcal V_H(\mathbf x)
-
\frac12\|\mathbf x\|^2.
\]
By assumption, \(\mathcal W_H\) is convex on \(H\).  Let $\alpha_\eta
=
\frac1{1+\eta^2}$.  
Completing the square gives
\begin{equation}
\frac12\|\mathbf x\|^2
+
\frac{
\|\mathbf y-\mathbf x\|^2
}{
2\eta^2
}
=
\frac{\alpha_\eta}{2}
\|\mathbf y\|^2
+
\frac{
\|\mathbf x-\alpha_\eta\mathbf y\|^2
}{
2\eta^2\alpha_\eta
}.
\label{eq:smoothing-complete-square}
\end{equation}
Consequently,
\[
e^{-\mathcal V_\eta(\mathbf y)}
=
e^{-\alpha_\eta\|\mathbf y\|^2/2}
A_\eta(\mathbf y),
\]
where
\[
A_\eta(\mathbf y)
:=
\int_H
\exp\left(
-\mathcal W_H(\mathbf x)
-
\frac{
\|\mathbf x-\alpha_\eta\mathbf y\|^2
}{
2\eta^2\alpha_\eta
}
\right)
\,\sigma_H(\mathrm d\mathbf x).
\]

After an affine isometric parametrization of \(H\), the integrand above
is jointly log-concave in the affine coordinate and \(\mathbf y\).
Prékopa's theorem therefore implies that \(A_\eta\) is log-concave.
Hence $-\log A_\eta$ 
is convex.  Since
\[
\mathcal V_\eta(\mathbf y)
=
\frac{\alpha_\eta}{2}
\|\mathbf y\|^2
-
\log A_\eta(\mathbf y),
\]
we obtain
\[
\nabla^2\mathcal V_\eta(\mathbf y)
\succeq
\alpha_\eta\mathbf I_d.
\]

For the upper Hessian bound, define the posterior probability measure
\[
\pi_{\eta,\mathbf y}(\mathrm d\mathbf x)
:=
\frac{
\exp\left(
-\frac{
\|\mathbf y-\mathbf x\|^2
}{
2\eta^2
}
\right)
\mu(\mathrm d\mathbf x)
}{
\displaystyle
\int_H
\exp\left(
-\frac{
\|\mathbf y-\mathbf x'\|^2
}{
2\eta^2
}
\right)
\mu(\mathrm d\mathbf x')
}.
\]
Let
\[
\mathbf m_\eta(\mathbf y)
:=
\int_H
\mathbf x\,
\pi_{\eta,\mathbf y}(\mathrm d\mathbf x).
\]
Differentiation of the Gaussian convolution gives
\[
\nabla\mathcal V_\eta(\mathbf y)
=
\frac1{\eta^2}
\left(
\mathbf y-\mathbf m_\eta(\mathbf y)
\right).
\]
Moreover,
\[
D\mathbf m_\eta(\mathbf y)
=
\frac1{\eta^2}
\operatorname{Cov}_{\pi_{\eta,\mathbf y}}(\mathbf x).
\]
Therefore
\[
\nabla^2\mathcal V_\eta(\mathbf y)
=
\frac1{\eta^2}\mathbf I_d
-
\frac1{\eta^4}
\operatorname{Cov}_{\pi_{\eta,\mathbf y}}(\mathbf x).
\]
Since covariance matrices are positive semidefinite,
\[
\nabla^2\mathcal V_\eta(\mathbf y)
\preceq
\frac1{\eta^2}\mathbf I_d.
\]
This proves
\eqref{eq:smoothing-two-sided-hessian}.

Finally, the change-of-variables formula gives
\[
p_{\eta,\mathrm{nor}}(\mathbf w)
=
\alpha_\eta^{-d/2}
p_\eta
\left(
\frac{\mathbf w}{\sqrt{\alpha_\eta}}
\right).
\]
Using
\eqref{eq:normalized-residual-potential},
this becomes
\[
p_{\eta,\mathrm{nor}}(\mathbf w)
=
\mathcal Z_{\eta,\mathrm{nor}}^{-1}
\exp\left(
-\frac12\|\mathbf w\|^2
-
Q_\eta(\mathbf w)
\right).
\]
Furthermore,
\[
\nabla^2Q_\eta(\mathbf w)
=
\frac1{\alpha_\eta}
\nabla^2\mathcal V_\eta
\left(
\frac{\mathbf w}{\sqrt{\alpha_\eta}}
\right)
-
\mathbf I_d.
\]
The lower bound in
\eqref{eq:smoothing-two-sided-hessian}
gives
\[
\nabla^2Q_\eta(\mathbf w)
\succeq
\mathbf O_d.
\]
The upper bound gives
\[
\begin{aligned}
\nabla^2Q_\eta(\mathbf w)
\preceq
\left(
\frac1{\alpha_\eta\eta^2}
-
1
\right)
\mathbf I_d
=
\left(
\frac{1+\eta^2}{\eta^2}
-
1
\right)
\mathbf I_d
=
\frac1{\eta^2}
\mathbf I_d.
\end{aligned}
\]
This proves
\eqref{eq:normalized-residual-hessian}.
\end{proof}

\begin{lemma}[Subgaussian linear marginals on affine supports]
\label{lem:slc-subgaussian-affine}
Let $H
=
\mathbf b_0+L
\subseteq
\mathbb R^d$ 
be a nonempty affine subspace, and let \(\rvec u\) have law
\[
\mu(\mathrm d\mathbf u)
=
\mathcal Z_H^{-1}
e^{-\mathcal V_H(\mathbf u)}
\,\sigma_H(\mathrm d\mathbf u),
\]
where \(\mathcal V_H\) is closed and proper and
\[
\mathbf u
\longmapsto
\mathcal V_H(\mathbf u)
-
\frac12\|\mathbf u\|^2
\]
is convex on \(H\).  Let $\mathcal P_L$ 
denote the orthogonal projection onto the direction subspace $L
=
H-H$.  
Then, for every
\(\mathbf v\in\mathbb R^d\)
and
\(\lambda\in\mathbb R\),
\begin{equation}
\mathbb E
\exp\left(
\lambda
\left\langle
\rvec u-\mathbb E\rvec u,
\mathbf v
\right\rangle
\right)
\le
\exp\left(
\frac{
\lambda^2
\|\mathcal P_L\mathbf v\|^2
}{
2
}
\right).
\label{eq:affine-subgaussian}
\end{equation}
In particular, if
\(\|\mathbf v\|\le1\),
the right-hand side is at most
\(e^{\lambda^2/2}\).  If
\(\dim H=0\),
the centered linear functional vanishes identically.
\end{lemma}

\begin{proof}[Proof of Lemma~\ref{lem:slc-subgaussian-affine}]
Let $r
:=
\dim L$,  
and choose a linear isometry $\mathcal U:
\mathbb R^r
\to
L$. 
If \(r=0\), then \(\rvec u\) is deterministic and the conclusion is
immediate.

Write $\rvec u
=
\mathbf b_0
+
\mathcal U\rvec z$,  
where \(\rvec z\) is \(1\)-strongly log-concave on
\(\mathbb R^r\).  Set $\mathbf w
:=
\mathcal U^\ast\mathbf v$.  
Then
\[
\left\langle
\rvec u-\mathbb E\rvec u,
\mathbf v
\right\rangle
=
\left\langle
\rvec z-\mathbb E\rvec z,
\mathbf w
\right\rangle.
\]

Let $\breve{\boldsymbol\zeta}
\sim
N(\mathbf 0_r,\mathbf I_r)$ 
be independent of \(\rvec z\), and define
\[
\rvec z_\tau
:=
\rvec z
+
\tau
\breve{\boldsymbol\zeta},
\qquad
\tau>0.
\]
By
Lemma~\ref{lem:gaussian-smoothing-slc},
the law of \(\rvec z_\tau\) is smooth and
\[
(1+\tau^2)^{-1}
\text{-strongly log-concave}.
\]
Applying
Lemma~\ref{lem:bl-herbst}
gives
\[
\mathbb E
\exp\left(
\lambda
\left\langle
\rvec z_\tau-\mathbb E\rvec z_\tau,
\mathbf w
\right\rangle
\right)
\le
\exp\left(
\frac{
\lambda^2(1+\tau^2)
\|\mathbf w\|^2
}{
2
}
\right).
\]

Since $\mathbb E\rvec z_\tau
=
\mathbb E\rvec z$ 
and the two summands are independent,
\[
\begin{aligned}
\mathbb E
\exp\left(
\lambda
\left\langle
\rvec z_\tau-\mathbb E\rvec z_\tau,
\mathbf w
\right\rangle
\right)
=
\mathbb E
\exp\left(
\lambda
\left\langle
\rvec z-\mathbb E\rvec z,
\mathbf w
\right\rangle
\right)
\exp\left(
\frac{
\lambda^2\tau^2
\|\mathbf w\|^2
}{
2
}
\right).
\end{aligned}
\]
Cancelling the Gaussian factor gives
\[
\mathbb E
\exp\left(
\lambda
\left\langle
\rvec z-\mathbb E\rvec z,
\mathbf w
\right\rangle
\right)
\le
\exp\left(
\frac{
\lambda^2
\|\mathbf w\|^2
}{
2
}
\right).
\]
Finally, $\|\mathbf w\|
=
\|\mathcal U^\ast\mathbf v\|
=
\|\mathcal P_L\mathbf v\|$. 
This proves
\eqref{eq:affine-subgaussian}.
\end{proof}

\begin{lemma}[Variational characterization of conditional-mean energy]
\label{lem:conditional-mean-variational}
Let \(\mathsf Z\) be a Polish space, let
\(\mathsf W\) be a Borel random element taking values in
\(\mathsf Z\), and let $\rvec h
\in
L^2(\Omega;\mathbb R^d)$.  
Here $\mathbb E[
\rvec h\mid\mathsf W]$ 
denotes $\mathbb E[
\rvec h\mid\sigma(\mathsf W)].$ 
Then
\begin{equation}
\begin{aligned}
\mathbb E
\left\|
\mathbb E[
\rvec h\mid\mathsf W]
\right\|_2^2
=
\sup_{\varphi\in C_b(\mathsf Z;\mathbb R^d)}
\Biggl\{
2\mathbb E
\left\langle
\rvec h,
\varphi(\mathsf W)
\right\rangle
-
\mathbb E
\left\|
\varphi(\mathsf W)
\right\|_2^2
\Biggr\}.
\end{aligned}
\label{eq:conditional-mean-variational}
\end{equation}
\end{lemma}

\begin{proof}[Proof of Lemma~\ref{lem:conditional-mean-variational}] 
Let $\nu
:=
\Law(\mathsf W)$.  
By the Doob--Dynkin lemma, there exists $\mathbf m
\in
L^2(\nu;\mathbb R^d)$
such that
\[
\mathbb E[
\rvec h\mid\mathsf W]
=
\mathbf m(\mathsf W)
\qquad
\text{almost surely}.
\]

For every $\varphi
\in
C_b(\mathsf Z;\mathbb R^d)$, 
the random vector
\(\varphi(\mathsf W)\)
is \(\sigma(\mathsf W)\)-measurable.  Hence
\[
\mathbb E
\left\langle
\rvec h,
\varphi(\mathsf W)
\right\rangle
=
\mathbb E
\left\langle
\mathbf m(\mathsf W),
\varphi(\mathsf W)
\right\rangle.
\]
Therefore
\[
\begin{aligned}
&
2\mathbb E
\left\langle
\rvec h,
\varphi(\mathsf W)
\right\rangle
-
\mathbb E
\left\|
\varphi(\mathsf W)
\right\|_2^2
\\
&\quad=
\mathbb E
\left\|
\mathbf m(\mathsf W)
\right\|_2^2
-
\mathbb E
\left\|
\mathbf m(\mathsf W)
-
\varphi(\mathsf W)
\right\|_2^2
\\
&\quad\le
\mathbb E
\left\|
\mathbf m(\mathsf W)
\right\|_2^2.
\end{aligned}
\]
This proves the upper bound in
\eqref{eq:conditional-mean-variational}.

Because \(\mathsf Z\) is Polish, every Borel probability measure on
\(\mathsf Z\) is regular, and $C_b(\mathsf Z;\mathbb R^d)$ 
is dense in $L^2(\nu;\mathbb R^d)$.  
Hence there exists a sequence $\varphi_n
\in
C_b(\mathsf Z;\mathbb R^d)$ 
such that
\[
\varphi_n
\to
\mathbf m
\qquad
\text{in }L^2(\nu;\mathbb R^d).
\]
Consequently,
\[
\mathbb E
\left\|
\mathbf m(\mathsf W)
-
\varphi_n(\mathsf W)
\right\|_2^2
\to0.
\]
Substituting \(\varphi_n\) into the preceding identity and passing to the
limit proves the reverse inequality.
\end{proof}

\begin{lemma}[Lower semicontinuity of posterior energy]
\label{lem:posterior-lsc}
Let $\rvec u_n$, $\rvec u$
be square-integrable random vectors in \(\mathbb R^d\), and let $\mathsf W_n$, $\mathsf W$
be random elements in a Polish space \(\mathsf Z\).  Suppose that
\[
\rvec u_n
\longrightarrow
\rvec u
\quad
\text{in}\quad L^2, \quad \text{and}\quad \mathsf W_n
\longrightarrow
\mathsf W
\quad
\text{almost surely}.
\]
Then
\[
\mathbb E
\left\|
\mathbb E[
\rvec u\mid\mathsf W]
-
\mathbb E\rvec u
\right\|^2
\le
\liminf_{n\to\infty}
\mathbb E
\left\|
\mathbb E[
\rvec u_n\mid\mathsf W_n]
-
\mathbb E\rvec u_n
\right\|^2.
\]
\end{lemma}

\begin{proof}[Proof of Lemma~\ref{lem:posterior-lsc}]
Define
\[
\rvec h_n
:=
\rvec u_n
-
\mathbb E\rvec u_n,
\qquad
\rvec h
:=
\rvec u
-
\mathbb E\rvec u.
\]
Since $\rvec u_n
\to
\rvec u$ in $L^2$
we have
\[
\begin{aligned}
\left\|
\rvec h_n-\rvec h
\right\|_{L^2}
\le
\left\|
\rvec u_n-\rvec u
\right\|_{L^2}
+
\left\|
\mathbb E(
\rvec u_n-\rvec u)
\right\|_2
\le
2
\left\|
\rvec u_n-\rvec u
\right\|_{L^2}
\to0.
\end{aligned}
\]
Thus $\rvec h_n
\to
\rvec h$ in $L^2$.
For $\varphi
\in
C_b(\mathsf Z;\mathbb R^d)$,  
define
\[
\begin{aligned}
J_n(\varphi)
:=
&
2\mathbb E
\left\langle
\rvec h_n,
\varphi(\mathsf W_n)
\right\rangle
-
\mathbb E
\left\|
\varphi(\mathsf W_n)
\right\|_2^2,
\\
J(\varphi)
:=
&
2\mathbb E
\left\langle
\rvec h,
\varphi(\mathsf W)
\right\rangle
-
\mathbb E
\left\|
\varphi(\mathsf W)
\right\|_2^2.
\end{aligned}
\] 
Since $\mathsf W_n
\to
\mathsf W$ almost surely 
and \(\varphi\) is continuous,
\[
\varphi(\mathsf W_n)
\to
\varphi(\mathsf W)
\qquad
\text{almost surely}.
\]
Since \(\varphi\) is bounded, dominated convergence gives
\[
\varphi(\mathsf W_n)
\to
\varphi(\mathsf W)
\qquad
\text{in }L^2.
\]
It follows that
\[
\mathbb E
\left\|
\varphi(\mathsf W_n)
\right\|_2^2
\to
\mathbb E
\left\|
\varphi(\mathsf W)
\right\|_2^2.
\]
Moreover,
\[
\begin{aligned}
&
\left|
\mathbb E
\left\langle
\rvec h_n,
\varphi(\mathsf W_n)
\right\rangle
-
\mathbb E
\left\langle
\rvec h,
\varphi(\mathsf W)
\right\rangle
\right|
\\
&\le
\left\|
\rvec h_n-\rvec h
\right\|_{L^2}
\left\|
\varphi(\mathsf W_n)
\right\|_{L^2}+
\left\|
\rvec h
\right\|_{L^2}
\left\|
\varphi(\mathsf W_n)
-
\varphi(\mathsf W)
\right\|_{L^2}\longrightarrow
0.
\end{aligned}
\]
Therefore $J_n(\varphi)
\to
J(\varphi).$ 
By
Lemma~\ref{lem:conditional-mean-variational},
\[
\mathbb E
\left\|
\mathbb E[
\rvec h_n\mid\mathsf W_n]
\right\|_2^2
=
\sup_{\psi\in C_b(\mathsf Z;\mathbb R^d)}
J_n(\psi).
\]
Hence, for every fixed
\(\varphi\in C_b(\mathsf Z;\mathbb R^d)\),
\[
\mathbb E
\left\|
\mathbb E[
\rvec h_n\mid\mathsf W_n]
\right\|_2^2
\ge
J_n(\varphi).
\]
Taking the lower limit gives
\[
\liminf_{n\to\infty}
\mathbb E
\left\|
\mathbb E[
\rvec h_n\mid\mathsf W_n]
\right\|_2^2
\ge
J(\varphi).
\]
Taking the supremum over
\(\varphi\in C_b(\mathsf Z;\mathbb R^d)\)
and applying the variational identity again yields
\[
\liminf_{n\to\infty}
\mathbb E
\left\|
\mathbb E[
\rvec h_n\mid\mathsf W_n]
\right\|_2^2
\ge
\mathbb E
\left\|
\mathbb E[
\rvec h\mid\mathsf W]
\right\|_2^2.
\]
Finally,
\[
\mathbb E[
\rvec h_n\mid\mathsf W_n]
=
\mathbb E[
\rvec u_n\mid\mathsf W_n]
-
\mathbb E\rvec u_n,
\]
and
\[
\mathbb E[
\rvec h\mid\mathsf W]
=
\mathbb E[
\rvec u\mid\mathsf W]
-
\mathbb E\rvec u.
\]
This proves the lemma.
\end{proof}

\section{External Analytic and Probabilistic Tools}

We use the following standard external results in the forms stated below.

\begin{enumerate}[label=(\roman*)]

\item
\textbf{Prékopa's theorem.~\citep{prekopa1973logarithmic}}
Let $F:
\mathbb R^p\times\mathbb R^q
\to
[0,\infty)$ 
be jointly log-concave and assume that $0
<
\int_{\mathbb R^q}
F(\mathbf x,\mathbf y)
\,\mathrm d\mathbf y
<
\infty$
for every \(\mathbf x\in\mathbb R^p\).  Then the marginal function $\mathbf x
\to
\int_{\mathbb R^q}
F(\mathbf x,\mathbf y)
\,\mathrm d\mathbf y$ 
is log-concave.

\item
\textbf{Brascamp--Lieb variance inequality.~
\citep{brascamp1976extensions}}
Let \(\mu\) have density $\mu(\mathrm d\mathbf x)
=
\mathcal Z^{-1}
e^{-V(\mathbf x)}
\,\mathrm d\mathbf x$ 
on \(\mathbb R^d\), where
\(V\in C^2(\mathbb R^d)\) and $\nabla^2V(\mathbf x)
\succeq
\alpha\mathbf I_d$ 
for some \(\alpha>0\).  Then, for every
\(\mathbf w\in\mathbb R^d\),
\[
\operatorname{Var}_\mu
\left(
\langle\rvec x,\mathbf w\rangle
\right)
\le
\mathbb E_\mu
\left[
\left\langle
\mathbf w,
\bigl(\nabla^2V(\rvec x)\bigr)^{-1}
\mathbf w
\right\rangle
\right]
\le
\frac{\|\mathbf w\|_2^2}{\alpha},
\qquad
\rvec x\sim\mu.
\]

\item
\textbf{Talagrand \(T_2\) inequality for strongly log-concave
measures.~\citep{otto2000generalization}}
Let \(\mu\) have density $\mu(\mathrm d\mathbf x)
=
\mathcal Z^{-1}
e^{-V(\mathbf x)}
\,\mathrm d\mathbf x$
on \(\mathbb R^d\), where $\nabla^2V(\mathbf x)
\succeq
\alpha\mathbf I_d$ 
for some \(\alpha>0\).  Then, for every probability measure
\(\nu\ll\mu\), $W_2^2(\nu,\mu)
\le
\frac{2}{\alpha}
\KL(\nu\|\mu)$.  
In particular, for a \(1\)-strongly log-concave measure, $W_2^2(\nu,\mu)
\le
2\KL(\nu\|\mu)$.  

\item
\textbf{Gaussian smallest singular value.~
\citep[ Exercise~7.13 (b)]{Vershynin_2026}}
Let $\rmat G
\in
\mathbb R^{m\times n}$, $1\le m\le n$,  
have independent \(N(0,1)\) entries.  Let $s_{\min}(\rmat G)
:=
s_m(\rmat G)$ 
denote its smallest nonzero singular value.  There exists a universal
constant \(c_{\mathrm{sg}}>0\) such that, for every \(t\ge0\),
\[
\mathbb P
\left(
s_{\min}(\rmat G)
\le
\sqrt n-\sqrt m-t
\right)
\le
2e^{-c_{\mathrm{sg}}t^2}.
\]
This is the wide-matrix form of the cited result, obtained by applying
the corresponding tall-matrix statement to
\(\rmat G^\top\).

\item
\textbf{Gaussian norm concentration.~
\citep[Theorem~3.1.1]{Vershynin_2026}}
Let $\rvec g
\sim
N(\mathbf 0_d,\mathbf I_d)$.  
There exists a universal constant
\(c_{\mathrm{norm}}>0\)
such that
\[
\mathbb P
\left(
\|\rvec g\|_2
\ge
2\sqrt d
\right)
\le
e^{-c_{\mathrm{norm}}d}.
\]

\end{enumerate}

\section{Smallest Singular Value}
\label{app:sv}

The following theorem is the spectral input used in
Lemma~\ref{lem:negative-optimum-event}.  Recall that
\(\gamma_d^{\mathrm{tr}}\)
denotes the standard Gaussian law conditioned on the Euclidean ball of
radius \(2\sqrt d\).

\begin{theorem}[Spectral event for the truncated Gaussian design]
\label{thm:sv-truncation}
There exists a universal integer
\(d_0\ge1\)
such that the following holds for every $d\ge d_0$ 
and every integer \(k\) satisfying $1
\le
k
\le
\frac d{16}$.  

Let $\rvec b_1,\ldots,\rvec b_k
\overset{\mathrm{i.i.d.}}{\sim}
\gamma_d^{\mathrm{tr}}$, $\rvec a_i
:=
\frac{\rvec b_i}{\sqrt d}$,  
and define $\rmat A
:=
\begin{pmatrix}
\rvec a_1,
\ldots,
\rvec a_k
\end{pmatrix}^\top
\in
\mathbb R^{k\times d}$.  
Then
\begin{equation}
\mathbb P_\Xi
\left(
\rmat A\rmat A^\top
\succeq
\frac14\mathbf I_k
\right)
\ge
\frac78.
\label{eq:sv-truncation-main}
\end{equation}
Consequently, in the notation of
Lemma~\ref{lem:negative-optimum-event},
one may take $c_{\mathrm{sv}}
=
\frac14.$ 
\end{theorem}

\begin{proof}[Proof of Theorem~\ref{thm:sv-truncation}]
Let $\rvec g_1,\ldots,\rvec g_k
\overset{\mathrm{i.i.d.}}{\sim}
N(\mathbf 0_d,\mathbf I_d)$ 
be unconditioned standard Gaussian random vectors, and define
\[
\rmat G
:=
\begin{pmatrix}
\rvec g_1,
\ldots,
\rvec g_k
\end{pmatrix}^\top
\in
\mathbb R^{k\times d},
\qquad
\rmat A^{\mathrm G}
:=
\frac1{\sqrt d}
\rmat G.
\]
Since \(k\le d\), the smallest nonzero singular value of
\(\rmat G\) is $s_{\min}(\rmat G)
=
s_k(\rmat G)$.  
By the Gaussian smallest-singular-value estimate stated above, for every
\(t\ge0\),
\begin{equation}
\mathbb P
\left(
s_k(\rmat G)
\le
\sqrt d-\sqrt k-t
\right)
\le
2e^{-c_{\mathrm{sg}}t^2}.
\label{eq:sv-standard}
\end{equation}
Set $t
:=
\frac{\sqrt d}{4}$.  
Since $k
\le
\frac d{16}$,  
we have $\sqrt k
\le
\frac{\sqrt d}{4}$,  
and therefore
\[
\sqrt d-\sqrt k-t
\ge
\sqrt d
-
\frac{\sqrt d}{4}
-
\frac{\sqrt d}{4}
=
\frac{\sqrt d}{2}.
\]
It follows from
\eqref{eq:sv-standard}
that
\begin{equation}
\mathbb P
\left(
s_k(\rmat G)
<
\frac{\sqrt d}{2}
\right)
\le
2
\exp\left(
-\frac{c_{\mathrm{sg}}d}{16}
\right).
\label{eq:unconditioned-spectral-failure}
\end{equation}

On the complementary event,
\[
\lambda_{\min}
\left(
\rmat G\rmat G^\top
\right)
=
s_k(\rmat G)^2
\ge
\frac d4.
\]
Hence
\[
\rmat A^{\mathrm G}
\bigl(
\rmat A^{\mathrm G}
\bigr)^\top
=
\frac1d
\rmat G\rmat G^\top
\succeq
\frac14\mathbf I_k.
\]

Define the row-truncation event
\[
\mathcal R
:=
\bigcap_{i=1}^k
\left\{
\|\rvec g_i\|_2
\le
2\sqrt d
\right\}.
\]
By Gaussian norm concentration and the union bound,
\begin{align}
\mathbb P(\mathcal R^c)
\le
\sum_{i=1}^k
\mathbb P
\left(
\|\rvec g_i\|_2
>
2\sqrt d
\right)
\nonumber
\le
k
e^{-c_{\mathrm{norm}}d}
\nonumber
\le
\frac d{16}
e^{-c_{\mathrm{norm}}d}.
\label{eq:row-truncation-failure}
\end{align}

Choose \(d_0\) sufficiently large that, for every \(d\ge d_0\),
\[
2
\exp\left(
-\frac{c_{\mathrm{sg}}d}{16}
\right)
\le
\frac1{16},
\]
and
\[
\frac d{16}
e^{-c_{\mathrm{norm}}d}
\le
\frac1{16}.
\]
Then
\begin{equation}
\mathbb P(\mathcal R)
\ge
\frac{15}{16},
\label{eq:row-truncation-probability}
\end{equation}
and
\begin{equation}
\mathbb P
\left(
\rmat A^{\mathrm G}
\bigl(
\rmat A^{\mathrm G}
\bigr)^\top
\not\succeq
\frac14\mathbf I_k
\right)
\le
\frac1{16}.
\label{eq:unconditioned-spectral-probability}
\end{equation}

Because the vectors
\(\rvec g_1,\ldots,\rvec g_k\)
are independent and
\(\mathcal R\)
is the intersection of the corresponding rowwise truncation events, the
conditional law $\Law
\left(
\rvec g_1,\ldots,\rvec g_k
\mid
\mathcal R
\right)$ 
is the product measure $\left(
\gamma_d^{\mathrm{tr}}
\right)^{\otimes k}$. 
Consequently,
\[
\Law
\left(
\rmat A^{\mathrm G}
\mid
\mathcal R
\right)
=
\Law(\rmat A).
\]

Therefore,
\begin{align}
&
\mathbb P_\Xi
\left(
\rmat A\rmat A^\top
\not\succeq
\frac14\mathbf I_k
\right)
\nonumber\\
&\quad=
\mathbb P
\left(
\rmat A^{\mathrm G}
\bigl(
\rmat A^{\mathrm G}
\bigr)^\top
\not\succeq
\frac14\mathbf I_k
\;\middle|\;
\mathcal R
\right)
\nonumber\\
&\quad\le
\frac{
\mathbb P
\left(
\rmat A^{\mathrm G}
\bigl(
\rmat A^{\mathrm G}
\bigr)^\top
\not\succeq
\frac14\mathbf I_k
\right)
}{
\mathbb P(\mathcal R)
}
\nonumber\\
&\quad\le
\frac{1/16}{15/16}
=
\frac1{15}.
\end{align}
Thus
\[
\mathbb P_\Xi
\left(
\rmat A\rmat A^\top
\succeq
\frac14\mathbf I_k
\right)
\ge
\frac{14}{15}
\ge
\frac78.
\]
This proves
\eqref{eq:sv-truncation-main}.
\end{proof}
\vskip 0.2in
\bibliography{main}

\end{document}